\documentclass[11pt]{amsart}
\usepackage[foot]{amsaddr}
\usepackage{here}
\usepackage{multicol}
\usepackage{caption}
\usepackage{amsmath,amsthm}
\usepackage{amssymb,latexsym}
\usepackage{amscd}
\usepackage{enumerate}
\usepackage{latexsym}
\usepackage{ascmac}
\usepackage{graphics}
\usepackage[dvipdfmx]{graphicx}
\usepackage{oldgerm}
\usepackage{bm}
\usepackage{seqsplit} %
\usepackage{tikz}
\usetikzlibrary{arrows.meta}
\usetikzlibrary{plotmarks}
\usepackage{mathtools}
\mathtoolsset{showonlyrefs=true}
\usepackage{hyperref}
\usepackage{setspace}
\usepackage{mathrsfs}%
\usepackage{arrayjobx}%
\usepackage{color}%
\usepackage{thmtools}%
\usepackage{ulem}%
\usepackage{enumitem}%
\usepackage{young}%
\makeatletter
\renewcommand{\footnoterule}{\kern -3pt \hrule width \textwidth \kern 2.6pt}

\newcommand{\mysos}{S\'{o}s}
\newcommand{\mysur}{Sur\'{a}nyi}

\newcommand{\myChi}[1]{\left[\hspace{-0.2em}\left[ {#1} \right]\hspace{-0.2em}\right]{}}

\newtheorem{prop}{Proposition}
\newtheorem{lemma}{Lemma}
\newtheorem{theo}{Theorem}
\newtheorem*{thm}{Theorem}
\newtheorem{coro}{Corollary}
\newtheorem{defi}{Definition}
\newtheorem{fact}{Fact}

\newtheorem{rema}{Remark}
\newtheorem{exam}{Example}

\definecolor{darkgreen}{rgb}{0,0.5,0}

\newcommand{\myv}[2]{{\textstyle{\binom{#1}{#2}}}}
\newcommand{\myh}[2]{{({#1}, {#2})}}
\newcommand{\nyh}[2]{{{#1}, {#2}}}
\newcommand{\upmapzero}[2] {u_{\nyh{#1}{#2}}}
\newcommand{\upmapzerodon}[1] {\breve{u}_{#1}}
\newcommand{\phidon}[1] {\breve{\phi}_{#1}}

\newcommand{\RT}[2]{{\mathrm{RT}_\nyh{#1}{#2}}}
\newcommand{\Yng}[2]{{\mathrm{Y}_\nyh{#1}{#2}}}
\newcommand{\LYng}[2]{{\mathrm{LY}_\nyh{#1}{#2}}}
\newcommand{\TYng}[2]{{\mathrm{TY}_\nyh{#1}{#2}}}
\newcommand{\TYngdon}[1]{{\breve{\mathrm{TY}}_{#1}}}
\newcommand{\Lsetitimono}[2] {\mathcal{E}_{\nyh{#1}{#2}}^{1,\text{in}}}
\newcommand{\Lsetitimonozero}[2] {{\mathcal{T}}_{\nyh{#1}{#2}}^{1,\text{in}}}

\newcommand{\Lsetitimonozerodon}[1] {\breve{\mathcal{T}}_{#1}}
\newcommand{\Garey}[2]{\mathrm{G}_{\nyh{#1}{#2}}}
\newcommand{\Gseq}[2]{\mathrm{G}_{\nyh{#1}{#2}}^{\mathrm{seq}}}
\newcommand{\Gint}[2]{\mathrm{G}_{\nyh{#1}{#2}}^{\mathrm{int}}}
\newcommand{\Inter}[2] {\mathcal{G}_{\nyh{#1}{#2}}}
\newcommand{\Interdon}[1] {\breve{\mathcal{G}}_{#1}}
\newcommand{\trimap}[2] {{V}_{\nyh{#1}{#2}}}
\newcommand{\trimapdon}[1] {\breve{V}_{#1}}
\newcommand{\omicron}{o}
\newcommand{\tp}[1]{{}^t{#1}}

\newcommand{\yas}[2]{\mathsf{S}_{#2}^{#1}}
\newcommand{\yast}[2]{\mathsf{ST}_{#2}^{#1}}
\newcommand{\yastdon}[1]{\breve{\mathsf{ST}}_{#1}}
\newcommand{\delf}[3]{\Delta_\nyh{#1}{#2}\left({#3}\right)}
\newcommand{\Th}[1]{\ensuremath{{#1}^{\textrm{th}}}}
\newcommand{\order}{size \relax}
\newsavebox{\BEm}
\savebox{\BEm}{\small\texttt{donogo@st.rim.or.jp}}
\author{Makoto Nagata$^{1)}$}
\address{1) Faculty of Pharmacy, Osaka Medical and Pharmaceutical University}
\author{Yoshinori Takei$^{2)}$ }
\address{2) Independent researcher, \usebox{\BEm}.}
\title[A simple recursive characterization of a generalized Farey tree]{A simple recursive relation characterizes a tree associated to generalized Farey sequences}
\begin{document}
\begin{abstract}
  This paper proves that two differently defined rooted binary trees
  are isomorphic. The first tree is one associated to a version of
  Farey sequences where the vertices correspond to the
  open intervals formed by two successive terms in the sequence.
  The other tree
  has the vertices consisting of pairs of positive integers
  whose adjacency is defined by a simple recursive relation.
  These trees appeared in a study
  of a generalization of a
  class of the permutations defined by {\mysos} and
  the bijection between it and the set of the Farey intervals
  due to {\mysur}.%
\end{abstract}
\maketitle
\section{Introduction}\label{sect:intro}
In this paper, $[a, b]$ and $(a, b)$ denote
the closed real interval
 $\{x : a \leq x \leq b \}$ and
the open real interval $\{x : a < x < b\}$ respectively.
  In order to avoid confusion, a pair of two numbers $(i, j)$
  will be denoted as
  the vertical vector $\myv{i}{j}$ often.
  For a positive integer $m$, $[m]$ denotes the set
  $\{1, 2, \ldots, m\}$. The symbols $\wedge$ and $\vee$
  mean the logical conjunction and the logical disjunction, respectively.

  The aim of this paper is %
  to show that the two differently defined binary rooted
    trees
  in a study \cite{ranking3} of a generalization
  of the class of permutations defined by {\mysos} \cite{sos}
  are isomorphic.
   In both trees, each of the vertices
  has double indices $\myv{m}{n}$
  (where $m$ and $n$ are referred to as \textit{horizontal}
  and \textit{vertical} indices respectively)
  of nonnegative integers
  and its level (i.e. the distance from the root)
  is associated to the sum  $m + n$ of the indices.

  The first tree is one related to a generalized version of
  Farey fractions shown below:
\begin{defi}[{\cite{ranking3}}, see also \cite{Glaisher}]\label{defi:garey}
  For nonnegative integers $m$ and $n$,
  let
\begin{equation*}
\Garey{m}{n} := \begin{cases} \{0,\infty\} & \text{\ if\ }mn = 0, \\
\{0,\infty\} \cup
  \left\{\frac{p}{q} : p \in [m], q \in [n] \right\} & \text{\ otherwise},
\end{cases}
\end{equation*}
and let $\Gseq{m}{n}$ be the increasing sequence formed by the elements
of $\Garey{m}{n}$ without repetition. 
Also, let $\Gint{m}{n}$
be the collection $\{ (a, b) : a, b \in \Garey{m}{n}, a < b,
a \text{\ and\ } b \text{\ are\ adjacent in \ } \Gseq{m}{n} \}$
of the open intervals formed by the adjacent terms in $\Gseq{m}{n}$.
Furthermore, let
$$
\Inter{m}{n} :=  \{ ((a, b), \myv{m}{n}) : (a, b) \in \Gint{m}{n} \}
$$
be the collection of the open intervals in $\Gint{m}{n}$ tagged
with the indices $\myv{m}{n}$.
\end{defi}

In the first tree to which we will refer as
$\mathbb{T}_N^{\nyh{\Interdon{}}{\trimapdon{}}}$ (where $N$ is the height of the tree),
the level $k$ vertices are the elements in
$\Interdon{k} := \bigsqcup_{\begin{smallmatrix}m, n\geq 0 \\ m + n = k\end{smallmatrix}} \Inter{m}{n}.$
A level $k$ vertex $((a, b), \myv{m}{n})$ has at most
two children in $\Interdon{k+1}$, 
one is of the form $((a, b'), \myv{m+1}{n})$ and
the other is of the form $((a', b), \myv{m}{n+1})$.
These children are defined as the inverse image
of $((a, b), \myv{m}{n})$ by 
a surjection $\trimapdon{k+1}:\Interdon{k+1} \rightarrow \Interdon{k}$
whose definition 
is based on the position of the interval $(a,b)$ in $(0,\infty)$
and satisfy $(a,b') \subset (a,b)$ and
$(a',b) \subset (a,b)$. See Fact \ref{fact:invtrimap}
and Definition \ref{defi:trimaptree}
in
Section \ref{sss:Inter} for the detail.

The second tree, where a vertex is of the form
$(s, t, \myv{m}{n})$ such that $s, t \in [mn]$, %
is defined by the following simple recursive relation.
The root is 
$(1, 1, \myv{1}{1})$ and a
 vertex $(s, t, \myv{m}{n})$ has level $k$ when
 $m + n = k + 2$. %
And the rule
 \begin{equation}
  \text{a vertex\ }
  (s, t, \myv{m}{n})
  \text{\ has\ }
\begin{cases}
  \text{a \textit{horizontal} child\ }
(s + n, t, \myv{m+1}{n}) & \text{\ if\ } s - t > -n \\
\text{a \textit{vertical} child\ }
(s, t + m, \myv{m}{n+1}) & \text{\ if\ } s - t < m
\end{cases}
\end{equation}
 determines the set of the level $k+1$ vertices and the
 edges between the levels $k$ and $k+1$
 (the first and the second lines are not exclusive,
 therefore, a level $k$ vertex has 
 at most two children in the set
 of the level $k+1$ vertices).
For example, the root $(1, 1, \myv{1}{1})$ at level $0$
has two children $(2,1,\myv{2}{1})$ and $(1,2,\myv{1}{2})$
which are all the  level $1$ vertices. By applying the relation 
recursively, we obtain the level $2, 3, \ldots, N$ vertices in order.
The resulting tree will be
referred to as $\mathbb{T}_N^{\nyh{\Lsetitimonozerodon{}}{\upmapzerodon{}}},$
whose formal definition is given in Definition \ref{defi:upmaptree},
Section \ref{sss:upmaptree}.

The main
 result of this paper is as follows.
\begin{thm}
For any nonnegative integer $N$,
  the trees $\mathbb{T}_N^{\Interdon{},\trimapdon{}}$
  and $\mathbb{T}_N^{\Lsetitimonozerodon{},\upmapzerodon{}}$
  are isomorphic.
\end{thm}

The rest of the paper is organized as follows:
Section \ref{sect:defknown} introduces
the definitions and known facts 
and uses them to define
these %
trees. The third tree $\mathbb{T}_N^{\TYngdon{},\phidon{}}$
that is seen as
  an intermediate object between the two is also introduced.
In Section \ref{sect:trees}, 
it is proved
that the three trees are
isomorphic. 
In Appendix, as well as an example of the trees,
  the background and implications of this
study in the context of the study of
generalizations of the inverses of {\mysos} permutations
\cite{sos} and {\mysur}'s bijection \cite{sur} are described.

\section{Definitions and known facts on the three trees}
\label{sect:defknown}
In this section,
the three trees $\mathbb{T}_N^{\nyh{\Interdon{}}{\trimapdon{}}}$,
$\mathbb{T}_N^{\nyh{\Lsetitimonozerodon{}}{\upmapzerodon{}}}$
and
$\mathbb{T}_N^{\TYngdon{},\phidon{}}$ of height $N$
are defined. For each of them, the subsections
  that describe the vertices and the edges are put respectively.
As described
in the previous section, the set of vertices
in any of the three trees are indexed with $\myv{m}{n}$ and
the level of them is defined based on the sum $m + n$.

For each of the three trees, an inter-level surjection
from the set of the level $k+1$ vertices to
the set of the level $k$ vertices is defined and used to define
the adjacency.

In the sequel, the Iverson-Knuth brackets
  \begin{equation}
    \myChi{P}  = \begin{cases} 1  & \text{if } P \text{ is true}, \\
      0 &  \text{if } P \text{ is false}
      \end{cases}
  \end{equation}
  are used extensively to denote the %
  truth value of a predicate
    $P$.
Double parentheses in
  $\theta\left(\myv{i}{j}\right)$ will be simplified as
  $\theta\myv{i}{j}$ for readability.

\subsection{The tree $\mathbb{T}_N^{\Interdon{},\trimapdon{}}$ of generalized Farey intervals}

\subsubsection{Level $k$ vertices of the generalized Farey intervals}
As introduced in Section \ref{sect:intro},
the level $k$ vertices of the tree $\mathbb{T}_N^{\Interdon{},\trimapdon{}}$
are defined as the elements of the following set:
\begin{defi}[{\cite[Notation 4.6]{ranking3}}]
For $k \geq 0,$ let
\begin{equation*}
\Interdon{k} := \bigsqcup_{\begin{smallmatrix}m, n\geq 0 \\ m + n = k \end{smallmatrix}} \Inter{m}{n}.
\end{equation*}  
\end{defi}

Each of the vertices allows the following \textit{transposition},
which
works as an involution over $\Interdon{k}$ to which the vertex belongs.
Note that swapping $m$ and $n$ induces a natural order-reversing
  bijection
  $\Gseq{m}{n}\ni p/q \mapsto q/p \in \Gseq{n}{m}$
  which in turn induces a bijection
  $\Gint{m}{n} \ni (a,b) \mapsto (b^{-1}, a^{-1}) \in \Gint{n}{m}$%
, where $0^{-1} = \infty$ and $\infty^{-1} = 0$.
\begin{defi}\label{defi:tpfarey}
  For an element $((a,b), \myv{m}{n}) \in \Inter{m}{n},$
  its transposition
  $\tp{((a,b), \myv{m}{n})}$
  is defined to be \\
  $((b^{-1},a^{-1}), \myv{n}{m})$
  (with the convention $0^{-1} = \infty, \infty^{-1} = 0$),
  which is an element of $\Inter{n}{m}$.
 The operation $\tp{\cdot}$ is referred to as the transposition
  for generalized Farey intervals.
\end{defi}
\subsubsection{The edges of $\mathbb{T}_N^{\Interdon{},\trimapdon{}}$}
  \label{sss:Inter}
For the generalized Farey intervals, the following map
from the level $k$ vertices
to the level $k-1$ vertices
is defined and used to define the adjacency of the tree 
$\mathbb{T}_N^{\Interdon{},\trimapdon{}}$.
We note that the well-definedness of the map $\trimap{m}{n}$ given below
depends on the facts \cite[Fact 4.2]{ranking3} that
  $$
  \Garey{m}{n} = \left\{a \in \Garey{m-1}{n} : a < \frac{m}{n} \right\} \sqcup
  \left\{ \frac{m}{n} \right\} \sqcup
  \left\{ b \in \Garey{m}{n-1} : b > \frac{m}{n} \right\}
  $$
  and \cite[Fact 4.7]{ranking3}
  that for
  an open interval $(a, b) \in \Gint{m}{n}$, exactly
  one of the following mutually exclusive four cases occurs:
  $$
  b < \frac{m}{n}, \ b = \frac{m}{n}, \ a = \frac{m}{n}, \ a > \frac{m}{n}.
  $$
  \begin{defi}[{\cite[Definition 4.8]{ranking3}}]\label{defi:trimap}
  Let $m, n$ be nonnegative integers with $m + n \geq 1$.
  The map $\trimap{m}{n} : \Inter{m}{n}\to \Inter{m-1}{n} \sqcup \Inter{m}{n-1}$
  is defined as follows:

  When $m = 0 \text{~and~} n \geq 1,$
  $$
  \trimap{0}{n}(((0,\infty),\myv{0}{n})) := ((0,\infty),\myv{0}{n-1})
  $$
  (Note that $((0,\infty),\myv{0}{n})$ is the only element of $\Inter{0}{n}$).

  When $m \geq 1 \text{~and~} n = 0,$
  $$
  \trimap{m}{0}(((0,\infty),\myv{m}{0})) := ((0,\infty),\myv{m-1}{0})
  $$
  (Note that $((0,\infty),\myv{m}{0})$ is the only element of $\Inter{m}{0}$).

  When $mn \geq 1,$ for $((a,b), \myv{m}{n}) \in \Inter{m}{n},$
  \begin{equation}
\trimap{m}{n}( ((a,b), \myv{m}{n}) ) 
:=
\begin{cases}
  ((a,b),\myv{m-1}{n}) & (b < \frac{m}{n}), \\
  ((a,\min\{c \in \Garey{m-1}{n} : a < c\}), \myv{m-1}{n} ) & (b = \frac{m}{n}), \\
  ((\max\{c \in \Garey{m}{n-1} : c < b\}, b), \myv{m}{n-1}) & (a = \frac{m}{n}), \\
  ((a,b),\myv{m}{n-1} ) & (a > \frac{m}{n}),
\end{cases}
\label{eq:trimapdef}
\end{equation}
  in which $\min\{c \in \Garey{m-1}{n} : a < c\} \geq b$
  (resp. $\max\{c \in \Garey{m}{n-1} : c < b\} \leq a$)
  holds by $b = \min\{c \in \Garey{m}{n} : a < c\}$
  and $\Garey{m-1}{n}\subset\Garey{m}{n}$ (resp.
  $a = \max\{c \in \Garey{m}{n} : c < b\} \leq a$
   and $\Garey{m}{n-1}\subset\Garey{m}{n}$).
\end{defi}
The map $\trimap{m}{n}$ is extended to the following inter-level map
of the generalized Farey intervals in an obvious way.
\begin{defi}[{\cite[Definition 4.9]{ranking3}}]\label{defi:trimapdon}
  Let  $k$ be a positive integer. Define the map
  $\trimapdon{k}:\Interdon{k}\to\Interdon{k-1}$ as follows:
  Given $x \in \Interdon{k},$ there exists the unique $\myv{m}{n}$ such
  that $m + n = k$ and $x \in \Inter{m}{n}$.
  Let $\trimapdon{k}(x) := \trimap{m}{n}(x) \in \Inter{m-1}{n} \sqcup \Inter{m}{n-1} \subset \Interdon{k-1}.$
\end{defi}

Also, note that the definition of the inter-level map
is consistent with the transposition
for generalized Farey intervals:
\begin{prop}\label{prop:tptrimap}
  Let $m, n$ be nonnegative integers with $m + n \geq 1$. The
  transposition for generalized Farey intervals
  and $\trimap{*}{*}$ commute,
    i.e., it holds that
    $$
    \trimap{n}{m}(\tp{((a,b),\myv{m}{n})}) =
    \trimap{n}{m}(((b^{-1}, a^{-1}),\myv{n}{m})) = \tp{(\trimap{m}{n}(((a,b),\myv{m}{n})))}\quad (((a,b),\myv{m}{n}) \in \Inter{m}{n}),
    $$
    which immediately implies that
      $\trimapdon{k}(\tp{x}) = \tp{(\trimapdon{k}(x))}$ for $k \geq 1$ and
      $x \in \Interdon{k}$.
\end{prop}    
\begin{proof}
First, suppose that $((a,b),\myv{m}{n})$ falls into the fourth
 case of \eqref{eq:trimapdef}. Then, $a > m/n$ holds.
 Writing
$((a',b'), \myv{m'}{n'}) := \tp{((a,b),\myv{m}{n})},$
this implies that $b' = a^{-1} < n/m = m'/n'$ and
that $((a',b'), \myv{m'}{n'})$ falls into the first case of 
\eqref{eq:trimapdef}.
Thus, $\trimap{n}{m}(\tp{((a,b),\myv{m}{n})})
   = \trimap{m'}{n'}(((a',b'), \myv{m'}{n'}))
= ((a',b'),\myv{m'-1}{n'})$,
where $((a',b'),\myv{m'-1}{n'}) = ((b^{-1}, a^{-1}), \myv{n-1}{m})$
is the transposition of $((a,b),\myv{m}{n-1})$ which coincides
with $\trimap{m}{n}(((a,b),\myv{m}{n}))$ computed as the fourth
case $a > m/n$ of \eqref{eq:trimapdef}.
Thus, it follows that $\trimap{n}{m}(\tp{((a,b),\myv{m}{n})})
   = \tp{(\trimap{m}{n}(((a,b),\myv{m}{n})))}. $
   The commutativity
   for $((a,b),\myv{m}{n})$ falling
  into the first case of \eqref{eq:trimapdef} is verified in a very
  similar way.

Next, suppose that $((a,b),\myv{m}{n})$ falls into the third
 case of \eqref{eq:trimapdef}.
Again by writing $((a',b'), \myv{m'}{n'}) := \tp{((a,b),\myv{m}{n})},$
$a = m/n$ implies that $b' = a^{-1} = n/m = m'/n'.$ Then
$((a',b'), \myv{m'}{n'}) \in \Inter{m'}{n'}$ falls into the second case
of \eqref{eq:trimapdef}, implying that
$\trimap{m'}{n'}(((a',b'), \myv{m'}{n'}))=
((a',\min\{c' \in \Garey{m'-1}{n'} : a' < c'\}), \myv{m'-1}{n'} )$.
By undoing the variable substitution,
RHS is $((b^{-1},\min\{c' \in \Garey{n-1}{m} : b^{-1} < c'\}), \myv{n-1}{m} )$,
where $\min\{c' \in \Garey{n-1}{m} : b^{-1} < c'\} =
(\max\{(c')^{-1}: c' \in \Garey{n-1}{m}, b^{-1} < c' \} )^{-1}
= (\max\{c: c \in \Garey{m}{n-1}, b > c \} )^{-1}.$
Thus, when $((a,b),\myv{m}{n})$ falls into the third
case of \eqref{eq:trimapdef}, it holds that
$\trimap{n}{m}( \tp{ ((a,b),\myv{m}{n}) }) =
((b^{-1}, (\max\{c: c \in \Garey{m}{n-1}, b > c \} )^{-1} ),
\myv{n-1}{m})$.
The last expression
coincides with $\tp{(\trimap{m}{n}(((a,b),\myv{m}{n})))}$
  where $\trimap{m}{n}((a,b),\myv{m}{n})$ is
  computed using the
  third case of \eqref{eq:trimapdef}.
Thus, it follows that $\trimap{n}{m}(\tp{((a,b),\myv{m}{n})})
   = \tp{(\trimap{m}{n}(((a,b),\myv{m}{n})))}. $
  A very similar argument
  is applied to the
  second case in
  \eqref{eq:trimapdef}.
\end{proof}

The surjectivity of {the inter-level map is crucial in this study.
\begin{fact}[{\cite[Proposition 4.10]{ranking3}}]\label{fact:invtrimap}
  Let $k \geq 1$ be an integer.
  The map $\trimapdon{k}:\Interdon{k} \to \Interdon{k-1}$ is surjective
  and for nonnegative $m, n$ with $m + n = k-1$,
  the inverse image of a singleton $\{((a,b),\myv{m}{n})\} \subset
  \Interdon{k-1}$ by $\trimapdon{k}$ satisfies the following property.

  When $mn = 0,$ (where only $(a,b) = (0, \infty)$ is possible)
  \begin{equation}
    \trimapdon{k}^{-1}(\{((0,\infty),\myv{m}{n})\})
    =
    \begin{cases}
    \{((0,\infty),\myv{1}{0}), ((0,\infty),\myv{0}{1}) \}
    & (\myv{m}{n} = \myv{0}{0}), \\
    \{((0,\frac{1}{n}),\myv{1}{n}), ((0,\infty),\myv{0}{n+1}) \}      
    & (m = 0 \text{~and~} n \geq 1 ), \\
    \{((0,\infty),\myv{m+1}{0}), ((m,\infty),\myv{m}{1}) \}          
    & (m \geq 1 \text{~and~} n = 0 ). 
    \end{cases}
  \end{equation}    

  When $mn \geq 1,$ it holds that
  $$
  \trimapdon{k}^{-1}(\{((a,b),\myv{m}{n})\}) \subset
  \Inter{m+1}{n} \sqcup \Inter{m}{n+1}
  $$
  with
  \begin{align*}
    |\trimapdon{k}^{-1}(\{((a,b),\myv{m}{n})\}) \cap \Inter{m+1}{n}| & =
    \myChi{a < \frac{m+1}{n}}, \\
    |\trimapdon{k}^{-1}(\{((a,b),\myv{m}{n})\}) \cap \Inter{m}{n+1}| & =
    \myChi{b > \frac{m}{n+1}} 
  \end{align*}
  and that
  \begin{align}
    a < \frac{m+1}{n} & \Rightarrow
    \trimapdon{k}^{-1}(\{((a,b),\myv{m}{n})\}) \cap \Inter{m+1}{n}
    = 
    \begin{cases}
      \{((a,\frac{m+1}{n}),\myv{m+1}{n})\} & (b \geq \frac{m+1}{n}), \\
      \{((a,b),\myv{m+1}{n})\} & (b <    \frac{m+1}{n}), 
    \end{cases} \\
    b > \frac{m}{n+1} & \Rightarrow
        \trimapdon{k}^{-1}(\{((a,b),\myv{m}{n})\}) \cap \Inter{m}{n+1}
    = 
    \begin{cases}
      \{((\frac{m}{n+1},b),\myv{m}{n+1})\} & (a \leq \frac{m}{n+1}), \\
      \{((a,b),\myv{m}{n+1})\} & (a >    \frac{m}{n+1}).
    \end{cases}
  \end{align}
  In other words, (i)
  $\trimapdon{k}^{-1}(\{((a,b),\myv{m}{n})\}) \cap \Inter{m+1}{n}$
  is nonempty if and only if $a < \frac{m+1}{n}$, and when it is nonempty,
  the unique element of it is of the form $((a,*),\myv{m+1}{n})$
  where $*=b\text{\ or\ }\frac{m+1}{n}$. (ii)
  $\trimapdon{k}^{-1}(\{((a,b),\myv{m}{n})\}) \cap \Inter{m}{n+1}$
  is nonempty if and only if $b > \frac{m}{n+1}$, and when it is nonempty,
  the unique element of it is of the form $((*,b),\myv{m}{n+1})$
  where $*=a \text{\ or\ }\frac{m}{n+1}$.
\end{fact}  
Note that $a \leq m/(n+1) < (m+1)/n \leq b$ is impossible because
$a$ and $b$ are adjacent in $\Gseq{m}{n}$ and $m/n$ is in $\Gseq{m}{n}.$
Note also that the statements (i) and (ii) are mutually reducible by considering
the transposition $\tp{((a,b),\myv{m}{n})}=((b^{-1},a^{-1}),\myv{n}{m})$.

The leveled sequence $\Interdon{0}, \Interdon{1}, \Interdon{2}, \ldots$
of the generalized Farey intervals
has the inter-level surjection $\trimapdon{k+1}:\Interdon{k+1}\to\Interdon{k}$
for each level $k=0, 1, \ldots$ as described above
and the smallest level $\Interdon{0}$
is the singleton $\{((0,\infty),\myv{0}{0})\}.$

Thus, the sequence allows a rooted tree structure
where the vertices of level $k$ are the \Th{k} term in the sequence
and the edges are defined by the inter-level maps.
\begin{defi}[{\cite[Sect. 7]{ranking3}}]\label{defi:trimaptree}
  Let $N$ be a nonnegative integer. For the sequence $\Interdon{0}, \Interdon{1}, \Interdon{2}, \ldots, \Interdon{N}$ and their inter-level surjections
  $\trimapdon{k+1}:\Interdon{k+1}\to\Interdon{k} ~(0\leq k < N), $ 
  let  $\mathbb{T}_N^{\Interdon{},\trimapdon{}}$ be the rooted tree 
  such that
  (i) the root is the unique element
$((0,\infty),\myv{0}{0})$
  of $\Interdon{0}$
  (ii)
for $0 \leq k \leq N$,
  the set of the level $k$ vertices  is $\Interdon{k}$
  (iii) for $0 \leq k < k' \leq N,$ vertices $v \in \Interdon{k}$
  and $v' \in \Interdon{k'}$ are adjacent if and only if
  $k' = k + 1$ and $\trimapdon{k'}(v') = v.$ In such an adjacent pair $(v,v')$,
  $v$ is said to be the parent of $v'$ and $v'$ is said to be
  a child of $v$.
\end{defi}

For a given level $k-1$ vertex $v$ of 
$\mathbb{T}_N^{\Interdon{},\trimapdon{}}$ that is not a leaf,
there is the unique $\myv{m}{n}$ such that $m + n = k-1$ and
$v \in \Inter{m}{n}$, then the children of $v$ are
$\trimapdon{k}^{-1}(\{v\})$ as given in Fact \ref{fact:invtrimap}.
Observe that $1 \leq |\trimapdon{k}^{-1}(\{v\})| \leq 2$,
that $|\trimapdon{k}^{-1}(\{v\}) \cap \Inter{m+1}{n}| \leq 1,$
and that $|\trimapdon{k}^{-1}(\{v\}) \cap \Inter{m}{n+1}| \leq 1.$
In other words, $v$ has at least one child, has at most one child
whose horizontal index is larger by one than that of $v$, and
has at most one child whose vertical index is larger by one than
that of $v$. In particular, $\mathbb{T}_N^{\Interdon{},\trimapdon{}}$
is a binary tree that all of the leaves has level $N$.

\subsection{The tree $\mathbb{T}_N^{\Lsetitimonozerodon{},\upmapzerodon{}}$
of the terminal pairs of difference equation type}
\subsubsection{Injective L-shapes of difference equation type}
To describe the vertices of $\mathbb{T}_N^{\Lsetitimonozerodon{},\upmapzerodon{}}$,
we review the definition of an
\textit{injective L-shape of difference equation type} 
\cite{ranking3}.
\begin{defi}[{\cite[Definition 5.5]{ranking3}}]\label{defi:Ldiffeq}
  Let $m$ and $n$ be positive integers. For a map
  $f:([m]\times\{1\})\cup (\{1\}\times[n]) \to[mn]$, let the condition
  $\mathrm{E}_\nyh{m}{n}$ be
\begin{multline*}
  \mathrm{E}_\nyh{m}{n} : 
   \left( f\myv{i+1}{1} - f\myv{i}{1} = \sum_{t \in [n]} \myChi{f\myv{1}{t} \leq f\myv{i+1}{1}} \quad (i \in [m-1])\right) \\
  \wedge
  \left(f\myv{1}{j+1} - f\myv{1}{j} = \sum_{s \in [m]} \myChi{f\myv{s}{1} \leq f\myv{1}{j+1}} \quad (j \in [n-1])\right).
\end{multline*}
\end{defi}  
\begin{exam}\label{exam:Lshape}
  The map $f:([4]\times\{1\}) \cup (\{1\}\times [3]) \to [12]$ defined as
\begin{equation}
  f := \left|
  \begin{matrix}6& & & \\3& & & \\1&2&4&7\\\end{matrix}
\right|,
    \label{eq:Lex43}
\end{equation}
where $\myv{i}{j} = \myv{1}{1}, \myv{4}{1}, \myv{1}{3}$
  correspond to the lower-left, lower-right, upper-left corners respectively,
satisfies
$\mathrm{E}_{\nyh{4}{3}}.$ The first half of the condition
for $\myv{i}{j} = \myv{3}{1}$ is verified as
\begin{align*}
f\myv{3+1}{1} - f\myv{3}{1} & = 7 - 4 = 3, \\
\sum_{t \in [3]} \myChi{f\myv{1}{t} \leq f\myv{3+1}{1}}
 & = \myChi{1 \leq 7} + \myChi{ 3 \leq 7} + \myChi{ 6 \leq 7} = 3.
\end{align*}
The first half of the condition for $\myv{2}{1}, \myv{1}{1}$
can be verified similarly.
Also, the latter half of the condition for $\myv{i}{j}=\myv{1}{2}$ is
satisfied as
\begin{align*}
  f\myv{1}{2+1} - f\myv{1}{2} & = 6 - 3 = 3. \\
  \sum_{s \in [4]} \myChi{f\myv{s}{1} \leq f\myv{1}{2+1}} &=
    \myChi{1 \leq 6} + \myChi{2 \leq 6} + \myChi{4 \leq 6} + \myChi{7 \leq 6}
    = 3.
\end{align*}
The latter half of the condition $\myv{i}{j}=\myv{1}{1}$ also can be
verified.
\end{exam}

The property $E_\nyh{m}{n}$
is inherited from \order $\myv{m}{n}$ to \order
$\myv{m-1}{n}$ or $\myv{m}{n-1}$ as follows.
\begin{fact}[{\cite[Proposition 5.8]{ranking3}}]\label{fact:Litimonorecur}
  Let $m$ and $n$ be positive integers.
  Suppose that a map
  $f: ([m]\times\{1\}) \cup (\{1\}\times[n]) \to[mn]$ satisfies
  $\mathrm{E}_\nyh{m}{n}.$ Then,
  \begin{itemize}
  \item[(i)] $f\myv{m}{1} < f\myv{1}{n}$ implies
    $(n \geq 2) \wedge (f\myv{1}{n} = f\myv{1}{n-1} + m) \wedge
    (f|_{([m]\times\{1\})\cup(\{1\}\times[n-1])} \text{\ satisfies \ }E_\nyh{m}{n-1})$.
    \medskip
  \item[(ii)] $f\myv{m}{1} > f\myv{1}{n}$ implies
    $(m \geq 2) \wedge (f\myv{m}{1} = f\myv{m-1}{1} + n) \wedge
    (f|_{([m-1]\times\{1\})\cup(\{1\}\times[n])} \text{\ satisfies \ }E_\nyh{m-1}{n})$.
  \end{itemize}    
\end{fact}
\begin{exam}
The $f$ appeared as \eqref{eq:Lex43} in Example \ref{exam:Lshape}
  that satisfies
$\mathrm{E}_\nyh{4}{3}$ falls into the case (ii) by
$7 = f\myv{4}{1} > 6 = f\myv{1}{3}.$  Indeed,
$f\myv{4}{1} = 7$ equals $f\myv{4-1}{1} + 3 = 4 + 3$. The restriction
to $([4-1]\times\{1\})\cup(\{1\}\times[3])$ is
\begin{equation}
  f' = \left|%
    \begin{matrix}6& &  \\3& &  \\1&2&4\\
    \end{matrix}\right|.
\label{eq:Lex33}
\end{equation}
Noting that $f'$ is just a restriction $f|_{([3]\times\{1\})\cup(\{1\}\times[3])}$
of $f$,
the first half of the condition $\mathrm{E}_\nyh{3}{3}$
for $f'$ is the first half of the condition
of $\mathrm{E}_\nyh{4}{3}$ restricted to $i = 1, 2.$
The second half of $\mathrm{E}_\nyh{3}{3}$ is verified as follows.
For $j = 2,$
\begin{align*}
  f\myv{1}{2+1} - f\myv{1}{2} & = 6 - 3 = 3. \\
  \sum_{s \in [3]} \myChi{f\myv{s}{1} \leq f\myv{1}{2+1}} &=
    \myChi{1 \leq 6} + \myChi{2 \leq 6} + \myChi{4 \leq 6} 
    = 3.
\end{align*}
For $j = 1,$
\begin{align*}
  f\myv{1}{1+1} - f\myv{1}{1} & = 3 - 1 = 2. \\
  \sum_{s \in [3]} \myChi{f\myv{s}{1} \leq f\myv{1}{j+1}} &= 
    \myChi{1 \leq 3} + \myChi{2 \leq 3} + \myChi{4 \leq 3} 
     = 2.
\end{align*}
Thus, $f'$ satisfies $E_\nyh{3}{3}$.
\end{exam}
\begin{defi}[{\cite[Definition 5.11]{ranking3}}]\label{defi:Lde}
  Let $m$ and $n$ be positive integers. An element of the set
  $$
  \Lsetitimono{m}{n} := \{ f: \ f:([m]\times\{1\}) \cup (\{1\}\times[n]) \to[mn]
  \text{\ is an injection satisfying\ } \mathrm{E}_\nyh{m}{n}
  \text{\ and\ } f\myv{1}{1} = 1 \}
  $$
  is said to be an \textit{injective L-shape of difference equation type}
  of \order $\myv{m}{n}$. For nonnegative integers $\myv{m}{n}$ with $mn=0$,
  put formally $\Lsetitimono{m}{n} = \emptyset$.
\end{defi}
As was the case for the generalized Farey intervals $\Inter{m}{n}$,
injective L-shapes allow a naturally defined
transposition induced by the swap of the
horizontal and vertical indices.
\begin{defi}\label{defi:Ltp}
  Let $m$ and $n$ be positive integers and let $f$ be
  an injective L-shape of \order $\myv{m}{n}$.
  The transposition $\tp{f}$ for L-shape  $f$ 
  is the injective L-shape
  $\tp{f}:([n] \times \{1\}) \cup (\{1\} \times [m]) \ni \myv{j}{i} \mapsto f\myv{i}{j} \in [mn] $ of \order $\myv{n}{m}$.
  It is easily checked that
  $f \in \Lsetitimono{m}{n}$ if and only if $\tp{f} \in \Lsetitimono{n}{m}$.
\end{defi}
\begin{exam}
  For the injective L-shape $f \in \Lsetitimono{4}{3}$
    defined in \eqref{eq:Lex43} of Example \ref{exam:Lshape}, its transposition
    is 
\begin{equation}
  \tp{f} = \left|
    \begin{matrix}7 & &  \\4& &  \\2& & \\ 1 & 3 & 6\end{matrix}
\right|,
\label{eq:Lex43t}
\end{equation}
an injective L-shape of \order $\myv{3}{4}$.
\end{exam}

\subsubsection{Level $k$ vertices of the terminal pairs of difference equation type }\label{sss:diffeqtype}
Fact \ref{fact:Litimonorecur}
in particular asserts that,
given the terminal pair $f\myv{m}{1}$ and $f\myv{1}{n}$
of $f \in \Lsetitimono{m}{n}$,
we obtain either the the value $f\myv{m-1}{n}$
of $f|_{([m-1]\times\{1\}) \cup (\{1\}\times[n])} \in \Lsetitimono{m-1}{n}$ 
or
the the value $f\myv{m}{n-1}$
of $f|_{([m]\times\{1\}) \cup (\{1\}\times[n-1])} \in \Lsetitimono{m}{n-1}$.
Repeating this,
one can easily determine all the terms of $f$.
In other words, by Fact \ref{fact:Litimonorecur},
an injective L-shape $f$ of difference equation type
of \order $\myv{m}{n}$
can be
``compressed'' to its terminal pair
$(\nyh{f\myv{m}{1}}{f\myv{1}{n}},\myv{m}{n})$,
tagged with the \order $\myv{m}{n}$ of the original L-shape,
without loss of the information.} Thus, the following set of
the terminal pairs
is essentially identical to
the set $\Lsetitimono{m}{n}.$
\begin{defi}[{\cite[Notation 6.2]{ranking3}}]\label{defi:termdiff}
  Let $m$ and $n$ be positive integers. An element of the set
  $$
  \Lsetitimonozero{m}{n} := \{ (f\myv{m}{1}, f\myv{1}{n}, \myv{m}{n} )
   : f \in \Lsetitimono{m}{n} 
  \}
  $$
  is said to be a \textit{terminal pair of difference equation type} of \order
  $\myv{m}{n}$.
\end{defi}
The level $k$ vertices of the tree
$\mathbb{T}_N^{\Lsetitimonozerodon{},\upmapzerodon{}}$
are defined to be the union of $\Lsetitimonozero{m}{n}$
with $m + n = k + 2$.
\begin{defi}[{\cite[Notation 6.3]{ranking3}}]
For $k \geq 0$, let
\begin{align*}
\Lsetitimonozerodon{k + 2} & := \bigsqcup_{\begin{smallmatrix}m, n\geq 1 \\ m + n = k+2\end{smallmatrix}} \Lsetitimonozero{m}{n}. 
\end{align*}
\end{defi}
Also over $\Lsetitimonozero{m}{n}$,
the \textit{transposition for terminal pairs},
that is consistent
with the transposition
$\tp{(\cdot)} : \Lsetitimono{m}{n} \rightarrow \Lsetitimono{n}{m}$
for
injective L-shapes of difference equation type (Definition \ref{defi:Ltp}),
is defined as follows:
\begin{defi}\label{defi:tpterm}
  For a terminal pair $(f\myv{m}{1}, f\myv{1}{n}, \myv{m}{n})
  \in \Lsetitimonozero{m}{n},$ its transposition
  $\tp{(f\myv{m}{1}, f\myv{1}{n}, \myv{m}{n})}$
  is defined as
  $(f\myv{1}{n}, f\myv{m}{1}, \myv{n}{m}) \in \Lsetitimonozero{n}{m}.$ 
  This induces $\tp{(\cdot)}: \Lsetitimonozerodon{k} \to \Lsetitimonozerodon{k}$
  in an obvious way.
\end{defi}

 For example, the corresponding terminal pair (indexed with $\myv{m}{n}$)
 of $f$ of \eqref{eq:Lex43} is $(7,6,\myv{4}{3}) \in \Lsetitimonozero{4}{3}$ and its transposition
 is $(6,7,\myv{3}{4}) \in \Lsetitimonozero{3}{4}$ which corresponds to
 $\tp{f}$ of \eqref{eq:Lex43t}.

\subsubsection{The edges of $\mathbb{T}_N^{\Lsetitimonozerodon{},\upmapzerodon{}}$}
 \label{sss:upmaptree}
As stated in Fact \ref{fact:Litimonorecur}, given
the terminal pair
$(\nyh{f\myv{m}{1}}{f\myv{1}{n}},\myv{m}{n})$
of an injective L-shape $f \in \Lsetitimono{m}{n}$ of difference equation type,
either of $f\myv{1}{n-1}$ or $f\myv{m-1}{1}$ is readily computed
depending on $f\myv{m}{1} < f\myv{1}{n}$ or $f\myv{m}{1} > f\myv{1}{n}$.
In other words, depending on
$f\myv{m}{1} < f\myv{1}{n}$ or $f\myv{m}{1} > f\myv{1}{n}$,
the terminal pair of the restriction
$f|_{([m]\times\{1\})\cup(\{1\}\times[n-1])} \in \Lsetitimono{m}{n-1}$ or
$f|_{([m-1]\times\{1\})\cup(\{1\}\times[n])} \in \Lsetitimono{m-1}{n}$ is defined.
Thus, the following inter-level map for the terminal pairs of difference
equation type is defined.
\begin{defi}[{\cite[Definitions 6.6, 6.7]{ranking3}}]\label{defi:upmap}
  Let $m, n$ be positive integers with $m+n \geq 3$. Define the map
  $\upmapzero{m}{n}: \Lsetitimonozero{m}{n} \to \Lsetitimonozero{m-1}{n}
   \sqcup
   \Lsetitimonozero{m}{n-1}$ by
   \begin{equation}
   \upmapzero{m}{n}((s, t,\myv{m}{n})) :=
   \begin{cases}
     (s - n, t, \myv{m-1}{n}) & (s > t), \\
     (s, t - m, \myv{m}{n-1}) & (s < t).
   \end{cases}
   \label{eq:upmapzero}
   \end{equation}
   (Note that this is well-defined when $n = 1$ because
   in this case $m$ satisfies $m \geq 2$ and
   $(s, t,\myv{m}{1}) \in \Lsetitimonozero{m}{1}$ forces
   $t = 1$, thus $s > t$
   and $\upmapzero{m}{1}((s, t,\myv{m}{1})) = (s - 1, t, \myv{m-1}{n})
   \in  \Lsetitimonozero{m-1}{1}$. Also, it is well-defined
   when $m=1$ and
   $\upmapzero{1}{n}((s, t,\myv{1}{n})) = (s, t-1, \myv{1}{n-1})
   \in  \Lsetitimonozero{1}{n-1}$.)
   Also, for $k \geq 3,$ define the map
   $\upmapzerodon{k}: \Lsetitimonozerodon{k} \rightarrow
   \Lsetitimonozerodon{k-1}$ as follows: Given
   $z \in \Lsetitimonozerodon{k}, $ there exist the unique 
   $s, t$ and  $\myv{m}{n}$
    such that
   $z = (s, t, \myv{m}{n})$. For this $(s, t, \myv{m}{n}),$
   let $\upmapzerodon{k}(z) := \upmapzero{m}{n}((s,t,\myv{m}{n})).$
\end{defi}  
Note that %
  the map $\upmapzerodon{k}$ and the transposition
  for terminal pairs
  commute:
\begin{prop}\label{prop:tpupmap}
  For $k \geq 3$ and $z \in \Lsetitimonozerodon{k},$ it holds that
  $$
  \upmapzerodon{k}(\tp{z}) = \tp{(\upmapzerodon{k}(z))}.
  $$
\end{prop}
\begin{proof}
  It is enough to show that
  $\tp{(\upmapzero{m}{n}((s,t,\myv{m}{n})))} = \upmapzero{n}{m}(\tp{(s,t,\myv{m}{n})})$
  for each $\myv{m}{n}$.
  Suppose that
  $(s, t, \myv{m}{n})$ falls into the first case of \eqref{eq:upmapzero}
  due to
  $s > t$. Then it holds that $\upmapzero{m}{n}(s,t,\myv{m}{n}) =
  (s - n, t, \myv{m-1}{n})$ and that
  $\tp{(\upmapzero{m}{n}((s,t,\myv{m}{n})))} = (t, s-n, \myv{n}{m-1})$.
  On the other hand, $(s',t',\myv{m'}{n'}) := \tp{(s,t,\myv{m}{n})}$
  satisfies $s' - t' = t - s < 0$ and falls into the second case
  of \eqref{eq:upmapzero}, implying that
  $\upmapzero{n}{m}(\tp{(s,t,\myv{m}{n})}) =
  \upmapzero{n}{m}((s',t',\myv{m'}{n'})) = (s', t'-m', \myv{m'}{n'-1})
  = (t, s-n, \myv{n}{m-1}).$ Thus, 
  $\tp{(\upmapzero{m}{n}((s,t,\myv{m}{n})))} = \upmapzero{n}{m}(\tp{(s,t,\myv{m}{n})})$
   holds. The argument for the case $s < t$ is similar.
\end{proof}

Inverting $\upmapzerodon{k}$ yields the rule of the
child-generation described in Section \ref{sect:intro}, as follows.
  \begin{fact}[{\cite[Proposition 6.9]{ranking3}}]\label{fact:invupmap}
  It holds that $\Lsetitimonozerodon{2} = \{((1,1,\myv{1}{1})\}$.
  For an integer $k \geq 3,$ the map
  $\upmapzerodon{k}: \Lsetitimonozerodon{k} \rightarrow
  \Lsetitimonozerodon{k-1}$ is surjective and the
  inverse image of a singleton
  $\{(s,t,\myv{m}{n})\} \in \Lsetitimonozero{m}{n}
  \ (\text{where\ } m + n = k - 1)$
  by it satisfies
  the following property. It holds that
  $$
  \upmapzerodon{k}^{-1}(\{(s,t,\myv{m}{n})\}) \subset
  \Lsetitimonozero{m+1}{n} \sqcup \Lsetitimonozero{m}{n+1},
  $$
  that
  \begin{align*}
    |\upmapzerodon{k}^{-1}(\{(s,t,\myv{m}{n})\}) \cap
    \Lsetitimonozero{m+1}{n}| & = 
    \myChi{s - t > -n}, \\
    |\upmapzerodon{k}^{-1}(\{(s,t,\myv{m}{n})\}) \cap
    \Lsetitimonozero{m}{n+1}| & = 
    \myChi{s -t < m }, 
  \end{align*}
and that  
\begin{align*}
  s - t > -n & \Rightarrow
  \upmapzerodon{k}^{-1}(\{(s,t,\myv{m}{n})\}) \cap
  \Lsetitimonozero{m+1}{n} = \{ (s+n, t, \myv{m+1}{n})\}, \\
  s - t < m & \Rightarrow
  \upmapzerodon{k}^{-1}(\{(s,t,\myv{m}{n})\}) \cap
    \Lsetitimonozero{m}{n+1} = \{ (s, t + m, \myv{m}{n+1})\}. \\
\end{align*}
Note that the cases $s-t > -n$ and $s-t < m$ are mutually reducible
by the transposition $\tp{(s,t,\myv{m}{n})}=(t,s,\myv{n}{m})$
for terminal pairs.
\end{fact}  

We have seen that the sequence $\Lsetitimonozerodon{2}, \Lsetitimonozerodon{3},
 \ldots$ of the injective terminal pairs of
difference equation type has the inter-level surjections
$\upmapzerodon{k+3}: \Lsetitimonozerodon{k+3} \to \Lsetitimonozerodon{k+2}$
for $k=0, 1, 2, \ldots$ and has the singleton $\{ (1,1,\myv{1}{1}) \}$
as the smallest level
$\Lsetitimonozerodon{2}$. This defines the adjacency of the tree
$\mathbb{T}_N^{\Lsetitimonozerodon{},\upmapzerodon{}}$.
\begin{defi}[{\cite[Sect. 7]{ranking3}}]\label{defi:upmaptree}
  Let $N$ be a nonnegative integer. For the sequence
  $\Lsetitimonozerodon{2}, \Lsetitimonozerodon{3}, \ldots,
  \Lsetitimonozerodon{N+2}$
  and their inter-level surjections
  $\upmapzerodon{k+3}:\Lsetitimonozerodon{k+3}\to\Lsetitimonozerodon{k+2} ~(0\leq k < N), $ 
  let  $\mathbb{T}_N^{\Lsetitimonozerodon{},\upmapzerodon{}}$ be the rooted tree 
  such that
  (i) the root is the unique element
$(1,1,\myv{1}{1})$
  of $\Lsetitimonozerodon{2}$
  (ii)
  for $0 \leq k \leq N$,
  the set of the level $k$ vertices is $\Lsetitimonozerodon{k+2}$
  (iii) for $0 \leq k < k' \leq N,$ vertices $v \in \Lsetitimonozerodon{k+2}$
  and $v' \in \Lsetitimonozerodon{k'+2}$ are adjacent if and only if
  $k' = k + 1$ and $\upmapzerodon{k'+2}(v') = v.$
In such an adjacent pair $(v,v')$,
$v$ is said to be the parent of $v'$ and $v'$ is said to be
  a child of $v$.
\end{defi}

For a given level $k-1$ vertex $v$ of 
$\mathbb{T}_N^{\Lsetitimonozerodon{},\upmapzerodon{}}$ that is not a leaf,
there is the unique $\myv{m}{n}$ such that $m + n = k+1$ and
$v \in \Lsetitimonozero{m}{n}$.
Then the children of $v$ are 
$\upmapzerodon{k+2}^{-1}(\{v\})$ as given in Fact \ref{fact:invupmap}.
Observe that $1 \leq |\upmapzerodon{k+2}^{-1}(\{v\})| \leq 2$,
that $|\upmapzerodon{k+2}^{-1}(\{v\}) \cap \Lsetitimonozero{m+1}{n}| \leq 1,$
and that $|\upmapzerodon{k+2}^{-1}(\{v\}) \cap \Lsetitimonozero{m}{n+1}|
\leq 1.$
In other words, $v$ has at least one child, has at most one child
whose horizontal index is larger by one than that of $v$, and
has at most one child whose vertical index is larger by one than
that of $v$. In particular, $\mathbb{T}_N^{\Lsetitimonozerodon{},\upmapzerodon{}}$
is a binary tree that all of the leaves has level $N$. 

\subsection{The tree $\mathbb{T}_N^{\TYngdon{},\phidon{}}$ of the Young terminal pairs}
The third tree $\mathbb{T}_N^{\TYngdon{},\phidon{}}$
is seen as an intermediate object between
$\mathbb{T}_N^{\Interdon{},\trimapdon{}}$
and $\mathbb{T}_N^{\Lsetitimonozerodon{},\upmapzerodon{}}$ in a sense.

\subsubsection{Young ranking tables and L-shapes}
To define the set of vertices of the tree, we start with
recalling the definition of \textit{Young ranking tables}
which are a kind of well-known Young tableaux having a rectangular
Ferrers diagram and a parametrization by a positive real number.
\begin{defi}[{\cite[Definition 4.13]{ranking3}}]\label{defi:yr}
  Let $m, n$ be positive integers and $\xi>0$ be a real number.
  A map $\tau_\nyh{1}{\xi}^\nyh{m}{n}: [m] \times [n] \to [mn]$
  defined by
\begin{equation}
  \tau_\nyh{1}{\xi}^{ \nyh{m}{n}} \myv{i}{j}
  := \sum_{s \in [m]} \sum_{t \in [n]}
  \myChi{s + t \xi \leq i + j \xi},
  \quad \myv{i}{j} \in [m] \times [n] \label{eq:taudefxi}
\end{equation}
is said to be a \textit{Young ranking table} of \order $\myv{m}{n}$
if it is injective. The collection of the Young ranking tables
of \order $\myv{m}{n}$ is denoted by $\Yng{m}{n}.$
\end{defi}
We note that RHS of \eqref{eq:taudefxi}
  is interpreted as the \textit{ranking} of the element
  $i + j\xi$ among the set $\{s + t\xi : \myv{s}{t} \in [m] \times [n]\}$
  by the standard order of the real numbers.

\begin{defi}[{\cite[Notation 7.2]{ranking3}}]    
For positive integers $m$ and $n$,
let
$$
\LYng{m}{n} = \{ \tau|_{([m]\times\{1\})\cup(\{1\}\times[n])} : \tau \in \Yng{m}{n}\}
$$
be the collection of the resulting injective L-shapes
  by restricting all the Young ranking tables to the set
  $([m]\times\{1\})\cup(\{1\}\times[n])$.
\end{defi}
Important facts are
that $\LYng{m}{n}$ can be identified with $\Yng{m}{n}$ and
that it is included in $\Lsetitimono{m}{n}$.
\begin{fact}[{\cite[Corollary 7.5]{ranking3}}]\label{fact:YandLY}
  The restriction
  $\Yng{m}{n} \ni \tau \mapsto \tau|_{([m]\times\{1\})\cup(\{1\}\times[n])}
  \in \LYng{m}{n}
  $ is bijective.
\end{fact}
\begin{fact}[{\cite[Proposition 7.7]{ranking3}}]\label{fact:Lincl}
For positive integers $m, n$, 
the inclusion
$$
\LYng{m}{n} \subset \Lsetitimono{m}{n}
$$
holds. 
\end{fact}

\subsubsection{Level $k$ vertices of the Young terminal pairs}
Thus, Fact \ref{fact:Litimonorecur} is in particular applicable to
the Young L-shapes $\LYng{m}{n}$ and the following set
of terminal pairs
is
essentially same as $\LYng{m}{n}$.
\begin{defi}[{\cite[Notation 7.8]{ranking3}}]\label{defi:TYng}
    For positive integers $m, n$, let
    $$
    \TYng{m}{n} :=
    \{ (\theta\myv{m}{1}, \theta\myv{1}{n}, \myv{m}{n}):
     \theta \in \Yng{m}{n} \},
    $$
which is a subset of $\Lsetitimonozero{m}{n}$ by Fact \ref{fact:Lincl}.
An element of $\TYng{m}{n}$ is referred to as a Young terminal pair.
\end{defi}
Combining the above with Fact \ref{fact:YandLY}, the Young terminal pairs
$\TYng{m}{n}$ can be
identified with the Young ranking tables $\Yng{m}{n}$.

Now, level $k$ vertices of the tree
are defined to be the
union of $\TYng{m}{n}$ with $m+n=k+2$.
\begin{defi}[{\cite[Notation 7.8]{ranking3}}]
For $k \geq 0$, let
\begin{align*}
\TYngdon{k+2} & := \bigsqcup_{\begin{smallmatrix}m, n\geq 1 \\ m + n = k+2 \end{smallmatrix}} \TYng{m}{n}.
\end{align*}
\end{defi}
By Fact \ref{fact:Lincl}, the set of level $k$ vertices
in $\mathbb{T}_N^{\TYngdon{},\phidon{}}$ is a subset
of the level $k$ vertices of 
$\mathbb{T}_N^{\Lsetitimonozerodon{},\upmapzerodon{}}$.

\subsubsection{The edges of $\mathbb{T}_N^{\TYngdon{},\phidon{}}$}
\label{sss:phitree}
To develop the inter-level maps for the Young terminal pairs that are
parallel to the inter-level maps for the generalized Farey
intervals and/or the injective terminal pairs of difference equation
type, we recall \textit{yet another {\mysur}'s bijection.}
\begin{fact}[{\cite[Corollary 11]{ranking}} {\cite[Proposition 4.14]{ranking3}}]\label{fact:yas}
A
map
  \begin{equation}
    \yas{m-1}{n-1}:\Gint{m-1}{n-1} \rightarrow \Yng{m}{n}
  \end{equation}
is defined by the following way;
  Given $(a,b) \in \Gint{m-1}{n-1},$ take $\xi \in (a, b)$ arbitrarily,
  then let
  \begin{equation}
    \yas{m-1}{n-1}((a,b)) := \tau_\nyh{1}{\xi}^{ \nyh{m}{n}}.
  \end{equation}
  The map is well-defined and is bijective also.
  The bijection $\yas{m-1}{n-1}$ 
  is referred to as \textit{yet another {\mysur}'s bijection}.
\end{fact}

Combining the above and the fact noted below Definition \ref{defi:TYng}
that the Young terminal pairs
$\TYng{m}{n}$ can be
identified with the Young ranking tables $\Yng{m}{n}$,
the bijection $\yas{m-1}{n-1}:\Gint{m-1}{n-1} \rightarrow \Yng{m}{n}$
induces the following bijection.
\begin{defi}\label{defi:yast}
For positive  integers $m, n$, let $\yast{m-1}{n-1}$ denote the bijection
$$
\yast{m-1}{n-1}:
\Inter{m-1}{n-1} \ni ((a,b),\myv{m-1}{n-1})
\mapsto
\left(\tau\myv{m}{1}, \tau\myv{1}{n},
\myv{m}{n}  \right) \in \TYng{m}{n},
$$
where $\tau = \yas{m-1}{n-1}((a,b)) \in \Yng{m}{n}.$ 
\end{defi}
\begin{exam}\label{exam:rank}
In Fig.~\ref{fig:yast}, the elements of $\Inter{3}{2}$ and their
images by $\yast{3}{2}$ which are elements of $\TYng{4}{3}$ are depicted.
The 7 elements of $\Garey{3}{2}$ are represented by dashed lines
of the corresponding slopes. A generalized Farey interval
$(a,b)$ corresponds to a region between a pair of two dashed lines
(of slopes $a$ and $b$ respectively).
Such a region is represented by a small circle inside it,
followed by the label $((a,b),\myv{3}{2})$ then $\mapsto$
$\yast{3}{2}(((a,b),\myv{3}{2}))$ which is of the form
$(
\tau_{\nyh{1}{\xi}}^{ \nyh{4}{3}}\myv{4}{1},
\tau_{\nyh{1}{\xi}}^{ \nyh{4}{3}}\myv{1}{3}, \myv{4}{3})$.
For example, %
the label
\begin{equation}
\left(\left(\tfrac{1}{1},\tfrac{3}{2}\right), \tbinom{3}{2}\right)\mapsto \left(7,6,\tbinom{4}{3}\right)
\end{equation}
following the \Th{4} circle from the left
indicates that
$$
\yast{3}{2}\left(\left(\tfrac{1}{1}, \tfrac{3}{2} \right), \myv{3}{2}\right) = (\tau_{\nyh{1}{\xi}}^{ \nyh{4}{3}}\myv{4}{1},
\tau_{\nyh{1}{\xi}}^{ \nyh{4}{3}}\myv{1}{3}, \myv{4}{3}) = (7,6,\myv{4}{3}),
$$
where we can use, say, %
$\xi = \frac{5}{4}$ as an element of the open interval
$\left(\frac{1}{1}, \frac{3}{2} \right)$. %
RHS is computed using \eqref{eq:taudefxi} as follows:
For $\xi = \frac{5}{4} = 1.25,$ it holds that
\begin{equation}
  |i + \xi j|_{\myv{i}{j} \in [4] \times [3]}
  = \left|\begin{matrix}
4.75 & 5.75 & 6.75 & 7.75\\
3.5  & 4.5  & 5.5  & 6.5\\
2.25 & 3.25 & 4.25 & 5.25
  \end{matrix}\right|,
\end{equation}
where $\myv{i}{j} = \myv{1}{1}, \myv{4}{1}, \myv{1}{3}$
  correspond to the lower-left, lower-right, upper-left corners respectively,
which yields
\begin{equation}
\left|\sum_{s \in [4]} \sum_{t \in [3]}
\myChi{s + \xi t \leq i + \xi j}
\right|_{\myv{i}{j} \in [4] \times [3]}
  = \left|\ensuremath{\begin{matrix}
    6 & 9 & 11 & 12 \\
    3 & 5 & 8 & 10 \\
    1 & 2 & 4 & 7
  \end{matrix}}\right|
\end{equation}
from which we obtain
$\tau_{\nyh{1}{\xi}}^{ \nyh{4}{3}}\myv{4}{1} = 7$
and
$\tau_{\nyh{1}{\xi}}^{ \nyh{4}{3}}\myv{1}{3} = 6$.
\begin{figure}
\centering  
\rotatebox{0}{\scalebox{0.8}{
    \begin{tikzpicture}[scale=7]
      \clip (3,2) rectangle + (1.6,1.0);
\draw[color=gray,dashed] (3,2) -- (4,2);
\draw[color=gray,dashed] (3,2) -- (3,3);

\draw[fill=white] (0.5,0.5) circle (0.01);

\draw[fill=white] (1.5,0.5) circle (0.01);

\draw[fill=white] (2.5,0.5) circle (0.01);

\draw[fill=white] (3.5,0.5) circle (0.01);

\draw[fill=white] (4.5,0.5) circle (0.01);

\draw[fill=white] (5.5,0.5) circle (0.01);
\draw (5.5,0.5) node [right] {$\left(\left(\frac{0}{1},\infty\right), \binom{5}{0}\right)\mapsto \left(6,1,\binom{6}{1}\right)$};

\draw[color=gray, dashed] (4,1) -- (4.25,2);
\draw[fill=white] (4.1,1.9) circle (0.01);
\draw (4.1,1.9) node [right] {$\left(\left(\frac{4}{1},\infty\right), \binom{4}{1}\right)\mapsto \left(5,6,\binom{5}{2}\right)$};

\draw[color=gray, dashed] (3,1) -- (3.33333,2);
\draw[fill=white] (3.125,1.875) circle (0.01);

\draw[color=gray, dashed] (4,1) -- (4.33333,2);
\draw[fill=white] (4.225,1.775) circle (0.01);
\draw (4.225,1.775) node [right] {$\left(\left(\frac{3}{1},\frac{4}{1}\right), \binom{4}{1}\right)\mapsto \left(6,5,\binom{5}{2}\right)$};

\draw[color=gray, dashed] (3,2) -- (3.33333,3);
\draw[fill=white] (3.125,2.875) circle (0.01);
\draw (3.125,2.875) node [right] {$\left(\left(\frac{3}{1},\infty\right), \binom{3}{2}\right)\mapsto \left(4,9,\binom{4}{3}\right)$};

\draw[color=gray, dashed] (2,1) -- (2.5,2);
\draw[fill=white] (2.16667,1.83333) circle (0.01);

\draw[color=gray, dashed] (3,1) -- (3.5,2);
\draw[fill=white] (3.29167,1.70833) circle (0.01);

\draw[color=gray, dashed] (4,1) -- (4.5,2);
\draw[fill=white] (4.29167,1.70833) circle (0.01);
\draw (4.29167,1.70833) node [right] {$\left(\left(\frac{2}{1},\frac{3}{1}\right), \binom{4}{1}\right)\mapsto \left(7,4,\binom{5}{2}\right)$};

\draw[color=gray, dashed] (3,2) -- (3.5,3);
\draw[fill=white] (3.29167,2.70833) circle (0.01);
\draw (3.29167,2.70833) node [right] {$\left(\left(\frac{2}{1},\frac{3}{1}\right), \binom{3}{2}\right)\mapsto \left(5,8,\binom{4}{3}\right)$};

\draw[color=gray, dashed] (2,2) -- (2.5,3);
\draw[fill=white] (2.16667,2.83333) circle (0.01);

\draw[color=gray, dashed] (2,3) -- (2.5,4);
\draw[fill=white] (2.16667,3.83333) circle (0.01);
\draw (2.16667,3.83333) node [right] {$\left(\left(\frac{2}{1},\infty\right), \binom{2}{3}\right)\mapsto \left(3,10,\binom{3}{4}\right)$};

\draw[color=gray, dashed] (1,1) -- (2,2);
\draw[fill=white] (1.25,1.75) circle (0.01);

\draw[color=gray, dashed] (2,1) -- (3,2);
\draw[fill=white] (2.41667,1.58333) circle (0.01);

\draw[color=gray, dashed] (3,1) -- (4,2);
\draw[fill=white] (3.41667,1.58333) circle (0.01);

\draw[color=gray, dashed] (4,1) -- (5,2);
\draw[fill=white] (4.41667,1.58333) circle (0.01);
\draw (4.41667,1.58333) node [right] {$\left(\left(\frac{1}{1},\frac{2}{1}\right), \binom{4}{1}\right)\mapsto \left(8,3,\binom{5}{2}\right)$};

\draw[color=gray, dashed] (3,2) -- (3.66667,3);
\draw[fill=white] (3.36667,2.63333) circle (0.01);
\draw (3.36667,2.63333) node [right] {$\left(\left(\frac{3}{2},\frac{2}{1}\right), \binom{3}{2}\right)\mapsto \left(6,7,\binom{4}{3}\right)$};

\draw[color=gray, dashed] (2,2) -- (3,3);
\draw[fill=white] (2.41667,2.58333) circle (0.01);

\draw[color=gray, dashed] (3,2) -- (4,3);
\draw[fill=white] (3.45,2.55) circle (0.01);
\draw (3.45,2.55) node [right] {$\left(\left(\frac{1}{1},\frac{3}{2}\right), \binom{3}{2}\right)\mapsto \left(7,6,\binom{4}{3}\right)$};

\draw[color=gray, dashed] (2,3) -- (3,4);
\draw[fill=white] (2.41667,3.58333) circle (0.01);
\draw (2.41667,3.58333) node [right] {$\left(\left(\frac{1}{1},\frac{2}{1}\right), \binom{2}{3}\right)\mapsto \left(4,9,\binom{3}{4}\right)$};

\draw[color=gray, dashed] (1,2) -- (2,3);
\draw[fill=white] (1.25,2.75) circle (0.01);

\draw[color=gray, dashed] (1,3) -- (2,4);
\draw[fill=white] (1.25,3.75) circle (0.01);

\draw[color=gray, dashed] (1,4) -- (2,5);
\draw[fill=white] (1.25,4.75) circle (0.01);
\draw (1.25,4.75) node [right] {$\left(\left(\frac{1}{1},\infty\right), \binom{1}{4}\right)\mapsto \left(2,9,\binom{2}{5}\right)$};

\draw[fill=white] (0.5,1.5) circle (0.01);

\draw[fill=white] (0.5,2.5) circle (0.01);

\draw[fill=white] (0.5,3.5) circle (0.01);

\draw[fill=white] (0.5,4.5) circle (0.01);

\draw[fill=white] (0.5,5.5) circle (0.01);
\draw (0.5,5.5) node [right] {$\left(\left(\frac{0}{1},\infty\right), \binom{0}{5}\right)\mapsto \left(1,6,\binom{1}{6}\right)$};

\draw[fill=white] (1.9,4.1) circle (0.01);
\draw (1.9,4.1) node [right] {$\left(\left(\frac{0}{1},\frac{1}{4}\right), \binom{1}{4}\right)\mapsto \left(6,5,\binom{2}{5}\right)$};

\draw[fill=white] (1.875,3.125) circle (0.01);

\draw[fill=white] (2.875,3.125) circle (0.01);
\draw (2.875,3.125) node [right] {$\left(\left(\frac{0}{1},\frac{1}{3}\right), \binom{2}{3}\right)\mapsto \left(9,4,\binom{3}{4}\right)$};

\draw[color=gray, dashed] (1,4) -- (2,4.25);
\draw[fill=white] (1.775,4.225) circle (0.01);
\draw (1.775,4.225) node [right] {$\left(\left(\frac{1}{4},\frac{1}{3}\right), \binom{1}{4}\right)\mapsto \left(5,6,\binom{2}{5}\right)$};

\draw[fill=white] (1.83333,2.16667) circle (0.01);

\draw[fill=white] (2.83333,2.16667) circle (0.01);

\draw[fill=white] (3.83333,2.16667) circle (0.01);
\draw (3.83333,2.16667) node [right] {$\left(\left(\frac{0}{1},\frac{1}{2}\right), \binom{3}{2}\right)\mapsto \left(10,3,\binom{4}{3}\right)$};

\draw[color=gray, dashed] (1,3) -- (2,3.33333);
\draw[fill=white] (1.70833,3.29167) circle (0.01);

\draw[color=gray, dashed] (2,3) -- (3,3.33333);
\draw[fill=white] (2.70833,3.29167) circle (0.01);
\draw (2.70833,3.29167) node [right] {$\left(\left(\frac{1}{3},\frac{1}{2}\right), \binom{2}{3}\right)\mapsto \left(8,5,\binom{3}{4}\right)$};

\draw[color=gray, dashed] (1,4) -- (2,4.33333);
\draw[fill=white] (1.70833,4.29167) circle (0.01);
\draw (1.70833,4.29167) node [right] {$\left(\left(\frac{1}{3},\frac{1}{2}\right), \binom{1}{4}\right)\mapsto \left(4,7,\binom{2}{5}\right)$};

\draw[fill=white] (1.75,1.25) circle (0.01);

\draw[fill=white] (2.75,1.25) circle (0.01);

\draw[fill=white] (3.75,1.25) circle (0.01);

\draw[fill=white] (4.75,1.25) circle (0.01);
\draw (4.75,1.25) node [right] {$\left(\left(\frac{0}{1},\frac{1}{1}\right), \binom{4}{1}\right)\mapsto \left(9,2,\binom{5}{2}\right)$};

\draw[color=gray, dashed] (1,2) -- (2,2.5);
\draw[fill=white] (1.58333,2.41667) circle (0.01);

\draw[color=gray, dashed] (2,2) -- (3,2.5);
\draw[fill=white] (2.58333,2.41667) circle (0.01);

\draw[color=gray, dashed] (3,2) -- (4,2.5);
\draw[fill=white] (3.58333,2.41667) circle (0.01);
\draw (3.58333,2.41667) node [right] {$\left(\left(\frac{1}{2},\frac{1}{1}\right), \binom{3}{2}\right)\mapsto \left(9,4,\binom{4}{3}\right)$};

\draw[color=gray, dashed] (2,3) -- (3,3.66667);
\draw[fill=white] (2.55,3.45) circle (0.01);
\draw (2.55,3.45) node [right] {$\left(\left(\frac{2}{3},\frac{1}{1}\right), \binom{2}{3}\right)\mapsto \left(6,7,\binom{3}{4}\right)$};

\draw[color=gray, dashed] (1,3) -- (2,3.5);
\draw[fill=white] (1.58333,3.41667) circle (0.01);

\draw[color=gray, dashed] (2,3) -- (3,3.5);
\draw[fill=white] (2.63333,3.36667) circle (0.01);
\draw (2.63333,3.36667) node [right] {$\left(\left(\frac{1}{2},\frac{2}{3}\right), \binom{2}{3}\right)\mapsto \left(7,6,\binom{3}{4}\right)$};

\draw[color=gray, dashed] (1,4) -- (2,4.5);
\draw[fill=white] (1.58333,4.41667) circle (0.01);
\draw (1.58333,4.41667) node [right] {$\left(\left(\frac{1}{2},\frac{1}{1}\right), \binom{1}{4}\right)\mapsto \left(3,8,\binom{2}{5}\right)$};

\end{tikzpicture}
}}
\caption{$\yast{3}{2}:\Inter{3}{2} \to \TYng{4}{3}$}\label{fig:yast}
\end{figure}
\end{exam}

The yet another {\mysur}'s bijection of Definition \ref{defi:yast}
is extended to the bijection from $\Interdon{k-2}$ to
$\TYngdon{k}$ in an obvious way.
\begin{defi}\label{defi:yastdon}
  For an integer $k \geq 3,$ the bijection
  $\yastdon{k-1}:\Interdon{k-2}\to \TYngdon{k}$ is defined as follows;
  Given $z \in \Interdon{k-2},$ let $\myv{m}{n}$ be the unique pair
  of integers such that $(m-1)+(n-1)=k-2$ and $z \in \Inter{m-1}{n-1}.$
  Let $\yastdon{k-1}(z) := \yast{m-1}{n-1}(z) \in \TYng{m}{n} \subset \TYngdon{k}.$
\end{defi}
Again we note that it commutes with the transposition:
\begin{prop}\label{prop:tpyast}
For an integer $k \geq 3$, it follows that
$$
\tp{(\cdot)}\circ\yastdon{k-1} = \yastdon{k-1}\circ\tp{(\cdot)},
$$
where $\tp{(\cdot)}$ of LHS is for terminal pairs
$\TYngdon{k}$, whereas
that of RHS is for 
generalized Farey intervals
$\Interdon{k-2}$.
\end{prop}
\begin{proof}
Suppose that $((a,b), \myv{m-1}{n-1}) \in \Interdon{k-2}$ with $(m-1) + (n-1) = k-2.$
By taking $\xi \in (a, b)$, we have
$\yastdon{k-1}(((a,b), \myv{m-1}{n-1})) =
(\tau_{\nyh{1}{\xi}}^{\nyh{m}{n}}\myv{m}{1}, \tau_{\nyh{1}{\xi}}^{\nyh{m}{n}}\myv{1}{n},
\myv{m}{n})$, hence 
$\tp{(\yastdon{k-1}(((a,b), \myv{m-1}{n-1})))}
 = (\tau_{\nyh{1}{\xi}}^{\nyh{m}{n}}\myv{1}{n}, 
 \tau_{\nyh{1}{\xi}}^{\nyh{m}{n}}\myv{m}{1},\allowbreak
 \myv{n}{m})$ by Definition \ref{defi:tpterm}.
 On the other hand, it holds that
 $\yastdon{k-1}(\tp{((a,b), \myv{m-1}{n-1})})
 = \yastdon{k-1}(((b^{-1},a^{-1}), \myv{n-1}{m-1})).$
By $\xi^{-1} \in (b^{-1},a^{-1}),$ we have a presentation
$\yastdon{k-1}(\tp{((a,b), \myv{m-1}{n-1})}) = 
(\tau_{\nyh{1}{\xi^{-1}}}^{\nyh{n}{m}}\myv{n}{1}, \tau_{\nyh{1}{\xi^{-1}}}^{\nyh{n}{m}}\myv{1}{m},
\myv{n}{m}).$
From Definition \ref{defi:yr}, it follows that
\begin{equation}
\tau_\nyh{1}{\xi}^{ \nyh{m}{n}} \myv{i}{j}
  = \sum_{s \in [m]} \sum_{t \in [n]}
  \myChi{s + t \xi \leq i + j \xi}
  =
  \sum_{t \in [n]}
  \sum_{s \in [m]} 
  \myChi{t + s \xi^{-1}   \leq j + i \xi^{-1}  }
  =
  \tau_\nyh{1}{\xi^{-1}}^{ \nyh{n}{m}} \myv{j}{i} \quad (\myv{i}{j} \in [m]\times[n])
\end{equation}
and in particular that 
\begin{equation}\label{eq:tptau}
(\tau_{\nyh{1}{\xi}}^{\nyh{m}{n}}\myv{1}{n}, \tau_{\nyh{1}{\xi}}^{\nyh{m}{n}}\myv{m}{1},
\myv{n}{m})
=
(\tau_{\nyh{1}{\xi^{-1}}}^{\nyh{n}{m}}\myv{n}{1}, \tau_{\nyh{1}{\xi^{-1}}}^{\nyh{n}{m}}\myv{1}{m},\myv{n}{m}),
\end{equation}
i.e., $\tp{(\yastdon{k-1}(((a,b), \myv{m-1}{n-1})))} =
\yastdon{k-1}(\tp{((a,b), \myv{m-1}{n-1})})$.
\end{proof}
  In particular, the image of $\TYngdon{k}$
  under the
  transposition for terminal pairs
  is included in the image of $\yastdon{k-1}$. Thus,
$\tp{(\cdot)}:\Lsetitimonozerodon{k} \to \Lsetitimonozerodon{k}$
induces
$\tp{(\cdot)}:\TYngdon{k} \to \TYngdon{k}$.

The inter-level surjection for the third tree is defined as follows.
\begin{defi}\label{defi:compat}
  For each integer $k \geq 2$,
  define a map $\phidon{k+1}:\TYngdon{k+1} \to \TYngdon{k}$
  as $\yastdon{k-1}\circ\trimapdon{k-1}\circ\yastdon{k}^{-1}$.
  In other words, $\phidon{k+1}$ is the unique map from
  $\TYngdon{k+1} $ to $\TYngdon{k}$ that makes the following
  diagram commutative:
  $$
\begin{CD}
\Interdon{k-1} @>\yastdon{k}>\text{bijective}> \TYngdon{k+1} \\
@V\trimapdon{k-1}VV @VV{\phidon{k+1}}V \\
\Interdon{k-2} @>\text{bijective}>\yastdon{k-1}> \TYngdon{k} 
\end{CD}  
$$
Note that
  $\phidon{k+1}:\TYngdon{k+1} \to \TYngdon{k}$ is surjective
  because of the surjectivity of $\trimapdon{k-1}$
  and the bijectivity of $\yastdon{k}$ and $\yastdon{k-1}$.
\end{defi}
The third tree $\mathbb{T}_N^{\TYngdon{},\phidon{}}$
  is such a tree that has the Young terminal pairs as its vertices
  and that the adjacency is defined by the inter-level surjections
  $\phidon{}$ which reflect the adjacency of
  $\mathbb{T}_N^{\Interdon{}, \trimapdon{}}.$
\begin{defi}\label{defi:phitree}
  Let $N$ be a nonnegative integer. For the sequence
  $\TYngdon{2}, \TYngdon{3}, \ldots,
  \TYngdon{N+2}$
  and their inter-level surjections
  $\phidon{k+3}:\TYngdon{k+3}\to\TYngdon{k+2} ~(0\leq k < N), $ 
  let  $\mathbb{T}_N^{\TYngdon{},\phidon{}}$ be the rooted tree 
  such that
  (i) the root is the unique element
$(1,1,\myv{1}{1})$
  of $\TYngdon{2}$
  (ii)
  for $0 \leq k \leq N$,
  the set of the level $k$ vertices is $\TYngdon{k+2}$
  (iii) for $0 \leq k < k' \leq N,$ vertices $v \in \TYngdon{k+2}$
  and $v' \in \TYngdon{k'+2}$ are adjacent if and only if
  $k' = k + 1$ and $\phidon{k'+2}(v') = v.$
In such an adjacent pair $(v,v')$,
$v$ is said to be the parent of $v'$ and $v'$ is said to be
a child of $v$. 
\end{defi}

  Definitions \ref{defi:compat} and \ref{defi:phitree}
  immediately yield the following assertion:
  \begin{prop}\label{prop:isom13}
The trees $\mathbb{T}_N^{\Interdon{},\trimapdon{}}$
    and $\mathbb{T}_N^{\TYngdon{},\phidon{}}$ are isomorphic
    where the correspondence between the vertices are given
    by the bijections $\yastdon{k+1}:\Interdon{k} \to \TYngdon{k+2},$
    $k = 0, 1, \ldots, N$.
  \end{prop}

\section{Isomorphism of the 3 trees}\label{sect:trees}

In \cite{ranking3}, the tree $\mathbb{T}_N^{\Interdon{},\trimapdon{}}$
was referred to as $\mathbb{T}_N^1$ (with replacing
each vertex $v \in \Interdon{k-1}\ (k=1, \ldots, N+1)$ by
$\yastdon{k-1}(v) \in \Lsetitimonozerodon{k+1}$ bijectively) and 
 $\mathbb{T}_N^{\Lsetitimonozerodon{},\upmapzerodon{}}$ was
 referred to as $\mathbb{T}_N^2$. By using a computer program,
 it was verified that 
 $\mathbb{T}_N^1$ and $\mathbb{T}_N^2$ are isomorphic as
 rooted trees upto $N=1000$.

A major part of the
 result of this paper is the following:
\begin{theo}\label{theo:main}
  For each nonnegative integer $N$, it holds that
    $\mathbb{T}_N^{\TYngdon{},\phidon{}}
 = \mathbb{T}_N^{\Lsetitimonozerodon{},\upmapzerodon{}}$.
\end{theo}
Combining Theorem \ref{theo:main}
  with Proposition \ref{prop:isom13},
we have the theorem stated in Section \ref{sect:intro} as a corollary:
\begin{coro}\label{coro:isom}
The trees $\mathbb{T}_N^{\Interdon{},\trimapdon{}}$
    and $\mathbb{T}_N^{\Lsetitimonozerodon{},\upmapzerodon{}}$ are isomorphic
    where the correspondence between the vertices are given
    by the bijections $\yastdon{k+1}:\Interdon{k} \to \TYngdon{k+2} = \Lsetitimonozerodon{k+2},$
    $k = 0, 1, \ldots, N$.
\end{coro}
Thus, all the three trees $\mathbb{T}_N^{\Interdon{},\trimapdon{}},
\mathbb{T}_N^{\Lsetitimonozerodon{},\upmapzerodon{}}$ and 
$\mathbb{T}_N^{\TYngdon{},\phidon{}}$ are isomorphic.
  As an example, a figure of the three isomorphic trees for $N=4$
  is given in Appendix \ref{sect:example}.

To prove Theorem \ref{theo:main}, let
us introduce a terminology that is convenient for the classification
of the children of a given vertex in a tree.
\begin{defi}\label{defi:hvchild}
  Let $v = ((a, b), \myv{m}{n})$ is a vertex of $\mathbb{T}_N^{\Interdon{}, \trimapdon{}}.$ The (potentially non-existent and unique if exists)
  child belonging to $\Inter{m+1}{n}$ (resp. $\Inter{m}{n+1}$)
  is referred to as the 
  \textit{horizontal child} (resp. \textit{vertical child}) of $v$.
  For a vertex $v = (s, t, \myv{m}{n})$ of 
  $\mathbb{T}_N^{\Lsetitimonozerodon{},\upmapzerodon{}}$
  (resp. $\mathbb{T}_N^{\TYngdon{},\phidon{}}$),
  its horizontal child in $\Lsetitimonozero{m+1}{n}$
  (resp. $\TYng{m+1}{n}$)
  and vertical child in $\Lsetitimonozero{m}{n+1}$ 
  (resp. $\TYng{m}{n+1}$) are
  defined similarly.
Note that, in any of the trees,
the horizontal (resp. vertical)
child of $\tp{v}$ is the transposition (for generalized Farey intervals/terminal pairs)
of the vertical (resp. horizontal) child
of $v$.
\end{defi}

\subsection{Lemmas and proofs} \label{sect:lemmaproof}
The proof of Theorem \ref{theo:main} depends on
  Lemmas \ref{lemma:phiupmapzero} and \ref{lem:branch} presented
  in this section.
We start with proving some propositions that are useful
for the proof of these lemmas.
The first one is the claim that 
a terminal pair of difference equation type is determined by
their difference.
\begin{prop}\label{prop:diffitimonozero}
  Let $k \geq 2$ and let $m, n$ be positive integers satisfying
  $m + n = k$. Let $v=(s, t, \myv{m}{n})$ and $v'=(s',t',\myv{m}{n})$
  be elements of $\Lsetitimonozero{m}{n}$. If $s - t = s' - t'$, then
  $v = v'.$
\end{prop}
\begin{proof}
  Induction on $k$. For $k = 2$, the only possible case is $m = n = 1$
  and $\Lsetitimonozero{1}{1}$ is a singleton $\{(1,1,\myv{1}{1})\}$,
  hence $v = v'$ obviously holds. Suppose that $k \geq 3$ and
  the assertion is valid
  for $k-1$.
  Suppose that $\myv{m}{n}$ satisfies $m + n = k$ and that
  $v=(s, t, \myv{m}{n}), v'=(s',t',\myv{m}{n}) \in \Lsetitimonozero{m}{n}$
  satisfy $s - t = s' - t'$.
  Consider the case
  $s > t.$ Then, by \eqref{eq:upmapzero} of Definition \ref{defi:upmap}, 
  it holds that $\upmapzero{m}{n}(v) = (s - n, t, \myv{m-1}{n})$.
  Also by
  $s' - t' = s - t > 0$, we have
  $\upmapzero{m}{n}(v') = (s' - n, t', \myv{m-1}{n})$.
  Then the difference of the terminal pair
  $\upmapzero{m}{n}(v)$ is $(s-n) - t = s  - t - n$ which coincides with
  the difference $(s'-n) - t' = s'  - t' - n = s - t - n$ of
  $\upmapzero{m}{n}(v')$.
  By $(m-1)+n=k-1$
  and
  the inductive assumption, we have $\upmapzero{m}{n}(v) = \upmapzero{m}{n}(v')$
  and thus $s-n = s'-n, t = t'$. Therefore, $v = v'$ holds when $s > t$.
  For the remaining case $s< t$,
  taking the transposition $\tp{v}$ and $\tp{v'}$
  for terminal pairs ({}Definition \ref{defi:tpterm})
  reduces the problem to the case $s > t$ (keeping the level $k=m+n$)
  and we have $\tp{v} = \tp{v'},$
  thus $v = v'$. This completes the induction.
\end{proof}

Since the difference is essential to identify a terminal pair of
difference equation type, we represent it as a real function
that is convenient to analyze Young terminal pairs.
\begin{defi}\label{defi:delf}
  For positive integers $m, n$ and
  a positive real number $\xi$, let
  \begin{equation}\label{eq:delfdef}
    \delf{m}{n}{\xi} :=
    \tau_\nyh{1}{\xi}^{ \nyh{m}{n}} \myv{m}{1}
    -
    \tau_\nyh{1}{\xi}^{ \nyh{m}{n}} \myv{1}{n},
  \end{equation}
  where $\tau_\nyh{1}{\xi}^{ \nyh{m}{n}}$ is as in \eqref{eq:taudefxi}
  of Definition \ref{defi:yr}.
\end{defi}
\begin{prop}\label{prop:delfinv}
For positive integers $m, n$ and
a positive real number $\xi$, it holds that
\begin{equation}
\delf{m}{n}{\xi} = -\delf{n}{m}{\xi^{-1}}.
\end{equation}
\end{prop}
\begin{proof}
  Rewriting RHS by using \eqref{eq:tptau} in the proof of
  Proposition \ref{prop:tpyast}, the assertion follows immediately.
\end{proof}
By %
rearranging the inequality in the Iverson-Knuth
brackets of \eqref{eq:taudefxi}, we have
\begin{equation}
  \tau_\nyh{1}{\xi}^{ \nyh{m}{n}} \myv{m}{1}
  =   \sum_{s=1}^{m} \sum_{t=1}^{n} \myChi{s \leq m + \xi (1-t)},  \quad
  \tau_\nyh{1}{\xi}^{ \nyh{m}{n}} \myv{1}{n}
  =   \sum_{s=1}^{m} \sum_{t=1}^{n} \myChi{s \leq 1 + \xi (n-t)},
\label{eq:xiterminal}
\end{equation}
respectively. This presentation gives the following
fundamental properties of
$\tau_\nyh{1}{\xi}^{ \nyh{m}{n}} \myv{m}{1},
\tau_\nyh{1}{\xi}^{ \nyh{m}{n}} \myv{1}{n}$ and
$\delf{m}{n}{\xi}$.
\begin{prop}\label{prop:fund}
  Let $m, n$ be positive integers with $m+n \geq 3$
  and let $\xi \in (0, \infty).$ Then,
\begin{itemize}
\item[(i)] $\tau_\nyh{1}{\xi}^{ \nyh{m}{n}} \myv{m}{1}$ is a non-increasing
  function of $\xi$.
\item[(ii)] $\tau_\nyh{1}{\xi}^{ \nyh{m}{n}} \myv{1}{n}$ is a non-decreasing
  function of $\xi$.
\item[(iii)] $\delf{m}{n}{\xi}$ is a non-increasing
  function of $\xi$.
\item[(iv)]
  $\delf{m}{n}{\xi} = \delf{m}{n}{\xi'}$ implies
  $(\tau_\nyh{1}{\xi}^{ \nyh{m}{n}} \myv{m}{1},
  \tau_\nyh{1}{\xi}^{ \nyh{m}{n}} \myv{1}{n},
  \myv{m}{n}) =
  (\tau_\nyh{1}{\xi'}^{ \nyh{m}{n}} \myv{m}{1},
  \tau_\nyh{1}{\xi'}^{ \nyh{m}{n}} \myv{1}{n},
  \myv{m}{n})$.
\item[(v)] If $n \geq 2$, then $\delf{m}{n}{\frac{m-1}{n-1}} = 0$.
\item[(vi)]
  If $m \geq 2$, then
  \begin{equation}
    \delf{m}{n}{\frac{m-1}{n}} = m - 1.
    \label{eq:vifirst}
  \end{equation}  
  Furthermore, it holds that
  \begin{equation}
    \delf{m}{n}{\frac{m-1}{n} - \omicron} = m-1+(\gcd(m-1,n)-1),
    \label{eq:visecond}
  \end{equation}
  where $x - \omicron$ means $x - h$ with sufficiently small $h > 0$.
  (More formally, LHS $\delf{m}{n}{\frac{m-1}{n\rule{0mm}{1.42ex}} - \omicron}$
   is written as $\lim_{h \downarrow 0} \delf{m}{n}{\frac{m-1}{n\rule{0mm}{1.42ex}} - h}$).
 \item[(vii)]  If $n \geq 2$, then
   \begin{equation}
   \delf{m}{n}{\frac{m}{n-1}} = 1-n. \label{eq:viifirst}
  \end{equation}
   Furthermore, it holds that
   \begin{equation}
     \delf{m}{n}{\frac{m}{n-1} + \omicron} = 1-n-(\gcd(m,n-1)-1),
\label{eq:viisecond}
   \end{equation}
  where $x + \omicron$ means $x + h$ with sufficiently small $h > 0$.
  (More formally, LHS $\delf{m}{n}{\frac{m}{n-1} + \omicron}$
   is written as $\lim_{h \downarrow 0} \delf{m}{n}{\frac{m}{n-1} + h}$).
\item[(viii)] If $n \geq 2$ and $\xi < \frac{m}{n-1}$, then
  $\delf{m+1}{n}{\xi} - \delf{m}{n}{\xi} = n $.
\end{itemize}
\end{prop}
\begin{proof}
  (i) In the presentation of $\tau_\nyh{1}{\xi}^{ \nyh{m}{n}} \myv{m}{1}$
  in
  \eqref{eq:xiterminal},
  the coefficient $1-t$ of $\xi$ of a summand is not positive.

  (ii) In the presentation of $\tau_\nyh{1}{\xi}^{ \nyh{m}{n}} \myv{1}{n}$
  in
  \eqref{eq:xiterminal},
  the coefficient $n-t$ of $\xi$ of a summand is not negative.

  (iii) It is a direct consequence from (i) and (ii).

  (iv) Without loss of generality, assume $\xi \leq \xi'$.
  Then
  $$
  \delf{m}{n}{\xi} - \delf{m}{n}{\xi'}
  = ( \tau_\nyh{1}{\xi}^{ \nyh{m}{n}} \myv{m}{1}
   - \tau_\nyh{1}{\xi'}^{ \nyh{m}{n}} \myv{m}{1}
   ) +
   ( \tau_\nyh{1}{\xi'}^{ \nyh{m}{n}} \myv{1}{n}
    - \tau_\nyh{1}{\xi}^{ \nyh{m}{n}} \myv{1}{n}
   ),
  $$
    where the value of each of the parentheses are nonnegative
    by (i) and (ii). The assumption forces these values to be $0$.

(v)
  The equation $m + \xi (1-t) = 1 + \xi (n-t)$
  has the solution $\xi =\frac{m-1}{n-1}$ (independent of $t$) when $n \geq 2$.
  Thus, $\xi = \frac{m-1}{n-1}$ implies that
  $\tau_\nyh{1}{\xi}^{ \nyh{m}{n}} \myv{m}{1} =
  \tau_\nyh{1}{\xi}^{ \nyh{m}{n}} \myv{1}{n}$.

  (vi) By the change of variables $s' = m - s$ and $t' = t - 1$, we obtain
$$  
  \tau_\nyh{1}{\xi}^{ \nyh{m}{n}} \myv{m}{1}
   = \sum_{s'=0}^{m-1} \sum_{t'=0}^{n-1} \myChi{m-s' \leq m + \xi (-t')}
 = \sum_{s'=0}^{m-1} \sum_{t'=0}^{n-1} \myChi{t' \xi  \leq s'},
$$
and by $s'' = s - 1$ and $t'' = n - t$, we have
$$  
\tau_\nyh{1}{\xi}^{ \nyh{m}{n}} \myv{1}{n}
  =   \sum_{s''=0}^{m-1} \sum_{t''=0}^{n-1} \myChi{s'' + 1  \leq 1 + \xi t''}
  =   \sum_{s''=0}^{m-1} \sum_{t''=0}^{n-1} \myChi{t'' \xi  \geq s'' }.
$$
  Hence we have
\begin{equation}
\delf{m}{n}{\xi} = \left(\sum_{s=0}^{m-1} \sum_{t=0}^{n-1} \myChi{t \xi \leq s}\right)
- \left(\sum_{s=0}^{m-1} \sum_{t=0}^{n-1} \myChi{t \xi \geq s}\right),
\label{eq:deltabalance0}
\end{equation}
which is rewritten as
\begin{equation}
\delf{m}{n}{\xi}  =
\left(\sum_{s=0}^{m-1} \sum_{t=0}^{n} \myChi{t \xi \leq s}
-
\sum_{s=0}^{m-1} \sum_{t=0}^{n} \myChi{t \xi \geq s}\right) 
-
\left(\sum_{s=0}^{m-1}  \myChi{n \xi \leq s}
-
\sum_{s=0}^{m-1}  \myChi{n \xi \geq s} \right). \label{eq:deltabalance}
\end{equation}
For the special case $\xi = \frac{m-1}{n}$, the value in the first
parentheses of \eqref{eq:deltabalance} is zero, since
\begin{equation}
\sum_{s=0}^{m-1} \sum_{t=0}^{n} \myChi{t \cdot \tfrac{m-1}{n} \leq s}
= \left|\left\{\myv{s}{t} \in (\{0\} \cup [m-1]) \times (\{0\} \cup [n]):
t \cdot \tfrac{m-1}{n} \leq s \right\}\right|,
\end{equation}
which is the number of the lattice points in the right triangle
$\triangle \myv{0}{0}\myv{m-1}{n}\myv{m-1}{0}$ (including the border),
whereas
\begin{equation}
\sum_{s=0}^{m-1} \sum_{t=0}^{n} \myChi{t \cdot \tfrac{m-1}{n} \geq s}
= \left|\left\{\myv{s}{t} \in (\{0\} \cup [m-1]) \times (\{0\} \cup [n]):  t \cdot \tfrac{m-1}{n} \geq s \right\}\right| 
\end{equation}
is the number of the lattice points in the right triangle
$\triangle \myv{0}{0}\myv{m-1}{n}\myv{0}{n}$ (including the border),
which is the image of the former triangle with the
congruent transformation $\myv{s}{t} \mapsto \myv{m-1-s}{n-t}$ 
by which the lattice points over the segment
$\myv{0}{0}\myv{m-1}{n}$ are stable as a set. Therefore,
it follows that
\begin{multline}
  \delf{m}{n}{\frac{m-1}{n}} =
  -
\left(\sum_{s=0}^{m-1}  \myChi{n  \cdot \tfrac{m-1}{n} \leq s}
-
\sum_{s=0}^{m-1}  \myChi{n  \cdot \tfrac{m-1}{n} \geq s} \right) \\
=   -
\left(\sum_{s=0}^{m-1}  \myChi{m-1 \leq s}
-
\sum_{s=0}^{m-1}  \myChi{m-1 \geq s} \right)
= - (1 - m) = m - 1. \label{eq:lastterm}
\end{multline}
The first equality \eqref{eq:vifirst}
has been shown. For the second 
equality \eqref{eq:visecond}, let us evaluate the effect of
the modification from $\xi = \frac{m-1}{n}$ to $\xi - \omicron$ on the first and the second
parentheses of \eqref{eq:deltabalance} respectively.
In the first parentheses, the triangle
$\triangle \myv{0}{0}\myv{m-1}{n}\myv{0}{n}$ loses lattice points
over the segment $\myv{0}{0}\myv{m-1}{n}$ other than the origin,
while $\triangle \myv{0}{0}\myv{m-1}{n}\myv{m-1}{0}$ keeps them.
The number of the lattice points over the segment
$\myv{0}{0}\myv{m-1}{n}$ other than the origin is $\gcd(m-1,n)$.
Thus the effect on the first parentheses of \eqref{eq:deltabalance}
is $+\gcd(m-1,n)$.
The second parentheses 
$\left(\sum_{s=0}^{m-1}  \myChi{m-1 \leq s} - \sum_{s=0}^{m-1}  \myChi{m-1 \geq s} \right)$ changes to
$$
\sum_{s=0}^{m-1}  \myChi{m-1-\omicron' \leq s} - \sum_{s=0}^{m-1}  \myChi{m-1-\omicron' \geq s}, \quad (\omicron' := n \omicron)
$$
where the first sum remain unchanged
however the second sum decreases by one due to the term for $s = m-1$.
Thus, the effect on the second parentheses is $+1$, thus
the effect on RHS of \eqref{eq:lastterm} is $-1$.
In total, the increment of the total value is $\gcd(m-1,n)-1.$

(vii) 
Substituting $\myv{m}{n}$ with $\myv{n}{m}$ in
\eqref{eq:vifirst} yields
$
\delf{n}{m}{\frac{n-1}{m}} = n - 1.
$
Let $\xi = \frac{m}{n-1}.$ By Proposition \ref{prop:delfinv},
it follows that
\begin{equation}
  \delf{m}{n}{\frac{m}{n-1}} = \delf{m}{n}{\xi} =
  - \delf{n}{m}{\xi^{-1}} = -\delf{n}{m}{\frac{n-1}{m}} = 1 - n,
\end{equation}
thus we have \eqref{eq:viifirst}.  To show \eqref{eq:viisecond},
given $o > 0$,
let $p > 0$ be the unique such positive number that
$(\xi + \omicron)^{-1} = \xi^{-1} - p,$ namely,
$p = \frac{\omicron}{\xi (\xi+\omicron)}$.
It holds that $p \rightarrow 0$ as $\omicron \rightarrow 0$.
   Then, by a similar argument
   as above using Proposition \ref{prop:delfinv} and
   \eqref{eq:visecond}, we have
\begin{equation}
  \delf{m}{n}{\frac{m}{n-1} + \omicron}
  = \delf{m}{n}{\xi + \omicron}
  = -\delf{n}{m}{(\xi + \omicron)^{-1}}
  = -\delf{n}{m}{\xi^{-1} - p} 
  = 1-n - (\gcd(n-1,m) - 1).
\end{equation}

(viii) From \eqref{eq:deltabalance0},
it follows that
\begin{equation}
  \delf{m+1}{n}{\xi} - \delf{m}{n}{\xi}
  = \left( \sum_{t=0}^{n-1} \myChi{t \xi \leq m} \right) 
-
 \left(\sum_{t=0}^{n-1} \myChi{t \xi \geq m} \right)
\end{equation}
which is $n - 0$ by the assumption $\xi < \frac{m}{n-1}.$
\end{proof}

The first key lemma is as follows:
\begin{lemma}%
  \label{lemma:phiupmapzero}
For an integer $k \geq 2,$ it
  holds that
$
\phidon{k+1} = \upmapzerodon{k+1}|_{\TYngdon{k+1}}.
$
\end{lemma}
\begin{proof}
Let $k \geq 2$.
To show $\phidon{k+1}(y) = \upmapzerodon{k+1}(y)$
  for an arbitrary $y \in \TYngdon{k+1},$ let
  $z := \yastdon{k}^{-1}(y)$ and let us prove the statement
  $\yastdon{k-1}\circ\trimapdon{k-1}(z) = \upmapzerodon{k+1}\circ\yastdon{k}(z),$
  which implies
  (by the bijectivity of $\yastdon{k}$)
  that
  $
\phidon{k+1}(y) = \yastdon{k-1}\circ\trimapdon{k-1}\circ\yastdon{k}^{-1}(y)
  = \yastdon{k-1}\circ\trimapdon{k-1}(z) = \upmapzerodon{k+1}\circ\yastdon{k}(z)
  = \upmapzerodon{k+1}(y).$

There exist the unique
$m, n \geq 0$ such that $z \in \Inter{m}{n}$
and $m+n=k-1$. %
Then, there 
is the unique $(a,b) \in \Gint{m}{n}$ such that
$z = ((a,b),\myv{m}{n})$. By Definitions \ref{defi:trimap}, \ref{defi:yastdon}
and \ref{defi:yast}, we
compute $\yastdon{k-1}\circ\trimapdon{k-1}(z)$
and $\upmapzerodon{k+1}\circ\yastdon{k}(z)$ to compare them.

First consider the case $mn = 0$. When $m \geq 1$ and $n = 0,$
only $z = ((0,\infty),\myv{m}{0})$ is possible and
we have $\trimapdon{k-1}(z) = ((0,\infty),\myv{m-1}{0}) \in \Inter{m-1}{0}$ from Definition 
\ref{defi:trimap}. Then by taking $\xi = 1 \in (0, \infty)$, we have
$\yastdon{k-1}\circ\trimapdon{k-1}(z) = \yast{m-1}{0}(((0,\infty),\myv{m-1}{0})) = (\tau_{\nyh{1}{1}}^{\nyh{m}{1}}\myv{m}{1}, \tau_{\nyh{1}{1}}^{\nyh{m}{1}}\myv{1}{1}, \myv{m}{1})$.
By \eqref{eq:xiterminal} with $\xi = 1$ and $n=1$,
it follows that $\tau_\nyh{1}{1}^{ \nyh{m}{1}} \myv{m}{1} = m$
and $\tau_\nyh{1}{1}^{ \nyh{m}{1}} \myv{1}{1} = 1$. Thus,
$\yastdon{k-1}\circ\trimapdon{k-1}(z) = $
$(m, 1, \myv{m}{1}).$
On the other hand, by a similar computation,
$\yastdon{k}(z) =
(\tau_\nyh{1}{1}^{ \nyh{m+1}{1}} \myv{m+1}{1},
\tau_\nyh{1}{1}^{ \nyh{m+1}{1}} \myv{1}{1}, \myv{m+1}{1}
  ) =
(m+1, 1, \myv{m+1}{1})$. By $m+1 > 1$ and
Definition \ref{defi:upmap},
$\upmapzerodon{k+1}\circ\yastdon{k}(z) = 
\upmapzero{m+1}{1}((m+1, 1, \myv{m+1}{1})) = (m+1 - 1, 1, \myv{m}{1}) = (m,1,\myv{m}{1}).$ Therefore, $\yastdon{k-1}\circ\trimapdon{k-1}(z) =
\upmapzerodon{k+1}\circ\yastdon{k}(z)
$ for the case $n = 0.$  The case $m = 0 \text{~and~} n \geq 1$ reduces 
to  the case $m \geq 1 \wedge n = 0$  by taking the transposition
for generalized Farey intervals (Definition \ref{defi:tpfarey}).
  Indeed, given $v = ((0, \infty), \myv{0}{n})$,
  the application of the above argument to $\tp{v}$ shows that
  $\yastdon{k-1}\circ\trimapdon{k-1}(\tp{v}) = \upmapzerodon{k+1}\circ\yastdon{k}(\tp{v})$. By Propositions
  \ref{prop:tptrimap},
  \ref{prop:tpyast} and \ref{prop:tpupmap}, we have
  $\tp{(\yastdon{k-1}\circ\trimapdon{k-1}(v))} =
  \tp{(\upmapzerodon{k+1}\circ\yastdon{k}(v))}$.

Next let us consider the case $mn \geq 1$. Consider the subcase
$b \leq \frac{m}{n}$. Then 
by \eqref{eq:trimapdef} of Definition \ref{defi:trimap}, it holds that
$\trimapdon{k-1}(z) = ((a,c),\myv{m-1}{n}) \in \Inter{m-1}{n}$
where $c \geq b$ holds. Thus, if we take $\xi \in (a,b)$ then it also
satisfies $\xi \in (a,c)$ and we have
\begin{equation}
\yastdon{k-1}\circ\trimapdon{k-1}(z) = \yast{m-1}{n}(((a,c), \myv{m-1}{n}))
= (\tau_\nyh{1}{\xi}^{\nyh{m}{n+1}}\myv{m}{1},
\tau_\nyh{1}{\xi}^{\nyh{m}{n+1}}\myv{1}{n+1}, \myv{m}{n+1})
\label{eq:counterclockwise}
\end{equation}
by Definition \ref{defi:yast} and Fact \ref{fact:yas}.
On the other hand, by using the $\xi$ taken above,
$$
\yastdon{k}(z) = (\tau_\nyh{1}{\xi}^{\nyh{m+1}{n+1}}\myv{m+1}{1},
\tau_\nyh{1}{\xi}^{\nyh{m+1}{n+1}}\myv{1}{n+1}, \myv{m+1}{n+1}).
$$
The condition $b \leq \frac{m}{n}$ implies that $\xi < \frac{(m+1)-1}{(n+1)-1}$.
Then by Definition \ref{defi:delf} and
Proposition \ref{prop:fund} (iii),(v),
we have
$$
\tau_\nyh{1}{\xi}^{\nyh{m+1}{n+1}}\myv{m+1}{1} - \tau_\nyh{1}{\xi}^{\nyh{m+1}{n+1}}\myv{1}{n+1} 
= \delf{m+1}{n+1}{\xi}
\geq \delf{m+1}{n+1}{\frac{(m+1)-1}{(n+1)-1}}
= 0,
$$
where LHS is not zero since $\tau_\nyh{1}{\xi}^{\nyh{m+1}{n+1}}$ is
a Young ranking table which is an injective table.
By this inequality and 
Definition \ref{defi:upmap}, we have
\begin{multline}
\upmapzerodon{k+1}\circ\yastdon{k}(z) = 
\upmapzero{m+1}{n+1}((\tau_\nyh{1}{\xi}^{\nyh{m+1}{n+1}}\myv{m+1}{1},
\tau_\nyh{1}{\xi}^{\nyh{m+1}{n+1}}\myv{1}{n+1}, \myv{m+1}{n+1}))  \\ =
(\tau_\nyh{1}{\xi}^{\nyh{m+1}{n+1}}\myv{m+1}{1}-(n+1),
\tau_\nyh{1}{\xi}^{\nyh{m+1}{n+1}}\myv{1}{n+1}, \myv{m}{n+1}).
\label{eq:clockwise}
\end{multline}

The two terminal pairs  \eqref{eq:counterclockwise} and
\eqref{eq:clockwise} have
the common index $\myv{m}{n+1}.$
So, let us compare their diffrences.
For \eqref{eq:counterclockwise}, the difference is 
$
\tau_\nyh{1}{\xi}^{\nyh{m}{n+1}}\myv{m}{1} \allowbreak - 
\tau_\nyh{1}{\xi}^{\nyh{m}{n+1}}\myv{1}{n+1} = \delf{m}{n+1}{\xi}.$
For \eqref{eq:clockwise}, the difference is 
$\tau_\nyh{1}{\xi}^{\nyh{m+1}{n+1}}\myv{m+1}{1} \allowbreak -(n+1)$
$-\tau_\nyh{1}{\xi}^{\nyh{m+1}{n+1}}\myv{1}{n+1} =
\delf{m+1}{n+1}{\xi}-(n+1).
$
By Proposition \ref{prop:fund} (viii) and $\xi < \frac{m}{n} = \frac{m}{(n+1)-1}$,
$\delf{m+1}{n+1}{\xi} - \delf{m}{n+1}{\xi} = n+1.$
Therefore the terminal pairs
$\yastdon{k-1}\circ\trimapdon{k-1}(z)$ and
$\upmapzerodon{k+1}\circ\yastdon{k}(z)$ have the same difference.
Since they are in $\Lsetitimonozero{m}{n+1}$, by Proposition \ref{prop:diffitimonozero}, we conclude that $\yastdon{k-1}\circ\trimapdon{k-1}(z) = \upmapzerodon{k+1}\circ\yastdon{k}(z)$ if $z=((a,b),\myv{m}{n})\in \Inter{m}{n}$
satisfies $b \leq \frac{m}{n}.$ In \eqref{eq:trimapdef} of
Definition \ref{defi:trimap}, there is the other subcase
$a \geq \frac{m}{n}$
to be considered. 
However, for such $v = ((a,b), \myv{m}{n}),$
taking the transposition $((a',b'), \myv{m'}{n'}) := \tp{v} =
((b^{-1}, a^{-1}), \myv{n}{m})$, $b' = a^{-1}$ satisfies
$b' = a^{-1} \leq (\frac{m}{n})^{-1} = \frac{n}{m} = \frac{m'}{n'}$
and the above argument for the subcase
$b \leq \frac{m}{n}$ is applicable
to obtain $\yastdon{k-1}\circ\trimapdon{k-1}(\tp{v}) = \upmapzerodon{k+1}\circ\yastdon{k}(\tp{v})$. By Propositions
  \ref{prop:tptrimap},
  \ref{prop:tpyast} and \ref{prop:tpupmap}, we have
  $\tp{(\yastdon{k-1}\circ\trimapdon{k-1}(v))} =
  \tp{(\upmapzerodon{k+1}\circ\yastdon{k}(v))}$.

Thus, we have $\yastdon{k-1}\circ\trimapdon{k-1}(z)
= \upmapzerodon{k+1}\circ\yastdon{k}(z)$ for $z \in \Interdon{k-1}.$
\end{proof}
We also need another lemma:
\begin{lemma}\label{lem:branch}
      Let $k \geq 2$ be an integer.
      If a vertex $x \in \TYngdon{k} \subset \Lsetitimonozerodon{k}$
      has a horizontal (resp. vertical)
      child $y \in \Lsetitimonozerodon{k+1}$ in the
      tree $\mathbb{T}_N^{\Lsetitimonozerodon{},\upmapzerodon{}},$
      then $y$ satisfies $y \in \TYngdon{k+1}$ and
      is also the unique horizontal (resp. vertical) child
      of $x$ in the tree $\mathbb{T}_N^{\TYngdon{},\phidon{}}$.
\end{lemma}
\begin{proof}
For given $k \geq 2,$ take $x \in \TYngdon{k}$ arbitrarily.
Let $\myv{m}{n}$
be the unique pair such that $m + n = k, m, n \geq 1$ and $x \in \TYng{m}{n}$.
Let $((a,b),\myv{m-1}{n-1}) = \yastdon{k-1}^{-1}(x) \in \Inter{m-1}{n-1}.$
Once the statement
\begin{equation}
  (\upmapzerodon{k+1}^{-1}(\{x\}) \cap \Lsetitimonozero{m+1}{n} \neq \emptyset)
  \quad\Rightarrow\quad
  (\phidon{k+1}^{-1}(\{x\}) \cap \TYng{m+1}{n} \neq \emptyset)
\label{eq:existential}  
\end{equation}
is proved, the assertion on the horizontal child
follows readily: Let $y$ be in
$\upmapzerodon{k+1}^{-1}(\{x\}) \cap \Lsetitimonozero{m+1}{n}$
(i.e. the unique horizontal child of $x$ in the tree
$\mathbb{T}_N^{\Lsetitimonozerodon{},\upmapzerodon{}}$ as in the note after Definition \ref{defi:upmaptree})
and 
let $z$ be in
$\phidon{k+1}^{-1}(\{x\}) \cap \TYng{m+1}{n}$
(i.e. the unique horizontal child of $x$ in the tree
$\mathbb{T}_N^{\TYngdon,\phidon{}}$ by Proposition \ref{prop:isom13} and the note after Definition \ref{defi:trimaptree}),
whose existence is implied by the existence of $y$ and
\eqref{eq:existential}. By Fact \ref{fact:Lincl} and
Lemma \ref{lemma:phiupmapzero},
It follows that $\upmapzerodon{k+1}(z)
= \phidon{k+1}(z).$ Since $z \in \phidon{k+1}^{-1}(\{x\}),$
we have $\phidon{k+1}(z) = x$ and thus $\upmapzerodon{k+1}(z)
= x$. This means,
with the assumption
$z \in \TYng{m+1}{n} \subset \Lsetitimonozero{m+1}{n}$, that
$z$ is a horizontal child of $x$ in the tree
$\mathbb{T}_N^{\Lsetitimonozerodon{},\upmapzerodon{}}$. However, 
$y$ is the unique horizontal
     child of $x$ in $\mathbb{T}_N^{\Lsetitimonozerodon{},\upmapzerodon{}}$.
Thus, we have the equality $y = z$ and 
$y \in \phidon{k+1}^{-1}(\{x\}) \cap \TYng{m+1}{n}$, the assertion
on the horizontal child.
In the sequel, we shall prove the statement \eqref{eq:existential}.

First, consider the case $n = 1.$ In this case, $x = (m,1,\myv{m}{1})$
is the unique vertex
and $\yastdon{k-1}^{-1}(x)$ is
  $((0,\infty),\myv{m-1}{0})$, as we have
  seen in the computation of $
  \yastdon{k-1}\circ\trimapdon{k-1}(((0, \infty), \myv{m}{0})) = $
  \\
  $\yast{m-1}{0}(((0,\infty),\myv{m-1}{0}))$
in 
the proof of Lemma \ref{lemma:phiupmapzero}.
By Fact \ref{fact:invtrimap},
the horizontal child of $\yastdon{k-1}^{-1}(x)$ is
$((0,\infty), \myv{1}{0})$ if $m=1$; $((0,\infty), \myv{m}{0})$ otherwise.
In both cases, $\yastdon{k}$ maps this horizontal child to
$z := (\tau_{\nyh{1}{\xi}}^{\nyh{m+1}{1}}\myv{m+1}{1},
\tau_{\nyh{1}{\xi}}^{\nyh{m+1}{1}}\myv{1}{1}, \myv{m+1}{1})$
where $\xi$ can be taken in the open interval $(0,\infty)$.
This $z$ is
  in
  $\yastdon{k}\circ\trimapdon{k-1}^{-1}\circ\yastdon{k-1}^{-1}(\{x\}) \cap \TYng{m+1}{1} = \phidon{k+1}^{-1}(\{x\}) \cap \TYng{m+1}{1}$
  and the desired existence has been established.

Next assume that $n \geq 2.$
With an arbitrary $\xi \in (a, b)$, $x$ is written as
$x = (\tau_{\nyh{1}{\xi}}^{\nyh{m}{n}}\myv{m}{1},
  \tau_{\nyh{1}{\xi}}^{\nyh{m}{n}}\myv{1}{n}, \myv{m}{n})$, however
    we take $\xi = a + \varepsilon$ where $\varepsilon > 0$
    is small enough that there is no element of $\Garey{m}{n}$
    (which includes any of $\Garey{m-1}{n}, \Garey{m}{n-1}, \Garey{m-1}{n-1})$
    in the half-open interval $(a, a+\varepsilon]$.
  Suppose that $x$ has the horizontal child in the tree
  $\mathbb{T}_N^{\Lsetitimonozerodon{},\upmapzerodon{} }$.
  By Definition \ref{defi:delf} and 
  Fact \ref{fact:invupmap}, this means that
  $\delf{m}{n}{\xi} = 
  \tau_{\nyh{1}{\xi}}^{\nyh{m}{n}}\myv{m}{1} - \tau_{\nyh{1}{\xi}}^{\nyh{m}{n}}\myv{1}{n}
  > -n.$ Since LHS is an integer, we have $\delf{m}{n}{\xi} \geq -n + 1$.
  Now we claim that $a > \frac{m}{n-1}$ is impossible.
  To obtain contradiction, suppose that $a > \frac{m}{n-1}$.
  Since $a \in \Garey{m-1}{n-1}$ is a boundary of a
  generalized Farey interval,
  for arbitrarily small $\eta > 0$ (which can be taken smaller than
  the minimum length of the elements of $\Gint{m-1}{n-1}$)
  $\tau_{\nyh{1}{a-\eta}}^{\nyh{m}{n}}$ and
    $\tau_{\nyh{1}{\xi}}^{\nyh{m}{n}}$ must be different because of
  the bijectivity of $\yas{m-1}{n-1}$.
  Then by Proposition \ref{prop:fund} (iv), we have
  $ \delf{m}{n}{a-\eta} \neq \delf{m}{n}{\xi}$.
  By Proposition \ref{prop:fund} (iii) and $\frac{m}{n-1} \leq a - \eta <
  a < \xi$,
  this means that $\delf{m}{n}{\frac{m}{n-1}} \geq
  \delf{m}{n}{a - \eta} > \delf{m}{n}{\xi}$,
  where $\delf{m}{n}{\frac{m}{n-1}} = -n + 1$ by the first equality
  \eqref{eq:viifirst}
  of Proposition \ref{prop:fund} (vii). However, this contradicts
  to $\delf{m}{n}{\xi} \geq -n + 1.$ Therefore
  $((a,b),\myv{m-1}{n-1}) \in \Inter{m-1}{n-1}$ should satisfy
  $a \leq \frac{m}{n-1}.$ Further, we claim that $a \neq \frac{m}{n-1}$.
  Suppose that $a = \frac{m}{n-1}$. Since
  $a \in \Garey{m-1}{n-1}$ and $m > m-1$, this happens only if
  $\gcd(m,n-1) > 1$. By  $a < \xi$,
  Proposition \ref{prop:fund} (iii) and the second equality
  \eqref{eq:viisecond}
  of (vii),
  where $\omicron$ is such positive quantity that is smaller than
  any predetermined positive number like $\varepsilon = \xi - a$,
  we have $\delf{m}{n}{a+\omicron} = \delf{m}{n}{\frac{m}{n-1}+\omicron}
  = -n+1-\gcd(m,n-1)+1   \geq
  \delf{m}{n}{\xi} \geq -n + 1$, which implies $1 - \gcd(m,n-1) \geq 0$,
  contradiction.

   Thus $a < \frac{m}{n-1}$ holds. By Fact \ref{fact:invtrimap},
   $((a,b),\myv{m-1}{n-1}) \in \Inter{m-1}{n-1}$ has a horizontal
   child of the form $((a,*),\myv{m}{n-1})$ where $*$ may be
   smaller than $b$. However, we took $\xi = a + \varepsilon$ with
   sufficiently small $\varepsilon > 0$, so $\xi$ remains in $(a,*)$.
   Therefore, the image of this horizontal child by the bijection
   $\yast{m}{n-1}$ can be written as
   $(\tau_{\nyh{1}{\xi}}^{\nyh{m+1}{n}}\myv{m+1}{1},
   \tau_{\nyh{1}{\xi}}^{\nyh{m+1}{n}}\myv{1}{n}, \myv{m+1}{n})$.
   Let $z := (\tau_{\nyh{1}{\xi}}^{\nyh{m+1}{n}}\myv{m+1}{1},
   \tau_{\nyh{1}{\xi}}^{\nyh{m+1}{n}}\myv{1}{n}, \myv{m+1}{n})$
   which belongs to $\TYng{m+1}{n} \subset\TYngdon{k+1}$
   by the definition of
   $\yast{m}{n-1}$.
   By %
   Definition \ref{defi:compat},
   $\phidon{k+1}(z)
      = \phidon{k+1}\circ\yastdon{k}(((a,*),\myv{m}{n-1}))
   = \yastdon{k-1}\circ\trimapdon{k-1}(((a,*),\myv{m}{n-1}))
   = \yastdon{k-1}(((a,b),\myv{m-1}{n-1})) = x$.
   Thus, $z$
   is a horizontal child of $x$ in $\mathbb{T}_N^{\TYngdon{},\phidon{}}$,
   the witness of the desired non-emptiness of $\phidon{k+1}^{-1}(\{x\}) \cap \TYng{m+1}{n}$.

   For the vertical child of $x$ by $\upmapzerodon{k+1},$
   one may repeat a similar argument by using Proposition \ref{prop:fund} (vi),
   however, taking the transposition $\tp{x}$ for terminal pairs
   reduces the problem to
   the horizontal child case.
   Namely, we have that
     if $\tp{x}$ has a horizontal child in
     $\mathbb{T}_N^{\Lsetitimonozerodon{},\upmapzerodon{}}$,
     then it is also a horizontal child of $\tp{x}$
     in $\mathbb{T}_N^{\TYngdon{},\phidon{}}$. Thus,
     if $x$ has a vertical child in
     $\mathbb{T}_N^{\Lsetitimonozerodon{},\upmapzerodon{}}$,
     then it is also a vertical child of $x$
     in $\mathbb{T}_N^{\TYngdon{},\phidon{}}$. 
     This completes the proof.
\end{proof}
\subsection{Proof of Theorem \ref{theo:main} }
The proof of Theorem \ref{theo:main} is presented below.
\begin{proof}
By Lemma \ref{lem:branch}, the equalities between the sets
    $\Lsetitimonozerodon{k}$ and $\TYngdon{k}$
    of the level $k-2$ vertices
    for $k = 2, 3, \ldots, N + 1$
    are proved inductively: For $k = 2,$
    it holds that $\TYngdon{2} = \Lsetitimonozerodon{2} = \{ (1,1,\myv{1}{1})\}$.
    Suppose that $k < N+1$ and that $\TYngdon{k} = \Lsetitimonozerodon{k}$.
    By the definitions of the inverse images
    of the maps
    $\phidon{k+1}:\TYngdon{k+1} \to \TYngdon{k}$ and
    $\upmapzerodon{k+1}: \Lsetitimonozerodon{k+1} \to \Lsetitimonozerodon{k}$,
    we have $\phidon{k+1}^{-1}(\TYngdon{k}) = \TYngdon{k+1}$ and
    $\upmapzerodon{k+1}^{-1}(\Lsetitimonozerodon{k}) = \Lsetitimonozerodon{k+1}.$
    Since all the (potential) children of a vertex are
    horizontal and vertical children of it, Lemma \ref{lem:branch}
    implies that $\upmapzerodon{k+1}^{-1}(\TYngdon{k}) \subset
    \phidon{k+1}^{-1}(\TYngdon{k}) = \TYngdon{k+1}.$
    By the inductive assumption $ \TYngdon{k} = \Lsetitimonozerodon{k},$
    we have $\Lsetitimonozerodon{k+1} =
    \upmapzerodon{k+1}^{-1}(\Lsetitimonozerodon{k})
    = \upmapzerodon{k+1}^{-1}(\TYngdon{k}) \subset \TYngdon{k+1}.$
    This inclusion, together with $\TYngdon{k+1}
    \subset \Lsetitimonozerodon{k+1}$ stated in
    Definition \ref{defi:TYng},
    shows that the equality
    $\Lsetitimonozerodon{k+1} = \TYngdon{k+1}$ holds.
    This completes the induction.

    Thus, the sets of vertices of
    $\mathbb{T}_N^{\Lsetitimonozerodon{},\upmapzerodon{}},$ and
    $\mathbb{T}_N^{\TYngdon{},\phidon{}}$ are the same.
    Lemma \ref{lemma:phiupmapzero} states
    that the inter-level map $\phidon{k+1}$ is just
    a restriction of $\upmapzerodon{k+1}$ to $\TYngdon{k+1}$
    which is in fact $\Lsetitimonozerodon{k+1},$
    hence we have $\phidon{k+1} = \upmapzerodon{k+1}$.
    Thus, the adjacency in $\mathbb{T}_N^{\Lsetitimonozerodon{},\upmapzerodon{}}$
    and the adjacency in $\mathbb{T}_N^{\TYngdon{},\phidon{}}$ are equivalent.
\end{proof}

\subsection{An implication}\label{sect:impl}

Recall the note on Fact \ref{fact:Litimonorecur} 
made in Section \ref{sss:diffeqtype}, just before
the Definition  \ref{defi:termdiff} of $\Lsetitimonozero{m}{n}$,
that 
given the terminal pair $f\myv{m}{1}$ and $f\myv{1}{n}$
of an injective L-shape 
$f \in \Lsetitimono{m}{n}$ of difference equation type,
one can easily determine all the terms of $f$.
Indeed, given 
$x = (f\myv{m}{1}, f\myv{1}{n}, \myv{m}{n} ) \in \Lsetitimonozero{m}{n},$
either of
$(f\myv{m-1}{1}, f\myv{1}{n}, \myv{m-1}{n} ) \in \Lsetitimonozero{m-1}{n}$
or
$(f\myv{m}{1}, f\myv{1}{n-1}, \myv{m}{n-1} ) \in \Lsetitimonozero{m}{n-1}$
is obtained as $\upmapzerodon{m+n}(x)$, depending on the truth value
$\myChi{f\myv{m}{1} \leq f\myv{1}{n}}$. Repeating the application of
$\upmapzerodon{}$ until we reach to $(1,1,\myv{1}{1})$, the sequence
$x, \upmapzerodon{m+n}(x), \upmapzerodon{m+n-1}\circ\upmapzerodon{m+n}(x),
\cdots,
\upmapzerodon{3}\circ \cdots \circ \upmapzerodon{m+n}(x)$
contains all the terms
$(f\myv{i}{j} : \myv{i}{j} \in ([m]\times \{1\}) \cup (\{1\}\times [n]) )$
of the L-shape $f$.
In terms of the tree $\mathbb{T}_N^{\Lsetitimonozerodon{},\upmapzerodon{}}$, the ascending path from the vertex 
$x = (f\myv{m}{1}, f\myv{1}{n}, \myv{m}{n} ) \in \Lsetitimonozero{m}{n}$
to the root $(1,1,\myv{1}{1})$ gives the ``decompression procedure''
from the terminal pair $x$ to the L-shape $f$.
Thus, considering a little redundant presentation of the tree
$\mathbb{T}_N^{\Lsetitimonozerodon{},\upmapzerodon{}}$
where each of vertices $(f\myv{m}{1}, f\myv{1}{n}, \myv{m}{n} )$
is replaced with its whole L-shape
$(f\myv{i}{j} : \myv{i}{j} \in ([m]\times \{1\}) \cup (\{1\}\times [n]) )$
for all $\myv{m}{n}$ such that $m+n \leq N+2$,
Theorem \ref{theo:main} implies the following assertion on
the injective L-shapes $\Lsetitimono{m}{n}$
of difference equation type.

\begin{coro}\label{coro:Lset}
  Let $m$ and $n$ be positive integers. It holds that
  \begin{equation}
    \LYng{m}{n} = \Lsetitimono{m}{n},
  \end{equation}    
  where the common cardinality is $|\Gint{m-1}{n-1}| = |\Garey{m-1}{n-1}| - 1.$ 
\end{coro}  
\begin{proof}
  Theorem \ref{theo:main} implies in particular that
  $\TYng{m}{n} = \Lsetitimonozero{m}{n}.$ Applying the decompression
  procedure (i.e., ascending $\mathbb{T}_N^{\Lsetitimonozerodon{},\upmapzerodon{}}$
  from a vertex to the root then replacing the vertex with the resulting
  L-shape obtained from the path)
  on both sides, we have $\LYng{m}{n} = \Lsetitimono{m}{n}$.
  By %
  Proposition \ref{prop:isom13}
  their cardinalities are
  $|\Inter{m-1}{n-1}| = |\Gint{m-1}{n-1}| = |\Garey{m-1}{n-1}|-1$.
\end{proof}  

\begin{rema}\label{rema:farey}
  The set $\Garey{m}{n}$ of Definition \ref{defi:garey} is closely related
  to the following 
  set $\mathcal{F}_{n}^{m}$ of the \textit{generalized Farey fractions}
  defined for $1 \leq m \leq n$
  in Glaisher {\cite[Sect. 6]{Glaisher}}:
  \begin{equation}
    \mathcal{F}_{n}^{m} = \{0\} \cup \left\{\frac{p}{q} \leq 1: p \in [m], q \in [n] \right\}.
  \end{equation}
  Note that it is a subset of the standard Farey fractions
  $\mathcal{F}_{n}$ and that
  $\mathcal{F}_{n}^{n} =  \mathcal{F}_{n}.$
  For the properties and the history
  of the study of $\mathcal{F}_{n}^{m}$, we refer to
  \cite[Section 1]{Matveev}.
For example, a presentation of $|\mathcal{F}_{n}^{m}|$
is given as  \cite[Proposition 1.29]{Matveev}, %
which yields a presentation for $|\Gint{m}{n}|$ readily.
\end{rema}

A further implication in terms of Young ranking tables is
given in Appendix \ref{sect:implranking}.

\newpage
\appendix
\section{The three isomorphic trees for $N=4$.}\label{sect:example}
As an example for Theorem \ref{theo:main} and Corollary \ref{coro:isom},
Fig.~\ref{fig:3trees} shows the isomorphic three trees
$\mathbb{T}_N^{\Interdon{}, \trimapdon{}}$,
$\mathbb{T}_N^{\Lsetitimonozerodon{},\upmapzerodon{}}$
and
$\mathbb{T}_N^{\TYngdon{},\phidon{}}$
of height $N=4$
in one picture.
Each of the squares corresponds to a $\Inter{m-1}{n-1}$, or
$\TYng{m}{n} = \Lsetitimonozero{m}{n}$ associated to it
by yet another {\mysur}'s bijection $\yastdon{}$ (Definition \ref{defi:yastdon}), which is
written as in Fig.~\ref{fig:yast}. Even levels of the tree
are shaded while odd levels are unshaded for visual distinction.
Small circles are the vertices
of the tree, and arrows represent the edges defined by the
inter-level surjections. 
The direction of an arrow is from a child
to its parent. For example, a vertex $((\frac{0}{1},\frac{1}{2}), \myv{1}{2})
\in \Inter{1}{2}$ (%
the circle surrounded by a triangle, 
  near the lower-right corner of the unshaded
  square corresponding to $\Inter{1}{2}$)
has a horizontal child
$((\frac{0}{1},\frac{1}{2}),\myv{2}{2})$ and
a vertical child $((\frac{1}{3},\frac{1}{2}),\myv{1}{3})$.
Corresponding to this,
$(4,3,\myv{2}{3}) \in \TYng{2}{3} = \Lsetitimonozero{2}{3}$
has a horizontal child $(7,3,\myv{3}{3})$ where $7 = 4 + 3,$
and a vertical child $(4,5,\myv{2}{4})$ where $5 = 3 + 2$.
On the other hand, the vertex
$((\frac{1}{1},\infty), \myv{1}{2})
\in \Inter{1}{2}$ (%
the circle surrounded by a square,
near the upper-left corner of the unshaded
  square corresponding to $\Inter{1}{2}$)
only has a vertical child $((\frac{1}{1},\infty),\myv{1}{3}) \in \Inter{1}{3}$.
Corresponding to this,
$(2,5,\myv{2}{3}) \in \TYng{2}{3} = \Lsetitimonozero{2}{3}$
has only a vertical child $(2,7,\myv{2}{4})$ where $7 = 5 + 2.$

Theorem \ref{theo:main} and Corollary \ref{coro:isom}
claim that a vertex depicted can be
interpreted as any of
a generalized Farey interval, a terminal pair
of difference equation type and a Young terminal pair and that
the edges are consistent with any of $\trimapdon{}, \upmapzerodon{}$
and $\phidon{}$. Also,
it is visually observed that the transposition
(for generalized Farey intervals/terminal pairs)
works
as an involution of the tree.
\begin{center}
\begin{figure}
\centering
  \rotatebox{0}{\resizebox{\textwidth}{!}{
   \begin{tikzpicture}[scale=4.8]
     \fill [color=yellow!10] (0,0) rectangle (0+1,0+1);
     \fill [color=yellow!10] (2,0) rectangle (2+1,0+1);
     \fill [color=yellow!10] (1,1) rectangle (1+1,1+1);
     \fill [color=yellow!10] (0,2) rectangle (0+1,2+1);
     \fill [color=yellow!10] (4,0) rectangle (4+1,0+1);
     \fill [color=yellow!10] (3,1) rectangle (3+1,1+1);
     \fill [color=yellow!10] (2,2) rectangle (2+1,2+1);
     \fill [color=yellow!10] (1,3) rectangle (1+1,3+1);
     \fill [color=yellow!10] (0,4) rectangle (0+1,4+1);
     \foreach \a in {0,1,2,3,4} {
       \draw (\a + 0.5, 0) node [below] {$m-1=\a$};
     }
     \foreach \a in {0,1,2,3,4} {
       \draw (0,\a + 0.5) node [left] {$n-1=\a$};
     }
     \foreach \bm/\bn/\cm/\cn in {0/0/1/1,0/1/1/2,1/0/2/1,0/2/1/3,1/1/2/2,2/0/3/1,0/3/1/4,1/2/2/3,2/1/3/2,3/0/4/1,0/4/1/5,1/3/2/4,2/2/3/3,3/1/4/2,4/0/5/1} {
      \draw (\bm + 0.05, \bn + 0.05) node [right] {$\Inter{\bm}{\bn}\to\TYng{\cm}{\cn}$};
     }
     \foreach \a in {0,1,2,3,4,5} {
       \draw[ultra thick,color=green!20] (\a,0) -- (0,\a) {};
       \draw  (\a,0) node [right,rotate=-45] {$k = (m-1) + (n-1) = \a, \Interdon{\a}, \TYngdon{\a + 2} $};
     }
\draw[fill=white] (0.5,0.5) circle (0.01);
\draw (0.5,0.5) node [right] {$\left(\left(\frac{0}{1},\infty\right), \binom{0}{0}\right)\mapsto \left(1,1,\binom{1}{1}\right)$};
\draw [-{Latex[length=5mm]},ultra thick,opacity=0.25,color=red] (1.5,0.5) -- (0.5,0.5);
\draw[fill=white] (1.5,0.5) circle (0.01);
\draw (1.5,0.5) node [right] {$\left(\left(\frac{0}{1},\infty\right), \binom{1}{0}\right)\mapsto \left(2,1,\binom{2}{1}\right)$};
\draw [-{Latex[length=5mm]},ultra thick,opacity=0.25,color=red] (2.5,0.5) -- (1.5,0.5);
\draw[fill=white] (2.5,0.5) circle (0.01);
\draw (2.5,0.5) node [right] {$\left(\left(\frac{0}{1},\infty\right), \binom{2}{0}\right)\mapsto \left(3,1,\binom{3}{1}\right)$};
\draw [-{Latex[length=5mm]},ultra thick,opacity=0.25,color=red] (3.5,0.5) -- (2.5,0.5);
\draw[fill=white] (3.5,0.5) circle (0.01);
\draw (3.5,0.5) node [right] {$\left(\left(\frac{0}{1},\infty\right), \binom{3}{0}\right)\mapsto \left(4,1,\binom{4}{1}\right)$};
\draw [-{Latex[length=5mm]},ultra thick,opacity=0.25,color=red] (4.5,0.5) -- (3.5,0.5);
\draw[fill=white] (4.5,0.5) circle (0.01);
\draw (4.5,0.5) node [right] {$\left(\left(\frac{0}{1},\infty\right), \binom{4}{0}\right)\mapsto \left(5,1,\binom{5}{1}\right)$};
\draw [-{Latex[length=5mm]},ultra thick,opacity=0.25,color=red] (3.125,1.875) -- (3.5,0.5);
\draw[color=gray, dashed] (3,1) -- (3.33333,2);
\draw[fill=white] (3.125,1.875) circle (0.01);
\draw (3.125,1.875) node [right] {$\left(\left(\frac{3}{1},\infty\right), \binom{3}{1}\right)\mapsto \left(4,5,\binom{4}{2}\right)$};
\draw [-{Latex[length=5mm]},ultra thick,opacity=0.25,color=red] (2.16667,1.83333) -- (2.5,0.5);
\draw[color=gray, dashed] (2,1) -- (2.5,2);
\draw[fill=white] (2.16667,1.83333) circle (0.01);
\draw (2.16667,1.83333) node [right] {$\left(\left(\frac{2}{1},\infty\right), \binom{2}{1}\right)\mapsto \left(3,4,\binom{3}{2}\right)$};
\draw [-{Latex[length=5mm]},ultra thick,opacity=0.25,color=red] (3.29167,1.70833) -- (2.16667,1.83333);
\draw[color=gray, dashed] (3,1) -- (3.5,2);
\draw[fill=white] (3.29167,1.70833) circle (0.01);
\draw (3.29167,1.70833) node [right] {$\left(\left(\frac{2}{1},\frac{3}{1}\right), \binom{3}{1}\right)\mapsto \left(5,4,\binom{4}{2}\right)$};
\draw [-{Latex[length=5mm]},ultra thick,opacity=0.25,color=red] (2.16667,2.83333) -- (2.16667,1.83333);
\draw[color=gray, dashed] (2,2) -- (2.5,3);
\draw[fill=white] (2.16667,2.83333) circle (0.01);
\draw (2.16667,2.83333) node [right] {$\left(\left(\frac{2}{1},\infty\right), \binom{2}{2}\right)\mapsto \left(3,7,\binom{3}{3}\right)$};
\draw [-{Latex[length=5mm]},ultra thick,opacity=0.25,color=red] (1.25,1.75) -- (1.5,0.5);
\draw[color=gray, dashed] (1,1) -- (2,2);
\draw[fill=white] (1.25,1.75) circle (0.01);
\draw (1.25,1.75) node [right] {$\left(\left(\frac{1}{1},\infty\right), \binom{1}{1}\right)\mapsto \left(2,3,\binom{2}{2}\right)$};
\draw [-{Latex[length=5mm]},ultra thick,opacity=0.25,color=red] (2.41667,1.58333) -- (1.25,1.75);
\draw[color=gray, dashed] (2,1) -- (3,2);
\draw[fill=white] (2.41667,1.58333) circle (0.01);
\draw (2.41667,1.58333) node [right] {$\left(\left(\frac{1}{1},\frac{2}{1}\right), \binom{2}{1}\right)\mapsto \left(4,3,\binom{3}{2}\right)$};
\draw [-{Latex[length=5mm]},ultra thick,opacity=0.25,color=red] (3.41667,1.58333) -- (2.41667,1.58333);
\draw[color=gray, dashed] (3,1) -- (4,2);
\draw[fill=white] (3.41667,1.58333) circle (0.01);
\draw (3.41667,1.58333) node [right] {$\left(\left(\frac{1}{1},\frac{2}{1}\right), \binom{3}{1}\right)\mapsto \left(6,3,\binom{4}{2}\right)$};
\draw [-{Latex[length=5mm]},ultra thick,opacity=0.25,color=red] (2.41667,2.58333) -- (2.41667,1.58333);
\draw[color=gray, dashed] (2,2) -- (3,3);
\draw[fill=white] (2.41667,2.58333) circle (0.01);
\draw (2.41667,2.58333) node [right] {$\left(\left(\frac{1}{1},\frac{2}{1}\right), \binom{2}{2}\right)\mapsto \left(4,6,\binom{3}{3}\right)$};
\draw [-{Latex[length=5mm]},ultra thick,opacity=0.25,color=red] (1.25,2.75) -- (1.25,1.75);
\draw[color=gray, dashed] (1,2) -- (2,3);
\draw[fill=white] (1.25,2.75) circle (0.01);
\node[mark size=4pt] at (1.25,2.75) {\pgfuseplotmark{square}};
\draw (1.25,2.75) node [right] {$\left(\left(\frac{1}{1},\infty\right), \binom{1}{2}\right)\mapsto \left(2,5,\binom{2}{3}\right)$};
\draw [-{Latex[length=5mm]},ultra thick,opacity=0.25,color=red] (1.25,3.75) -- (1.25,2.75);
\draw[color=gray, dashed] (1,3) -- (2,4);
\draw[fill=white] (1.25,3.75) circle (0.01);
\draw (1.25,3.75) node [right] {$\left(\left(\frac{1}{1},\infty\right), \binom{1}{3}\right)\mapsto \left(2,7,\binom{2}{4}\right)$};
\draw [-{Latex[length=5mm]},ultra thick,opacity=0.25,color=red] (0.5,1.5) -- (0.5,0.5);
\draw[fill=white] (0.5,1.5) circle (0.01);
\draw (0.5,1.5) node [right] {$\left(\left(\frac{0}{1},\infty\right), \binom{0}{1}\right)\mapsto \left(1,2,\binom{1}{2}\right)$};
\draw [-{Latex[length=5mm]},ultra thick,opacity=0.25,color=red] (0.5,2.5) -- (0.5,1.5);
\draw[fill=white] (0.5,2.5) circle (0.01);
\draw (0.5,2.5) node [right] {$\left(\left(\frac{0}{1},\infty\right), \binom{0}{2}\right)\mapsto \left(1,3,\binom{1}{3}\right)$};
\draw [-{Latex[length=5mm]},ultra thick,opacity=0.25,color=red] (0.5,3.5) -- (0.5,2.5);
\draw[fill=white] (0.5,3.5) circle (0.01);
\draw (0.5,3.5) node [right] {$\left(\left(\frac{0}{1},\infty\right), \binom{0}{3}\right)\mapsto \left(1,4,\binom{1}{4}\right)$};
\draw [-{Latex[length=5mm]},ultra thick,opacity=0.25,color=red] (0.5,4.5) -- (0.5,3.5);
\draw[fill=white] (0.5,4.5) circle (0.01);
\draw (0.5,4.5) node [right] {$\left(\left(\frac{0}{1},\infty\right), \binom{0}{4}\right)\mapsto \left(1,5,\binom{1}{5}\right)$};
\draw [-{Latex[length=5mm]},ultra thick,opacity=0.25,color=red] (1.875,3.125) -- (0.5,3.5);
\draw[fill=white] (1.875,3.125) circle (0.01);
\draw (1.875,3.125) node [right] {$\left(\left(\frac{0}{1},\frac{1}{3}\right), \binom{1}{3}\right)\mapsto \left(5,4,\binom{2}{4}\right)$};
\draw [-{Latex[length=5mm]},ultra thick,opacity=0.25,color=red] (1.83333,2.16667) -- (0.5,2.5);
\draw[fill=white] (1.83333,2.16667) circle (0.01);
\node[mark size=5pt] at (1.83333,2.16667) {\pgfuseplotmark{triangle}};
\draw (1.83333,2.16667) node [right] {$\left(\left(\frac{0}{1},\frac{1}{2}\right), \binom{1}{2}\right)\mapsto \left(4,3,\binom{2}{3}\right)$};
\draw [-{Latex[length=5mm]},ultra thick,opacity=0.25,color=red] (2.83333,2.16667) -- (1.83333,2.16667);
\draw[fill=white] (2.83333,2.16667) circle (0.01);
\draw (2.83333,2.16667) node [right] {$\left(\left(\frac{0}{1},\frac{1}{2}\right), \binom{2}{2}\right)\mapsto \left(7,3,\binom{3}{3}\right)$};
\draw [-{Latex[length=5mm]},ultra thick,opacity=0.25,color=red] (1.70833,3.29167) -- (1.83333,2.16667);
\draw[color=gray, dashed] (1,3) -- (2,3.33333);
\draw[fill=white] (1.70833,3.29167) circle (0.01);
\draw (1.70833,3.29167) node [right] {$\left(\left(\frac{1}{3},\frac{1}{2}\right), \binom{1}{3}\right)\mapsto \left(4,5,\binom{2}{4}\right)$};
\draw [-{Latex[length=5mm]},ultra thick,opacity=0.25,color=red] (1.75,1.25) -- (0.5,1.5);
\draw[fill=white] (1.75,1.25) circle (0.01);
\draw (1.75,1.25) node [right] {$\left(\left(\frac{0}{1},\frac{1}{1}\right), \binom{1}{1}\right)\mapsto \left(3,2,\binom{2}{2}\right)$};
\draw [-{Latex[length=5mm]},ultra thick,opacity=0.25,color=red] (2.75,1.25) -- (1.75,1.25);
\draw[fill=white] (2.75,1.25) circle (0.01);
\draw (2.75,1.25) node [right] {$\left(\left(\frac{0}{1},\frac{1}{1}\right), \binom{2}{1}\right)\mapsto \left(5,2,\binom{3}{2}\right)$};
\draw [-{Latex[length=5mm]},ultra thick,opacity=0.25,color=red] (3.75,1.25) -- (2.75,1.25);
\draw[fill=white] (3.75,1.25) circle (0.01);
\draw (3.75,1.25) node [right] {$\left(\left(\frac{0}{1},\frac{1}{1}\right), \binom{3}{1}\right)\mapsto \left(7,2,\binom{4}{2}\right)$};
\draw [-{Latex[length=5mm]},ultra thick,opacity=0.25,color=red] (1.58333,2.41667) -- (1.75,1.25);
\draw[color=gray, dashed] (1,2) -- (2,2.5);
\draw[fill=white] (1.58333,2.41667) circle (0.01);
\draw (1.58333,2.41667) node [right] {$\left(\left(\frac{1}{2},\frac{1}{1}\right), \binom{1}{2}\right)\mapsto \left(3,4,\binom{2}{3}\right)$};
\draw [-{Latex[length=5mm]},ultra thick,opacity=0.25,color=red] (2.58333,2.41667) -- (1.58333,2.41667);
\draw[color=gray, dashed] (2,2) -- (3,2.5);
\draw[fill=white] (2.58333,2.41667) circle (0.01);
\draw (2.58333,2.41667) node [right] {$\left(\left(\frac{1}{2},\frac{1}{1}\right), \binom{2}{2}\right)\mapsto \left(6,4,\binom{3}{3}\right)$};
\draw [-{Latex[length=5mm]},ultra thick,opacity=0.25,color=red] (1.58333,3.41667) -- (1.58333,2.41667);
\draw[color=gray, dashed] (1,3) -- (2,3.5);
\draw[fill=white] (1.58333,3.41667) circle (0.01);
\draw (1.58333,3.41667) node [right] {$\left(\left(\frac{1}{2},\frac{1}{1}\right), \binom{1}{3}\right)\mapsto \left(3,6,\binom{2}{4}\right)$};

   \end{tikzpicture}
}}   
\caption{Trees $\mathbb{T}_4^{\Interdon{},\trimapdon{}} \simeq
  \mathbb{T}_4^{\TYngdon{},\phidon{}} = \mathbb{T}_4^{\Lsetitimonozerodon{},\upmapzerodon{}}$} \label{fig:3trees}
\end{figure}
\end{center}
\clearpage

\section{Relation to the study of the inverses of
  S\'os permutations}
For a positive integer $m$ and a real number $\alpha$,
  the permutation that sorts the fractional parts
  of $m$ real numbers $\alpha, 2\alpha, \ldots, m\alpha$ in increasing order is referred to as a {\it {\mysos} permutation\/} of degree $m$. 
This class of permutations
  was introduced in \cite{sos} to prove
  so-called three gaps theorem.
  For a detailed description of {\mysos} permutations,
  we refer to \cite{BKPT}.
  {\mysur} showed that there is a
  bijection to the set of {\mysos} permutations
  from the set of the open intervals formed by consecutive two 
  terms of the $\Th{m}$ Farey sequence, which maps
  the pair formed by the denominators of two successive
  terms in the \Th{m} Farey sequence to the pair
  of the first and last terms
  of a {\mysos} permutation of degree $m$ \cite[Satz I]{sur}
  (see also \cite{Shu}), which in turn
  determine the entire permutation by the recurrence
  given in \cite[Theorem I]{sos}.

  The {\it inverses\/} of the {\mysos} permutations have
  been also studied \cite{obr, coo, tikan1,tikan2,ranking2,arXiv2024}.
  Among them, \cite{tikan1} looked at a recurrence
  (which is different from one given in \cite[Theorem I]{sos})
  satisfied by the terms of a permutation
  and
  showed that any of the inverse of a {\mysos} permutation satisfies
  the recurrence. Recently, the converse of this fact, %
  that a permutation satisfying the recurrence is the inverse
  of a {\mysos} permutation, was shown \cite{arXiv2024}.
  As an application, \cite{arXiv2024} presented a
  combinatorial
  procedure that lifts each of the permutations of degree
  $m-1$ satisfying the recurrence for $m-1$
  to the degree $m$ permutations that satisfy
  the recurrence for $m$. Starting with $\{\mathrm{Id} : [1] \rightarrow [1]\}$
  of degree $1$, repeating the procedure 
  forms a rooted binary tree, where the set of the vertices
  of level $m-1$ are the permutations of
  degree $m$
  satisfying the
  recurrence for $m$. Then it was shown that the tree is
  isomorphic to a tree in which the vertices of level
  $m-1$ are the open intervals formed by successive
  two terms of $\Th{m}$ Farey sequence (or equivalently,
  the inverse of the permutation corresponding
  to the interval via {\mysur}'s bijection)
  and level $m-1$ and $m$ vertices are adjacent
  if and only if the latter vertex, an open interval,
  is an open sub-interval
  of the former vertex. The isomorphic trees provide
  a purely combinatorial way to enumerate
  the inverses of the {\mysos} permutations whose
  definition depends on a real number $\alpha$.

  As another branch of study, a two-dimensional counterpart
  of the inverses of
  the {\mysos} permutations/{\mysur}'s bijection
  is seeked in a series of studies
  \cite{ranking,ranking2,ranking3}. In \cite{ranking},
  the bijections $\theta: [m] \times [n] \rightarrow [mn]$
  were considered as a 2d-extension of permutations.
  A special class of such $\theta$,
  those which have the presentation
  \begin{equation}
    \theta\myv{i}{j} =
    \sum_{s \in [m]}\sum_{t \in [n]}
    \myChi{ \{\alpha s + \beta t\} \leq \{\alpha i + \beta j\}  }
  \quad (\myv{i}{j} \in [m] \times [n])
  \end{equation}
  where $\{\cdot\}$ denotes the fractional part of a real number
  and $\myv{\alpha}{\beta} \in (0,1)^2$, was referred to as the
  \textit{ranking tables} and was considered as
  an extension of the inverses of the {\mysos} permutations.
  The 2d-counterpart of the Farey intervals in this setting
  are the \textit{Farey regions} which are the resulting
  open connected components when the square $[0,1]^2$ is
  divided by \textit{Farey lines} $DF(m,n)$, those lines
  whose slopes and intercepts are rational numbers
  whose numerators and denominators are bounded in terms
  of $m$ and $n$. Farey lines/regions and their relation
  to digital geometry
  were studied in
  \cite{FldFd,Acco,HAL,KFur,KErr}.
  In \cite{ranking}, it was shown
  that for each of a pair $\myv{m}{n}$ of positive integers,
  there is a surjection from the Farey regions
  (resulting from the division of $[0,1]^2$ by the lines
   $L(m,n)$ which is a slightly larger set than $DF(m-1,n-1)$)
  to the ranking tables of 
  \order $\myv{m}{n}$. This surjection may be seen
  as a 2d-counterpart of {\mysur}'s bijection. %
  Recently, the surjection turned out to be bijective
    for both of the two definitions of the Farey regions 
    (the division of $[0,1]^2$ by $L(m,n)$ and by $DF(m-1,n-1)$)
    \cite{ranking4}.
    
  In \cite{ranking},
  restricting the map to the \textit{fan regions},
  which is the collection of the Farey regions
  that touch the origin $\myv{0}{0}$,
  it was shown that the restriction is a bijection
  from the fan regions to \textit{Young ranking tables},
  which are those ranking tables in which
  entries are sorted in increasing order
  both row- and column-wise.

  In \cite{ranking2}, a number of combinatorial types of bijections
  $\theta: [m] \times [n] \rightarrow [mn]$ were defined.
  Though these types are defined so that the ranking tables
    fall into them, they are purely combinatorial
    in that their definitions do not refer to the real parameters
    $\alpha, \beta$ of the ranking tables.
  Among them, \textit{injective tables of difference equation type} 
  and \textit{injective tables of point-symmetric sum congruent type} are
  such types that if a $\theta$ satisfies both,
  then the L-shape restriction $\theta|_{([m]\times\{1\})\cup(\{1\}\times[n])}$
  determines the entire $\theta:[m]\times[n]\to[mn].$
  When $n = 1$, the former type reduces to the set of permutations
  satisfying a congruential recurrence
  \cite[Theorem 37]{ranking2}.
  Later, the set was
  proved to be the set of the inverses of {\mysos} permutations
  \cite{arXiv2024}.
  In \cite{ranking2}, it was shown
  that a ranking table is of both types. Thus, a ranking table
  is determined by its L-shape restriction. This result is quoted
  as Fact \ref{fact:YandLY} in this paper.
  Also in \cite{ranking2}, a numerical experiment
  was conducted for a limited number of combinations of $\myv{m}{n}$
  to show that an injective table of 
  difference equation type is always a ranking table
  for such $\myv{m}{n}$.

  As it turned out that the L-shape restriction is essential for
  the ranking tables, L-shapes and  their trees were studied 
  in \cite{ranking3}. A generalized set 
$\Garey{m}{n}$ of Farey fractions (Definition \ref{defi:garey})
  was considered and the fan regions
were identified with $\Gint{m}{n}$.
And it was shown that
the L-shape restriction
$\theta|_{([m] \times \{1\})  \cup (\{1\}\times[n])} $
  of a Young ranking table $\theta$
  is always an \textit{injective L-shape of difference equation type},
  as provided in Fact \ref{fact:Lincl}.
  Then, hoping the isomorphic result as obtained in \cite{arXiv2024},
  the two trees
  $\mathbb{T}_N^1 = \mathbb{T}_N^{\Interdon{},\trimapdon{}}$
  and
  $\mathbb{T}_N^2 = \mathbb{T}_N^{\Lsetitimonozerodon{},\upmapzerodon{}}$
  were compared upto $N=1000$.
  They are in fact isomorphic for arbitrary $N$,
  as shown in Theorem \ref{theo:main}
  of the current paper.
  An implication on L-shape
  is presented as Corollary \ref{coro:Lset}.
  A further implication on the Young ranking tables is given  below.

\section{An implication on the Young ranking tables}\label{sect:implranking}
In this section, definitions and known facts on
\textit{Farey regions}, \textit{injective tables} and their
types, and \textit{ranking tables} are quoted. Then an implication
of Theorem \ref{theo:main} on the Young ranking tables
is shown.

\subsection{Farey regions and generalized Farey sequences}
\begin{defi}[\cite{FldFd}]
Let $m$ and $n$ be nonnegative integers.
The set of points
\begin{equation}
  DF\myh{m}{n} := \left\{ \myv{\alpha}{\beta} \in [0,1]^2 :
  \exists \myv{i}{j} \in (\mathbb{Z} \cap [-m,m]) \times
  (\mathbb{Z} \cap [-n,n]) \setminus \{\myv{0}{0}\}
  \ \ \  i \alpha + j \beta \in \mathbb{Z} \right\} \label{eq:DFdef}
\end{equation}
is referred to as the Farey lines of order $\myv{m}{n}$.
\end{defi}
Note that for given $\myv{m}{n}$, $DF\myh{m}{n}$ (which is referred to as
$L_F(m+1,n+1)$ in \cite{ranking}) 
partition $[0,1]^2$ into a number of open connected components.
\begin{defi}[\cite{HAL}]
  The set $FF\myh{m}{n}$
  of all open connected components resulting from
  the partition of $[0,1]^2$ by $DF\myh{m}{n}$ is referred to as
  the Farey regions of order $\myv{m}{n}$.
\end{defi}
\newsavebox{\BLFtwothree}
\savebox{\BLFtwothree}{\scalebox{0.4}{\begin{tikzpicture}[scale=10]
\draw (0,0) -- (0.25,0.125) -- (0.333333,0) -- (0,0);
\draw (0,0) -- (0.111111,0.333333) -- (0,0.5) -- (0,0);
\draw (0,0) -- (0.142857,0.285714) -- (0.111111,0.333333) -- (0,0);
\draw (0,0) -- (0.142857,0.285714) -- (0.166667,0.25) -- (0,0);
\draw (0,0) -- (0.2,0.2) -- (0.166667,0.25) -- (0,0);
\draw (0,0) -- (0.25,0.125) -- (0.2,0.2) -- (0,0);
\draw (1,0) -- (0.75,0.125) -- (0.666667,0) -- (1,0);
\draw (0.888889,0.333333) -- (1,0.5) -- (1,0) -- (0.888889,0.333333);
\draw (1,0) -- (0.8,0.2) -- (0.833333,0.25) -- (1,0);
\draw (1,0) -- (0.857143,0.285714) -- (0.833333,0.25) -- (1,0);
\draw (1,0) -- (0.857143,0.285714) -- (0.888889,0.333333) -- (1,0);
\draw (1,0) -- (0.75,0.125) -- (0.8,0.2) -- (1,0);
\draw (0.25,0.875) -- (0,1) -- (0.333333,1) -- (0.25,0.875);
\draw (0,0.5) -- (0.111111,0.666667) -- (0,1) -- (0,0.5);
\draw (0.2,0.8) -- (0,1) -- (0.166667,0.75) -- (0.2,0.8);
\draw (0.142857,0.714286) -- (0,1) -- (0.166667,0.75) -- (0.142857,0.714286);
\draw (0.142857,0.714286) -- (0,1) -- (0.111111,0.666667) -- (0.142857,0.714286);
\draw (0.25,0.875) -- (0,1) -- (0.2,0.8) -- (0.25,0.875);
\draw (0.75,0.875) -- (1,1) -- (0.666667,1) -- (0.75,0.875);
\draw (0.888889,0.666667) -- (1,1) -- (1,0.5) -- (0.888889,0.666667);
\draw (0.857143,0.714286) -- (1,1) -- (0.888889,0.666667) -- (0.857143,0.714286);
\draw (0.857143,0.714286) -- (1,1) -- (0.833333,0.75) -- (0.857143,0.714286);
\draw (0.8,0.8) -- (1,1) -- (0.833333,0.75) -- (0.8,0.8);
\draw (0.75,0.875) -- (1,1) -- (0.8,0.8) -- (0.75,0.875);
\draw (0.5,0) -- (0.4,0.1) -- (0.333333,0) -- (0.5,0);
\draw (0.25,0.125) -- (0.285714,0.142857) -- (0.333333,0) -- (0.25,0.125);
\draw (0.285714,0.142857) -- (0.333333,0.166667) -- (0.333333,0) -- (0.285714,0.142857);
\draw (0.375,0.125) -- (0.333333,0.166667) -- (0.333333,0) -- (0.375,0.125);
\draw (0.4,0.1) -- (0.375,0.125) -- (0.333333,0) -- (0.4,0.1);
\draw (0.5,0) -- (0.6,0.1) -- (0.666667,0) -- (0.5,0);
\draw (0.6,0.1) -- (0.625,0.125) -- (0.666667,0) -- (0.6,0.1);
\draw (0.625,0.125) -- (0.666667,0.166667) -- (0.666667,0) -- (0.625,0.125);
\draw (0.714286,0.142857) -- (0.666667,0.166667) -- (0.666667,0) -- (0.714286,0.142857);
\draw (0.75,0.125) -- (0.714286,0.142857) -- (0.666667,0) -- (0.75,0.125);
\draw (0.5,0) -- (0.571429,0.142857) -- (0.5,0.25) -- (0.5,0);
\draw (0.5,0) -- (0.5,0.25) -- (0.428571,0.142857) -- (0.5,0);
\draw (0.5,0) -- (0.6,0.1) -- (0.571429,0.142857) -- (0.5,0);
\draw (0.5,0) -- (0.428571,0.142857) -- (0.4,0.1) -- (0.5,0);
\draw (0.4,0.9) -- (0.5,1) -- (0.333333,1) -- (0.4,0.9);
\draw (0.375,0.875) -- (0.4,0.9) -- (0.333333,1) -- (0.375,0.875);
\draw (0.333333,0.833333) -- (0.375,0.875) -- (0.333333,1) -- (0.333333,0.833333);
\draw (0.333333,0.833333) -- (0.285714,0.857143) -- (0.333333,1) -- (0.333333,0.833333);
\draw (0.285714,0.857143) -- (0.25,0.875) -- (0.333333,1) -- (0.285714,0.857143);
\draw (0.6,0.9) -- (0.5,1) -- (0.666667,1) -- (0.6,0.9);
\draw (0.714286,0.857143) -- (0.75,0.875) -- (0.666667,1) -- (0.714286,0.857143);
\draw (0.666667,0.833333) -- (0.714286,0.857143) -- (0.666667,1) -- (0.666667,0.833333);
\draw (0.666667,0.833333) -- (0.625,0.875) -- (0.666667,1) -- (0.666667,0.833333);
\draw (0.625,0.875) -- (0.6,0.9) -- (0.666667,1) -- (0.625,0.875);
\draw (0.428571,0.857143) -- (0.5,1) -- (0.5,0.75) -- (0.428571,0.857143);
\draw (0.5,0.75) -- (0.5,1) -- (0.571429,0.857143) -- (0.5,0.75);
\draw (0.4,0.9) -- (0.5,1) -- (0.428571,0.857143) -- (0.4,0.9);
\draw (0.571429,0.857143) -- (0.5,1) -- (0.6,0.9) -- (0.571429,0.857143);
\draw (0.125,0.375) -- (0,0.5) -- (0.111111,0.333333) -- (0.125,0.375);
\draw (0,0.5) -- (0.166667,0.5) -- (0.142857,0.571429) -- (0,0.5);
\draw (0,0.5) -- (0.166667,0.5) -- (0.142857,0.428571) -- (0,0.5);
\draw (0,0.5) -- (0.125,0.625) -- (0.111111,0.666667) -- (0,0.5);
\draw (0,0.5) -- (0.142857,0.571429) -- (0.125,0.625) -- (0,0.5);
\draw (0.142857,0.428571) -- (0,0.5) -- (0.125,0.375) -- (0.142857,0.428571);
\draw (1,0.5) -- (0.875,0.625) -- (0.888889,0.666667) -- (1,0.5);
\draw (0.833333,0.5) -- (1,0.5) -- (0.857143,0.428571) -- (0.833333,0.5);
\draw (0.833333,0.5) -- (1,0.5) -- (0.857143,0.571429) -- (0.833333,0.5);
\draw (0.875,0.375) -- (1,0.5) -- (0.888889,0.333333) -- (0.875,0.375);
\draw (0.857143,0.428571) -- (1,0.5) -- (0.875,0.375) -- (0.857143,0.428571);
\draw (1,0.5) -- (0.857143,0.571429) -- (0.875,0.625) -- (1,0.5);
\draw (0.285714,0.571429) -- (0.333333,0.666667) -- (0.333333,0.5) -- (0.285714,0.571429);
\draw (0.25,0.5) -- (0.333333,0.5) -- (0.285714,0.571429) -- (0.25,0.5);
\draw (0.333333,0.5) -- (0.5,0.5) -- (0.4,0.4) -- (0.333333,0.5);
\draw (0.333333,0.333333) -- (0.4,0.4) -- (0.333333,0.5) -- (0.333333,0.333333);
\draw (0.333333,0.5) -- (0.5,0.5) -- (0.4,0.6) -- (0.333333,0.5);
\draw (0.4,0.6) -- (0.333333,0.666667) -- (0.333333,0.5) -- (0.4,0.6);
\draw (0.333333,0.333333) -- (0.285714,0.428571) -- (0.333333,0.5) -- (0.333333,0.333333);
\draw (0.25,0.5) -- (0.333333,0.5) -- (0.285714,0.428571) -- (0.25,0.5);
\draw (0.2,0.7) -- (0.25,0.75) -- (0.222222,0.666667) -- (0.2,0.7);
\draw (0.2,0.6) -- (0.25,0.625) -- (0.222222,0.666667) -- (0.2,0.6);
\draw (0.5,0.25) -- (0.5,0.5) -- (0.444444,0.333333) -- (0.5,0.25);
\draw (0.5,0.25) -- (0.428571,0.285714) -- (0.444444,0.333333) -- (0.5,0.25);
\draw (0.4,0.4) -- (0.5,0.5) -- (0.444444,0.333333) -- (0.4,0.4);
\draw (0.6,0.2) -- (0.5,0.25) -- (0.571429,0.142857) -- (0.6,0.2);
\draw (0.5,0.25) -- (0.5,0.5) -- (0.555556,0.333333) -- (0.5,0.25);
\draw (0.5,0.25) -- (0.571429,0.285714) -- (0.555556,0.333333) -- (0.5,0.25);
\draw (0.5,0.25) -- (0.571429,0.285714) -- (0.6,0.2) -- (0.5,0.25);
\draw (0.4,0.2) -- (0.5,0.25) -- (0.428571,0.285714) -- (0.4,0.2);
\draw (0.4,0.2) -- (0.5,0.25) -- (0.428571,0.142857) -- (0.4,0.2);
\draw (0.25,0.625) -- (0.333333,0.666667) -- (0.285714,0.571429) -- (0.25,0.625);
\draw (0.666667,0.333333) -- (0.714286,0.428571) -- (0.666667,0.5) -- (0.666667,0.333333);
\draw (0.666667,0.5) -- (0.75,0.5) -- (0.714286,0.428571) -- (0.666667,0.5);
\draw (0.5,0.5) -- (0.666667,0.5) -- (0.6,0.6) -- (0.5,0.5);
\draw (0.6,0.6) -- (0.666667,0.666667) -- (0.666667,0.5) -- (0.6,0.6);
\draw (0.5,0.5) -- (0.666667,0.5) -- (0.6,0.4) -- (0.5,0.5);
\draw (0.666667,0.333333) -- (0.6,0.4) -- (0.666667,0.5) -- (0.666667,0.333333);
\draw (0.714286,0.571429) -- (0.666667,0.666667) -- (0.666667,0.5) -- (0.714286,0.571429);
\draw (0.666667,0.5) -- (0.75,0.5) -- (0.714286,0.571429) -- (0.666667,0.5);
\draw (0.5,0.5) -- (0.5,0.75) -- (0.555556,0.666667) -- (0.5,0.5);
\draw (0.571429,0.714286) -- (0.5,0.75) -- (0.555556,0.666667) -- (0.571429,0.714286);
\draw (0.5,0.5) -- (0.6,0.6) -- (0.555556,0.666667) -- (0.5,0.5);
\draw (0.75,0.25) -- (0.8,0.3) -- (0.777778,0.333333) -- (0.75,0.25);
\draw (0.75,0.375) -- (0.8,0.4) -- (0.777778,0.333333) -- (0.75,0.375);
\draw (0.5,0.75) -- (0.4,0.8) -- (0.428571,0.857143) -- (0.5,0.75);
\draw (0.5,0.5) -- (0.5,0.75) -- (0.444444,0.666667) -- (0.5,0.5);
\draw (0.428571,0.714286) -- (0.5,0.75) -- (0.444444,0.666667) -- (0.428571,0.714286);
\draw (0.428571,0.714286) -- (0.5,0.75) -- (0.4,0.8) -- (0.428571,0.714286);
\draw (0.5,0.75) -- (0.6,0.8) -- (0.571429,0.714286) -- (0.5,0.75);
\draw (0.5,0.75) -- (0.6,0.8) -- (0.571429,0.857143) -- (0.5,0.75);
\draw (0.666667,0.333333) -- (0.75,0.375) -- (0.714286,0.428571) -- (0.666667,0.333333);
\draw (0.2,0.4) -- (0.142857,0.428571) -- (0.166667,0.5) -- (0.2,0.4);
\draw (0.166667,0.5) -- (0.25,0.5) -- (0.2,0.4) -- (0.166667,0.5);
\draw (0.142857,0.571429) -- (0.2,0.6) -- (0.166667,0.5) -- (0.142857,0.571429);
\draw (0.166667,0.5) -- (0.25,0.5) -- (0.2,0.6) -- (0.166667,0.5);
\draw (0.25,0.25) -- (0.2,0.3) -- (0.222222,0.333333) -- (0.25,0.25);
\draw (0.25,0.375) -- (0.2,0.4) -- (0.222222,0.333333) -- (0.25,0.375);
\draw (0.285714,0.142857) -- (0.333333,0.166667) -- (0.25,0.25) -- (0.285714,0.142857);
\draw (0.25,0.25) -- (0.333333,0.333333) -- (0.333333,0.166667) -- (0.25,0.25);
\draw (0.5,0.5) -- (0.4,0.6) -- (0.444444,0.666667) -- (0.5,0.5);
\draw (0.6,0.4) -- (0.5,0.5) -- (0.555556,0.333333) -- (0.6,0.4);
\draw (0.4,0.8) -- (0.333333,0.833333) -- (0.375,0.875) -- (0.4,0.8);
\draw (0.333333,0.666667) -- (0.428571,0.714286) -- (0.4,0.8) -- (0.333333,0.666667);
\draw (0.4,0.8) -- (0.333333,0.833333) -- (0.333333,0.666667) -- (0.4,0.8);
\draw (0.666667,0.166667) -- (0.6,0.2) -- (0.625,0.125) -- (0.666667,0.166667);
\draw (0.571429,0.285714) -- (0.666667,0.333333) -- (0.6,0.2) -- (0.571429,0.285714);
\draw (0.666667,0.166667) -- (0.6,0.2) -- (0.666667,0.333333) -- (0.666667,0.166667);
\draw (0.857143,0.571429) -- (0.8,0.6) -- (0.833333,0.5) -- (0.857143,0.571429);
\draw (0.75,0.5) -- (0.833333,0.5) -- (0.8,0.6) -- (0.75,0.5);
\draw (0.8,0.4) -- (0.857143,0.428571) -- (0.833333,0.5) -- (0.8,0.4);
\draw (0.75,0.5) -- (0.833333,0.5) -- (0.8,0.4) -- (0.75,0.5);
\draw (0.8,0.7) -- (0.75,0.75) -- (0.777778,0.666667) -- (0.8,0.7);
\draw (0.8,0.6) -- (0.75,0.625) -- (0.777778,0.666667) -- (0.8,0.6);
\draw (0.666667,0.833333) -- (0.714286,0.857143) -- (0.75,0.75) -- (0.666667,0.833333);
\draw (0.666667,0.666667) -- (0.75,0.75) -- (0.666667,0.833333) -- (0.666667,0.666667);
\draw (0.333333,0.166667) -- (0.4,0.2) -- (0.333333,0.333333) -- (0.333333,0.166667);
\draw (0.333333,0.166667) -- (0.4,0.2) -- (0.375,0.125) -- (0.333333,0.166667);
\draw (0.25,0.75) -- (0.333333,0.833333) -- (0.333333,0.666667) -- (0.25,0.75);
\draw (0.428571,0.285714) -- (0.333333,0.333333) -- (0.4,0.2) -- (0.428571,0.285714);
\draw (0.333333,0.333333) -- (0.25,0.375) -- (0.285714,0.428571) -- (0.333333,0.333333);
\draw (0.333333,0.833333) -- (0.285714,0.857143) -- (0.25,0.75) -- (0.333333,0.833333);
\draw (0.666667,0.166667) -- (0.75,0.25) -- (0.666667,0.333333) -- (0.666667,0.166667);
\draw (0.6,0.8) -- (0.666667,0.833333) -- (0.666667,0.666667) -- (0.6,0.8);
\draw (0.6,0.8) -- (0.666667,0.833333) -- (0.625,0.875) -- (0.6,0.8);
\draw (0.666667,0.666667) -- (0.571429,0.714286) -- (0.6,0.8) -- (0.666667,0.666667);
\draw (0.75,0.625) -- (0.666667,0.666667) -- (0.714286,0.571429) -- (0.75,0.625);
\draw (0.714286,0.142857) -- (0.666667,0.166667) -- (0.75,0.25) -- (0.714286,0.142857);
\draw (0.111111,0.333333) -- (0.142857,0.285714) -- (0.166667,0.333333) -- (0.125,0.375) -- (0.111111,0.333333);
\draw (0.166667,0.25) -- (0.142857,0.285714) -- (0.166667,0.333333) -- (0.2,0.3) -- (0.166667,0.25);
\draw (0.166667,0.25) -- (0.2,0.2) -- (0.25,0.25) -- (0.2,0.3) -- (0.166667,0.25);
\draw (0.2,0.2) -- (0.25,0.125) -- (0.285714,0.142857) -- (0.25,0.25) -- (0.2,0.2);
\draw (0.2,0.6) -- (0.166667,0.666667) -- (0.2,0.7) -- (0.222222,0.666667) -- (0.2,0.6);
\draw (0.222222,0.666667) -- (0.25,0.625) -- (0.333333,0.666667) -- (0.25,0.75) -- (0.222222,0.666667);
\draw (0.428571,0.285714) -- (0.333333,0.333333) -- (0.4,0.4) -- (0.444444,0.333333) -- (0.428571,0.285714);
\draw (0.166667,0.75) -- (0.2,0.7) -- (0.25,0.75) -- (0.2,0.8) -- (0.166667,0.75);
\draw (0.142857,0.714286) -- (0.166667,0.666667) -- (0.2,0.7) -- (0.166667,0.75) -- (0.142857,0.714286);
\draw (0.25,0.5) -- (0.2,0.6) -- (0.25,0.625) -- (0.285714,0.571429) -- (0.25,0.5);
\draw (0.571429,0.142857) -- (0.6,0.1) -- (0.625,0.125) -- (0.6,0.2) -- (0.571429,0.142857);
\draw (0.555556,0.666667) -- (0.6,0.6) -- (0.666667,0.666667) -- (0.571429,0.714286) -- (0.555556,0.666667);
\draw (0.777778,0.333333) -- (0.8,0.3) -- (0.833333,0.333333) -- (0.8,0.4) -- (0.777778,0.333333);
\draw (0.75,0.25) -- (0.666667,0.333333) -- (0.75,0.375) -- (0.777778,0.333333) -- (0.75,0.25);
\draw (0.8,0.2) -- (0.75,0.25) -- (0.8,0.3) -- (0.833333,0.25) -- (0.8,0.2);
\draw (0.833333,0.25) -- (0.8,0.3) -- (0.833333,0.333333) -- (0.857143,0.285714) -- (0.833333,0.25);
\draw (0.4,0.8) -- (0.375,0.875) -- (0.4,0.9) -- (0.428571,0.857143) -- (0.4,0.8);
\draw (0.714286,0.428571) -- (0.75,0.375) -- (0.8,0.4) -- (0.75,0.5) -- (0.714286,0.428571);
\draw (0.875,0.625) -- (0.833333,0.666667) -- (0.857143,0.714286) -- (0.888889,0.666667) -- (0.875,0.625);
\draw (0.8,0.7) -- (0.833333,0.666667) -- (0.857143,0.714286) -- (0.833333,0.75) -- (0.8,0.7);
\draw (0.8,0.7) -- (0.75,0.75) -- (0.8,0.8) -- (0.833333,0.75) -- (0.8,0.7);
\draw (0.75,0.75) -- (0.714286,0.857143) -- (0.75,0.875) -- (0.8,0.8) -- (0.75,0.75);
\draw (0.111111,0.666667) -- (0.125,0.625) -- (0.166667,0.666667) -- (0.142857,0.714286) -- (0.111111,0.666667);
\draw (0.222222,0.333333) -- (0.25,0.25) -- (0.333333,0.333333) -- (0.25,0.375) -- (0.222222,0.333333);
\draw (0.2,0.3) -- (0.166667,0.333333) -- (0.2,0.4) -- (0.222222,0.333333) -- (0.2,0.3);
\draw (0.166667,0.333333) -- (0.2,0.4) -- (0.142857,0.428571) -- (0.125,0.375) -- (0.166667,0.333333);
\draw (0.285714,0.428571) -- (0.25,0.375) -- (0.2,0.4) -- (0.25,0.5) -- (0.285714,0.428571);
\draw (0.125,0.625) -- (0.142857,0.571429) -- (0.2,0.6) -- (0.166667,0.666667) -- (0.125,0.625);
\draw (0.4,0.6) -- (0.333333,0.666667) -- (0.428571,0.714286) -- (0.444444,0.666667) -- (0.4,0.6);
\draw (0.555556,0.333333) -- (0.571429,0.285714) -- (0.666667,0.333333) -- (0.6,0.4) -- (0.555556,0.333333);
\draw (0.75,0.625) -- (0.666667,0.666667) -- (0.75,0.75) -- (0.777778,0.666667) -- (0.75,0.625);
\draw (0.777778,0.666667) -- (0.8,0.6) -- (0.833333,0.666667) -- (0.8,0.7) -- (0.777778,0.666667);
\draw (0.857143,0.285714) -- (0.833333,0.333333) -- (0.875,0.375) -- (0.888889,0.333333) -- (0.857143,0.285714);
\draw (0.875,0.625) -- (0.857143,0.571429) -- (0.8,0.6) -- (0.833333,0.666667) -- (0.875,0.625);
\draw (0.75,0.5) -- (0.8,0.6) -- (0.75,0.625) -- (0.714286,0.571429) -- (0.75,0.5);
\draw (0.833333,0.333333) -- (0.8,0.4) -- (0.857143,0.428571) -- (0.875,0.375) -- (0.833333,0.333333);
\draw (0.25,0.75) -- (0.285714,0.857143) -- (0.25,0.875) -- (0.2,0.8) -- (0.25,0.75);
\draw (0.4,0.1) -- (0.428571,0.142857) -- (0.4,0.2) -- (0.375,0.125) -- (0.4,0.1);
\draw (0.625,0.875) -- (0.6,0.8) -- (0.571429,0.857143) -- (0.6,0.9) -- (0.625,0.875);
\draw (0.8,0.2) -- (0.75,0.125) -- (0.714286,0.142857) -- (0.75,0.25) -- (0.8,0.2);
\end{tikzpicture}}}
\begin{figure}
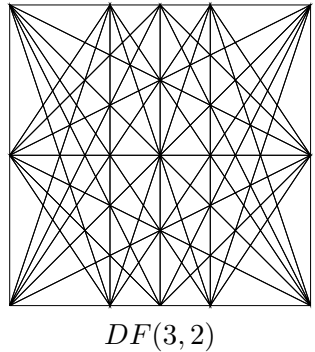

\centering
\begin{tabular}{c}
  \rotatebox{0}{\usebox{\BLFtwothree}}  \\
  $DF\myh{3}{2}$ 
\end{tabular}
\caption{An example of Farey lines $DF\myh{m}{n}$ and resulting
 Farey regions $FF\myh{m}{n}$}\label{fig:DF}
\end{figure}
For example, $DF\myh{3}{2}$ contains the intersection of $[0,1]^2$
and the lines of the form $i \alpha + j \beta = k$ where
$|i| = 0, 1, 2, 3$, $|j|=0, 1, 2$ with $|i|+|j|>0$ and $k \in \mathbb{Z}$
as  shown in Fig.~\ref{fig:DF}. These lines partition $[0,1]^2$
to yield the Farey regions of order $\myv{3}{2}$ consisting of
the 180 open connected components which are 40 open quadrilaterals
plus 140 open triangles. 
\begin{defi}\label{defi:fanregion}
  Among the set of the Farey regions of order $\myv{m}{n}$,
  the subset consisting of the regions that the boundary
  contains the origin $\myv{0}{0}$ are referred to as
  the set of the \textit{fan regions} of order $\myv{m}{n}$.
\end{defi}
For example, as shown in Fig.~\ref{fig:DF}, the fan regions
of order $\myv{3}{2}$ consist of 6 triangles around the origin
which are the shaded triangles
below the line $m \alpha + n \beta = 1$
in Fig.~\ref{fig:DFzero}.
They correspond to the partition of $[0,1]^2$ by such lines
in 
\begin{equation*}
 DF\myh{3}{2} = \{ \myv{\alpha}{\beta} : i \alpha - j \beta = 0, i \in \{0,1,2,3\}
  , j \in \{0,1,2\}, \myv{i}{j}\neq\myv{0}{0} \}
\end{equation*}
 that intersect with the origin
as shown in Fig.~\ref{fig:DFzero}.
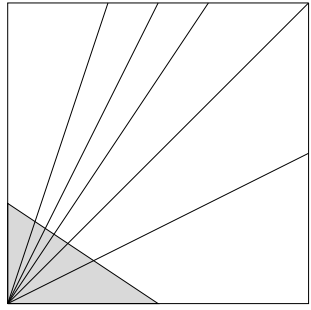
\begin{figure}
\centering
\scalebox{0.4}{\begin{tikzpicture}[scale=10]
\draw [fill=gray!30] (0,0) -- (0.5,0) -- (0,0.333333) -- (0,0);
\draw (0,0) -- (1.0,0.0) -- (1.0,1.0);
\draw (0,0) -- (1.0,0.5);
\draw (0,0) -- (0.333333,1.0);
\draw (0,0) -- (0.5,1.0);
\draw (0,0) -- (0.667, 1.0);
\draw (0,0) -- (1.0,1.0);
\draw (0,0) -- (0.0,1.0) -- (1.0,1.0);
\end{tikzpicture}}
\caption{The fan regions of order $\myv{3}{2}$}
\label{fig:DFzero}
\end{figure}
The slopes of these lines form a $7$-term sequence
\begin{equation*}
  0 = \frac{0}{1} < \frac{1}{2} < \frac{1}{1} < \frac{3}{2} <
  \frac{2}{1} <
  \frac{3}{1} < \frac{1}{0} = \infty
\end{equation*}
and the 6 open intervals between two adjacent terms of it
bijectively correspond to the 6 fan regions of order $\myv{3}{2}.$
In general, the slopes of those lines in $DF\myh{m}{n}$
which intersect the origin
correspond to $\Gseq{m}{n}$ of Definition \ref{defi:garey}
and the fan regions are identified
with the set $\Gint{m}{n}$ of generalized Farey intervals.
\subsection{Injective tables of difference equation/point-symmetric sum congruent types}
In \cite{ranking}, a two-dimensional analog of a permutation was defined.
\begin{defi}[\cite{ranking}]
Let $m$ and $n$ be positive integers.
An  \textit{injective table} of \order $\myv{m}{n}$
is an injection $\theta:[m]\times [n] \rightarrow [mn].$
The first and second arguments $i \in [m]$ and $j \in [n]$ of $\theta$ are
referred to as horizontal and vertical indices of $\theta$, respectively.
\end{defi}
Note that an injective table is in fact a bijection and
there are $(mn)!$ injective tables  of \order $\myv{m}{n}$ in total.
Also, an injective table of \order $\myv{m}{1}$ (resp. $\myv{1}{n}$)
is identified with a permutation over $[m]$ (resp. $[n]$).

  For an injective table $\theta$, its \textit{transposition}
  ${}^t \theta$ is defined by the swap of the horizontal and vertical
  indices, similarly as the transposition for an injective L-shape
  (Definition \ref{defi:Ltp}).
\begin{defi}\label{defi:thetatp}
  Let $m$ and $n$ be positive integers and let $\theta$
  be an injective table $\theta: [m]\times[n]\to[mn].$
  The transposition ${}^t \theta$ for the injective table $\theta$
  is the injection
  from $[n]\times[m]$ to $[mn]$ defined by
  \begin{equation}
 {}^t \theta\myv{j}{i} := \theta\myv{i}{j} \quad (\myv{j}{i} \in [n]\times[m]).
  \end{equation}
\end{defi}

For example, for $\myv{m}{n}=\myv{4}{3}$,
\begin{equation}
  \theta =
  \left|\begin{matrix}6&9&11&12\\3&5&8&10\\1&2&4&7\\\end{matrix}\right|
\label{eq:ex43}
\end{equation}
where $\theta\myv{1}{1}=1$ locates lower-left corner,
$\theta\myv{m}{1}=7$ and $\theta\myv{1}{n}=6$
locate lower-right and upper-left corners respectively,
is an example of injective tables of \order $\myv{4}{3}$.
Also, its transposition is
\begin{equation}
  {}^t \theta =
  \left|\begin{matrix}7&10&12\\4&8&11\\2&5&9\\1&3&6\\\end{matrix}\right|.
\label{eq:ex43tp}
\end{equation}

An \textit{L-shape} appears as the restriction of a map
$[m] \times [n] \to [mn]$ to the subset of indices $([m] \times \{1\}) \cup
(\{1\} \times [n])$ which corresponds to a substructure of
shape like the letter ``L'' in an $m \times n$ table.
\begin{defi}[\cite{ranking3}]
Let $m$ and $n$ be positive integers.
An \textit{injective L-shape} of \order $\myv{m}{n}$
is an injection $\theta_L:([m] \times \{1\}) \cup (\{1\} \times [n]) \rightarrow [mn].$
\end{defi}
For example,
\begin{equation}
  \theta_L = \left|
  \begin{matrix}6& & & \\3& & & \\1&2&4&7\\\end{matrix}\right|
\label{eq:Lex43again}
\end{equation}
is an injective L-shape of \order $\myv{4}{3},$ which is in fact
the restriction of \eqref{eq:ex43}.

Note that a given injective L-shape $\theta_L$ of \order $\myv{m}{n}$
may be or may not be the restriction 
of an injective table of \order $\myv{m}{n}$ to $([m] \times \{1\}) \cup (\{1\} \times [n])$,
though 
the restriction $\theta|_{([m] \times \{1\}) \cup (\{1\} \times [n])}$
of an injective table $\theta$ of \order $\myv{m}{n}$
is always an injective L-shape of \order $\myv{m}{n}$.

\subsection{Injective tables and L-shapes}\label{sss:diffeqtypea}
In \cite{ranking2}, a number of types of tables were studied. Among them,
the definitions of two types of tables are
presented below.
\begin{defi}[{\cite[Definition 31]{ranking2}}]\label{defi:deq}
  Let $m \geq 2$ and $n$ be positive integers and let
  $\theta$ be a map $\theta: [m]\times[n] \rightarrow [mn]$.
  If for all $\myv{i}{j} \in [m-1]\times[n]$ the values
  \begin{multline}\label{eq:hdiffeq}
  \theta\myv{i+1}{j} - \theta\myv{i}{j}
  - \sum_{t \in [n]} \myChi{\theta\myv{1}{t} \leq \theta\myv{i+1}{j}}
  + \sum_{t \in [n]} \myChi{\theta\myv{m}{t} \leq \theta\myv{i}{j}} \\
  - (m-1)n \cdot \left(\myChi{\theta\myv{2}{1} \leq \theta\myv{i+1}{j}}
  - \myChi{\theta\myv{1}{1} \leq \theta\myv{i}{j}}\right)
\end{multline}
  are equal, then $\theta$ is said to be a table of
  \textit{horizontal difference equation type.\/}
Also, if the transposition ${}^t \theta$ is a 
table of horizontal difference equation type,
then a table $\theta$ is said to be a table of
\textit{vertical difference equation type.\/}
Moreover, if a table is simultaneously of 
horizontal/vertical difference equation types,
then the table is said to be of
\textit{difference equation type.\/}
\end{defi}
\begin{defi}[{\cite[Definition 33]{ranking2}}]\label{defi:pss}
  Let $m $ and $n$ be positive integers and let
  $\theta$ be a map $\theta: [m]\times[n] \rightarrow [mn]$.
  If the congruences
\begin{equation}\label{eq:pss}
  \theta\myv{i}{j}+\theta\myv{m+1-i}{n+1-j} \equiv \theta\myv{1}{1}+\theta\myv{m}{n}
   \pmod{mn}
\end{equation}  
are satisfied
for all $\myv{i}{j} \in [m]\times[n]$, then $\theta$ is
said to be of \textit{point-symmetric sum congruent
type.\/}
\end{defi}

Suppose that we are given  $m \geq 2, n \geq 1$ and
the values of L-shape $\theta|_{([m]\times \{1\}) \cup (\{1\}\times [n] )}$
  of an injective table $\theta:[m]\times[n]\rightarrow [mn]$ of
  horizontal difference equation type which is also of point-symmetric sum
  congruent type.
In \cite[Remark 43, Theorem 44, Corollary 45]{ranking2},
    such a procedure is provided
  that determines
  the whole table $\theta'$ from (a part of) the values
  $\theta'|_{([m]\times \{1\}) \cup (\{1\}\times [n] )}$, under the condition
  that $\theta'$ \textit{is a ranking table\/}. In fact,
  \textit{being an injective table of horizontal
    difference equation and point-symmetric
    sum congruent type\/} is an enough condition
  on $\theta'$ 
  for the procedure to run correctly.
  Thus, by the procedure,
  determining $\theta|_{([m]\times \{1\}) \cup (\{1\}\times [n] )}$
  is equivalent to identifying the entire $\theta$.

\begin{fact}[{A consequence of \cite[Remark 43, Theorem 44 and Corollary 45]{ranking2}}]\label{fact:detproc}
  Suppose that $m \geq 2, n \geq 1$ and $\theta$ is an injective
  table of \order $\myv{m}{n}$ which is simultaneously
  of horizontal
  difference equation type and of point-symmetric sum congruent
  type.
  There is a procedure that determines $\theta$ from the values
  $\theta\myv{2}{1}, \theta\myv{m}{1}$ and $\theta\myv{1}{j}$ for $j \in [n].$
\end{fact}

\subsection{Ranking tables, Young ranking tables and Young L-shapes}
Another special class of injective tables, that are parameterized
by real numbers $\alpha, \beta, \gamma$
and can be seen as a two dimensional
version of the inverses of the {\mysos} permutations,
was also defined \cite{ranking,ranking2}.
\begin{defi}[{\cite[Definition 9]{ranking2}}]\label{defi:rankingfrac}
  Let $m$ and $n$ be positive integers.
  For $\myv{\alpha}{\beta} \in \mathbb{R}^2$, $\gamma \in \mathbb{R}$ and
  $\myv{i}{j} \in [m] \times [n],$
  let
  $$
  \tau_\nyh{\alpha}{\beta}^{\gamma}\myv{i}{j} =
  \sum_{s \in [m]} \sum_{t \in [n]} \myChi{ 
\{ (s-1) \alpha + (t-1) \beta + \gamma\}
    \leq 
\{ (i-1) \alpha + (j-1) \beta + \gamma\}
  }
  $$
 where $\{ \cdot \}$ in LHS denotes the fractional part of a real number.

  The map $\tau_\nyh{\alpha}{\beta}^{\gamma}: [m] \times [n] \rightarrow [mn]$ is said to be a
  \textit{ranking table} if it is an injective table. The map is
  denoted as $\tau_\nyh{\alpha}{\beta}^{\gamma, \nyh{m}{n}}$ when the indication
  of the \order $\myv{m}{n}$ is necessary.
  The collection of all the ranking tables 
  $\tau_\nyh{\alpha}{\beta}^{0,\nyh{m}{n}}$ 
  is denoted by $\RT{m}{n}.$ Note that $\tau\myv{1}{1} = 1$
  for $\tau \in \RT{m}{n}$ by definition.
\end{defi}  
  In \cite{ranking}, it was shown
  that for each of a pair $\myv{m}{n}$ of positive integers,
  there is a surjection from the Farey regions,
  resulting from the division $[0,1]^2$ by the lines
\begin{equation}
  L\myh{m}{n} := \left\{ \myv{\alpha}{\beta} \in [0,1]^{2}  :
  \exists \myv{i}{j} \in (\mathbb{Z} \cap [0,m]) \times
  (\mathbb{Z} \cap [-n+1,n]) \setminus \{\myv{0}{0}\}
  \ \ \ i \alpha + j \beta \in \mathbb{Z} \right\} \label{eq:Linedef}
\end{equation}
  which is slightly larger set than $DF(m-1,n-1)$,
  to the ranking tables of $\gamma = \alpha + \beta$ and 
  \order $\myv{m}{n}$. This surjection may be seen
  as a 2d-counterpart of {\mysur}'s bijection \cite[Sats I]{sur}.%
\begin{fact}[{\cite[Theorem 4]{ranking}}]\label{fact:asur}
  Let $m, n $ be positive integers and let
  $\myv{\alpha}{\beta}, \myv{\alpha'}{\beta'} \in (0,1)^2.$
  If $\myv{\alpha}{\beta}$ and $\myv{\alpha'}{\beta'}$
  belong to a common Farey region (by $L(m,n)$ of \cite{ranking})
  of \order $\myv{m}{n}$, then
  $\tau_\nyh{\alpha}{\beta}^{\alpha+\beta, \nyh{m}{n}}$
  and 
  $\tau_\nyh{\alpha'}{\beta'}^{\alpha'+\beta', \nyh{m}{n}}$
  are the same. In other words, $\myv{\alpha}{\beta} \mapsto
  \tau_\nyh{\alpha}{\beta}^{\alpha+\beta, \nyh{m}{n}}$ can be regarded as
  a map from the set of the Farey regions of (by $L(m,n)$ of \cite{ranking})
  order $\myv{m}{n}$
  to the set of the ranking tables with $\gamma = \alpha + \beta$.
  Furthermore, this map is surjective.
\end{fact}  
Recently, the above 2d-counterpart and
  its variant, where $L(m,n)$ and $\gamma = \alpha+\beta$ are replaced with $DF(m-1,n-1)$ and $\gamma = 0$ respectively,
  are turned out to be bijective.
\begin{fact}[{\cite[Theorems 3.23 and 3.1]{ranking4}}]\label{fact:asurmore}
    The surjective map of Fact \ref{fact:asur}
    from the set of the Farey regions of (by $L(m,n)$ of \cite{ranking})
  order $\myv{m}{n}$
  to the set of the ranking tables with $\gamma = \alpha + \beta$ is bijective.
  Also, the map
  $\myv{\alpha}{\beta} \mapsto
  \tau_\nyh{\alpha}{\beta}^{0, \nyh{m}{n}}$
  defines a bijection from the Farey regions $FF(m-1, n-1)$ (which
  results from the division of $[0,1]^2$ by $DF(m-1, n-1)$)
  to $\RT{m}{n}$.
\end{fact}

A further special class of interest is \textit{Young ranking tables.}
They are a kind of well-known \textit{Young tableaux} (where
any of the rows or the columns forms an increasing sequence) of rectangular
shapes.
\begin{defi}[{\cite[Sect. 6]{ranking}} {\cite[Definition 14]{ranking2}}]
Let $m$ and $n$ be positive integers.
A ranking table $\tau_\nyh{\alpha}{\beta}^{\gamma, \nyh{m}{n}}$
is said to be a \textit{Young ranking table} if the monotonicity
$ i \leq i' \Rightarrow
\tau_\nyh{\alpha}{\beta}^{\gamma, \nyh{m}{n}}\myv{i}{j} \leq \tau_\nyh{\alpha}{\beta}^{\gamma, \nyh{m}{n}}\myv{i'}{j}$ and
$ j \leq j' \Rightarrow
\tau_\nyh{\alpha}{\beta}^{\gamma, \nyh{m}{n}}\myv{i}{j} \leq \tau_\nyh{\alpha}{\beta}^{\gamma, \nyh{m}{n}}\myv{i}{j'}$ are satisfied for all $i, i' \in [m]$ and for all $j, j' \in [n]$. The collection of all the
Young ranking tables of \order $\myv{m}{n}$ is denoted by
$\Yng{m}{n}.$
\end{defi}
The above definition sounds different from Definition \ref{defi:yr}
(by the presence of the fractional part $\{\cdot\}$ in Definition \ref{defi:rankingfrac}),
however they are in fact equivalent \cite[Definition 4.13]{ranking3}.

Hereafter we focus on the case $\gamma = 0$ and
  the indication of $\gamma = 0$ in $\tau_\nyh{\alpha}{\beta}^{0, \nyh{m}{n}}$
  is suppressed, i.e.,
  $\tau_\nyh{\alpha}{\beta}^{\nyh{m}{n}}$
  means
  $\tau_\nyh{\alpha}{\beta}^{0,\nyh{m}{n}}.$

The following fact motivates the study of the fan regions
of Definition \ref{defi:fanregion} and the generalized Farey
sequences of Definition \ref{defi:garey}.
\begin{fact}[{\cite[Proposition 9, Theorem 10, Corollary 11]{ranking}}]\label{fact:ranking}
  Let $m, n$ be positive integers and let $\myv{\alpha}{\beta},
  \myv{\alpha'}{\beta'} \in (0,1)^2$
  are points that do not belong to the lines $DF\myh{m-1}{n-1}$.
\begin{itemize}
\item[(i)] The ranking table $\tau_\nyh{\alpha}{\beta}^{ \nyh{m}{n}}$
  is a Young ranking table if and only if $(m-1) \alpha + (n-1) \beta < 1$.
\item[(ii)] Suppose that $(m-1) \alpha + (n-1) \beta < 1$ and that
  $(m-1) \alpha' + (n-1) \beta' < 1$. Two Young ranking tables
  $\tau_\nyh{\alpha}{\beta}^{ \nyh{m}{n}}$ and
  $\tau_\nyh{\alpha'}{\beta'}^{ \nyh{m}{n}}$
  are the same if and only if $\myv{\alpha}{\beta}$ and
  $\myv{\alpha'}{\beta'}$ are in a common fan region of \order $\myv{m-1}{n-1}$.
\item[(iii)] (A restatement of Definition \ref{defi:yr})
  If $\myv{\alpha}{\beta}$ satisfies
  $(m-1) \alpha + (n-1) \beta < 1$, then it follows that
\begin{equation}
  \tau_\nyh{\alpha}{\beta}^{ \nyh{m}{n}} \myv{i}{j}
  = \sum_{s \in [m]} \sum_{t \in [n]}
  \myChi{s + t \frac{\beta}{\alpha} \leq i + j \frac{\beta}{\alpha}},
  \quad \myv{i}{j} \in [m] \times [n]. \label{eq:taudef}
\end{equation}
  Under the assumption, the Young ranking table 
  $\tau_\nyh{\alpha}{\beta}^{ \nyh{m}{n}}$ is determined only by the
  ratio $\xi = \beta/\alpha$. Note also that
  when $\tau_{\nyh{\alpha}{\beta}}^{ \nyh{m}{n}}$ is promised to be
  a Young ranking table, it is computed as
  $\tau_{\nyh{1}{\xi}}^{\nyh{m}{n}}$
  even though $(m-1)\cdot 1 + (n-1) \xi < 1$ is
  not satisfied.
\item[(iv)](A restatement of Fact \ref{fact:yas})
  The following procedure gives an well-defined
  bijection \\ $\yas{m-1}{n-1}:\Gint{m-1}{n-1} \rightarrow \Yng{m}{n}$;
  Given $(a,b) \in \Gint{m-1}{n-1},$ take $\xi \in (a, b)$ arbitrarily,
  then let \\ $\yas{m-1}{n-1}((a,b)) := \tau_\nyh{1}{\xi}^{ \nyh{m}{n}}$.
  The bijection is referred to as \textit{yet another {\mysur}'s bijection}.
\end{itemize}
\end{fact}
For example, For $\myv{m}{n}=\myv{4}{3}$ and $\xi = \frac{5}{4}$,
a simple computation using 
\eqref{eq:taudef} yields
\begin{equation}
\tau_\nyh{1}{5/4}^{ \nyh{4}{3}} = \left|\begin{matrix}6&9&11&12\\3&5&8&10\\1&2&4&7\\\end{matrix}\right|, \label{eq:Y43ex}
\end{equation}
where $\tau_\nyh{1}{5/4}^{ \nyh{4}{3}}\myv{1}{1}$ locates the lower-left corner.
(In fact Example \ref{exam:rank} and
the example of $\theta$ in \eqref{eq:ex43} are this
$\tau_\nyh{1}{5/4}^{ \nyh{4}{3}}$.)
The generalized Farey sequence $\Gseq{3}{2}$ for $\myv{m-1}{n-1}=\myv{3}{2}$
consists of the 7 terms
\begin{equation*}
  0 < \frac{1}{2} < \frac{1}{1} < \frac{3}{2} <
  \frac{2}{1} <
  \frac{3}{1} <  \infty
\end{equation*}
and $\xi = \frac{5}{4}$ belongs to $\left(\frac{1}{1}, \frac{3}{2} \right)
\in \Gint{3}{2}.$ Thus, the Young ranking table
$\tau_\nyh{1}{5/4}^{ \nyh{4}{3}}$ of \eqref{eq:Y43ex}
is the image $\yas{3}{2}\left(\left(\frac{1}{1}, \frac{3}{2} \right)\right) $
of this open interval $\left(\frac{1}{1}, \frac{3}{2} \right)$
by yet another {\mysur}'s bijection.

A procedure to invert $\yas{m-1}{n-1}$, i.e.,
  to obtain $(a, b) = {(\yas{m-1}{n-1})}^{-1} (\tau) \in
  \Gint{m-1}{n-1}$
  for given $\tau \in \Yng{m}{n}$ is presented in
  \cite[Corollary 5.11]{ranking4}. See also \cite[Proposition 5.8]{ranking4}.

\subsection{A characterization of the Young ranking tables}
An important fact is that a ranking table
is simultaneously difference equation type and point-symmetric sum congruent
type \cite[Theorem 39 (i)(iii)]{ranking2}.

Another note to be recalled here is one made in the end of
Section \ref{sss:diffeqtypea} that if a table
is simultaneously difference equation type and
point-symmetric sum congruent type then its L-shape
determines all the terms.

Therefore, Corollary \ref{coro:Lset} implies the following
characterization of the Young ranking tables.
        \begin{coro} \label{coro:Yng}The following holds.
          \begin{itemize}
          \item[(i)] Suppose that $m \geq 2$ and $n \geq 1$ are
            integers and 
            $\theta:[m]\times [n] \rightarrow [mn]$
            satisfies that (I) it is an injective table,
            (a-h) it is of horizontal difference equation type,
            (b) it is of point-symmetric sum congruent type.
            If $\theta$ further satisfies (c) its L-shape
            $\theta|_{([m]\times \{1\}) \cup (\{1\}\times [n] )}$
            is in $\Lsetitimono{m}{n}$, then $\theta$ is
            a Young ranking table.

            Conversely, if $\theta$ is a Young ranking table,
              then $\theta$ satisfies
              (I) $\wedge$ (a-h) $\wedge$ (b) $\wedge$ (c).
          \item[(ii)] Suppose that $m \geq 1$ and $n \geq 2$ are
            integers and 
            $\theta:[m]\times [n] \rightarrow [mn]$
            satisfies that (I) it is an injective table,
            (a-v) it is of vertical difference equation type,
            (b) it is of point-symmetric sum congruent type.
            If $\theta$ further satisfies (c) its L-shape
            is in $\Lsetitimono{m}{n}$, then $\theta$ is
            a Young ranking table.

            Conversely, if $\theta$ is a Young ranking table,
              then $\theta$ satisfies
              (I) $\wedge$ (a-v) $\wedge$ (b) $\wedge$ (c).
          \end{itemize}
        \end{coro}
In a word, a Young ranking table, whose definition depends on
a real number $\xi = \frac{\beta}{\alpha}$, has a purely combinatorial
characterization that is the conjunction of
(I) being an injective table
(a-h) or (a-v) being a table of horizontal or vertical equation type
(b) being a table of point-symmetric sum congruent type
and (c) having an L-shape in $\Lsetitimono{m}{n}$.

In \cite{arXiv2024}, the set of the inverses of the
  {\mysos} permutations, whose definition involves a real parameter $\alpha$,
  is given purely combinatorial characterization by a congruential recurrence.
  Corollary \ref{coro:Yng} may be seen as an analog of the characterization
  done in \cite{arXiv2024}.
\begin{proof}  
(i) Suppose that $\theta$ satisfies (I), (a-h), (b) and (c). 
  By (I), (c) and Corollary \ref{coro:Lset},
  $\theta|_{([m]\times \{1\}) \cup (\{1\}\times [n] )} \in \LYng{m}{n}$,
    which means that there exists a Young ranking table $y \in \Yng{m}{n}$
    such that
    $\theta|_{([m]\times \{1\}) \cup (\{1\}\times [n] )} \\
    = y|_{([m]\times \{1\}) \cup (\{1\}\times [n] )}.$
      By Fact \ref{fact:detproc}, (I), (a-h) and (b), these equal L-shapes determine the same
      table $\theta = y.$
      Conversely, if $\theta$ is
        a Young ranking table, it satisfies (a-h) and (b)
        by \cite[Theorem 39(i)(iii)]{ranking2}, as noted in the
        beginning of this subsection. By Fact \ref{fact:Lincl},
        the property (c) is also satisfied. (I) is a part of the definition
      of a ranking table.
        
      (ii) Taking the transposition $\tp{\theta}$ reduces the problem to (i),
      where (a-h) is turned to (a-v).
\end{proof}

\end{document}